\documentclass[reqno,11pt,letterpaper]{amsart}
\usepackage[proof]{sdmacros}

\newcommand{\llangle}{\langle\!\langle}
\newcommand{\rrangle}{\rangle\!\rangle}
\DeclareMathOperator{\so}{\mathfrak{so}}

\title[The Ruelle zeta function at zero for nearly hyperbolic 3-manifolds]%
{The Ruelle zeta function at zero\\ for nearly hyperbolic 3-manifolds}
\author{Mihajlo Ceki\'c}
\email{mihajlo.cekic@math.uzh.ch}
\curraddr{Institut f\"ur Mathematik, Universit\"at Z\"urich, Winterthurerstrasse 190, CH-8057 Z\"urich, Switzerland}
\address{Laboratoire de Math\'ematiques d'Orsay, Universit\'e Paris-Saclay, CNRS, 91405 Orsay, France}
\author{Benjamin Delarue}
\email{bdelarue@math.upb.de}
\address{Institut f\"ur Mathematik, Universit\"at Paderborn, Paderborn, Germany; formerly known as Benjamin Küster}
\author{Semyon Dyatlov}
\email{dyatlov@math.mit.edu}
\address{Department of Mathematics, Massachusetts Institute of Technology, Cambridge, MA 02139}
\author{Gabriel P. Paternain}
\email{g.p.paternain@dpmms.cam.ac.uk}
\address{Department of Pure Mathematics and Mathematical Statistics,
University of Cambridge,
Cambridge CB3 0WB, UK}

\begin{document}

\begin{abstract}
We show that for a generic conformal metric perturbation of a compact hyperbolic 3-manifold~$\Sigma$ with Betti number $b_1$, the order of vanishing of the Ruelle zeta function at zero equals $4-b_1$, while in the hyperbolic case it is equal to $4-2b_1$. This is in contrast to
the 2-dimensional case where the order of vanishing is a topological invariant.
The proof uses the microlocal approach to dynamical zeta functions,
giving a geometric description of generalized Pollicott--Ruelle resonant differential forms at~0 in the hyperbolic case and using first variation for the perturbation.
To show that the first variation is generically nonzero we introduce a new identity
relating pushforwards of products of resonant and coresonant 2-forms on the sphere bundle $S\Sigma$ with harmonic 1-forms on~$\Sigma$.
\end{abstract}

\maketitle

\addtocounter{section}{1}
\addcontentsline{toc}{section}{1. Introduction}

Let $(\Sigma,g)$ be a compact connected oriented 3-dimensional Riemannian manifold of negative sectional curvature.
The Ruelle zeta function
\begin{equation}
\label{e:ruelle-zeta}
\zeta_{\mathrm R}(\lambda)=\prod_{\gamma}\big(1-e^{i\lambda T_{\gamma}}\big),\quad
\Im\lambda\gg 1
\end{equation}
is a converging product for $\Im\lambda$ large enough and continues meromorphically to $\lambda\in\mathbb C$
as proved by Giulietti--Liverani--Pollicott~\cite{giuletti-liverani-pollicott-13} and Dyatlov--Zworski~\cite{dyatlov-zworski-16}.
Here the product is taken over all primitive closed geodesics~$\gamma$ on~$(\Sigma,g)$ and
$T_{\gamma}$ is the length of $\gamma$.

In this paper we study the order of vanishing of $\zeta_{\mathrm R}$ at $\lambda=0$, defined
as the unique integer $m_{\mathrm R}(0)$ such that
$\lambda^{-m_{\mathrm R}(0)}\zeta_{\mathrm R}(\lambda)$ is holomorphic and nonzero at~0. Our main result is
\begin{theo}
\label{thm:main}
Let $(\Sigma,g_H)$ be a compact connected oriented hyperbolic 3-manifold and $b_1(\Sigma)$
be the first Betti number of~$\Sigma$. Then:

1. For $(\Sigma,g_H)$ we have $m_{\mathrm R}(0)=4-2b_1(\Sigma)$.

2. There exists an open and dense set $\mathscr O\subset C^\infty(\Sigma;\mathbb R)$
such that for any $\mathbf b\in\mathscr O$, there exists $\varepsilon>0$
such that for any $\tau\in (-\varepsilon,\varepsilon)\setminus \{0\}$
and $g_\tau:=e^{-2\tau\mathbf b}g_H$,
the manifold $(\Sigma,g_\tau)$ has $m_{\mathrm R}(0)=4-b_1(\Sigma)$.
\end{theo}
Part~1 of Theorem~\ref{thm:main} was proved by Fried~\cite[Theorem~3]{Fried86}
using the Selberg trace formula. The novelty is part~2, which says that
\emph{for generic conformal perturbations of the hyperbolic metric the order
of vanishing of $\zeta_{\mathrm R}$ equals $4-b_1(\Sigma)$.}
In particular, when $b_1(\Sigma)>0$ (fulfilled in many cases, in particular for mapping tori over pseudo-Anosov maps \cite[Theorem 13.4]{Farb-Margalit-12}), $m_{\mathrm R}(0)$ is not topologically invariant.
Theorem~\ref{thm:main} is the first result on instability of the order of vanishing of $\zeta_{\mathrm R}$ at~0 for Riemannian metrics. It is in contrast to the 2-dimensional case, where Dyatlov--Zworski~\cite{dyatlov-zworski-17} showed that $m_{\mathrm R}(0)=b_1(\Sigma)-2$
for \emph{any} compact connected oriented negatively curved surface $(\Sigma,g)$,
and is complementary to a recent breakthrough on the (acyclic) \emph{Fried conjecture} by Dang--Guillarmou--Rivi\`ere--Shen~\cite{dang-guillarmou-riviere-shen-20}, see~\S\ref{s:history} below.

A result similar to Theorem~\ref{thm:main} holds for contact perturbations of $S\Sigma$, see Theorem~\ref{t:pertnondeg-contact} in~\S\ref{s:general-perturb}.

\subsection{Outline of the proof}

We now outline the proof of Theorem~\ref{thm:main}. We use the microlocal approach to Pollicott--Ruelle resonances and dynamical zeta functions, which we review here~-- see~\S\ref{sec:contactflows} for details and~\S\ref{s:history} for a historical overview. Let $M=S\Sigma$ be the sphere bundle of $(\Sigma,g)$ and $X\in C^\infty(M;TM)$ be the generator of the geodesic flow. The geodesic flow is a \emph{contact flow}, i.e. there exists a 1-form $\alpha\in C^\infty(M;T^*M)$ such that $\iota_X\alpha=1$, $\iota_Xd\alpha=0$, and $\alpha\wedge d\alpha\wedge d\alpha$ is a nonvanishing volume form. When $g$ has negative curvature, the geodesic flow is \emph{Anosov}, i.e. the tangent spaces $T_\rho M$ decompose into a direct sum of the flow, unstable, and stable subspaces. Denote by $E_u^*$, $E_s^*$ the dual unstable/stable subbundles of the cotangent bundle $T^* M$, that is, $E_u^*$, $E_s^*$ are the annihilators of unstable/stable plus flow directions; these define closed conic subsets of~$T^*M$.

Define the spaces of \emph{resonant $k$-forms at~0}
\begin{equation}
  \label{e:Res-k-intro}
\Res^k_0:=\{u\in \mathcal D'(M;\Omega^k)\mid \iota_X u=0,\ \mathcal L_X u=0,\
\WF(u)\subset E_u^*\}.
\end{equation}
Here $\Omega^k$ is the (complexified) bundle of $k$-forms,
$\mathcal L_X=d\iota_X+\iota_Xd$ is the Lie derivative with respect to~$X$,
and for any distribution $u\in\mathcal D'(M;\Omega^k)$ we denote by $\WF(u)\subset T^*M\setminus 0$ the \emph{wavefront set} of $u$,
see for instance~\cite[Chapter~8]{hoermander-03}. The wavefront set condition
makes $\Res^k_0$ into a finite dimensional space, which is a consequence of the interpretation of $\Res_0^k$ as the eigenspace at~0 of the operator $P_{k,0}:=-i\mathcal L_X$ acting on certain anisotropic Sobolev spaces tailored to the flow (see \cite[Theorem~1.7]{faure-sjostrand-11} and~\cite[Lemma~2.2]{dyatlov-zworski-17}). We similarly define the spaces of \emph{generalized resonant $k$-forms at~0}
$$
\Res^{k,\ell}_0:=\{u\in\mathcal D'(M;\Omega^k)\mid\iota_X u=0,\ \mathcal L_X^\ell u=0,\
\WF(u)\subset E_u^*\},\quad
\Res^{k,\infty}_0:=\bigcup_{\ell\geq 1} \Res^{k,\ell}_0.
$$
The \emph{semisimplicity} condition for $k$-forms states that $\Res^{k,\infty}_0=\Res^k_0$, which means that the operator $P_{k,0}$ has no nontrivial Jordan blocks at~0.
We also have the dual spaces of \emph{generalized coresonant $k$-forms at~0},
replacing $E_u^*$ with $E_s^*$ in the wavefront set condition:
$$
\Res^{k,\ell}_{0*}:=\{u_*\in\mathcal D'(M;\Omega^k)\mid\iota_X u_*=0,\ \mathcal L_X^\ell u_*=0,\
\WF(u)\subset E_s^*\}.
$$
Since $E_u^*\cap E_s^*=\{0\}$, wavefront set calculus makes it possible to define
$u\wedge u_*$ as a distributional differential form as long as
$\WF(u)\subset E_u^*$, $\WF(u_*)\subset E_s^*$.

The order of vanishing of the Ruelle zeta function at~0 can be expressed as the alternating sum of the dimensions of the spaces of generalized resonant $k$-forms, see~\eqref{e:alternator}:
$$
m_{\mathrm R}(0)=\sum_{k=0}^4 (-1)^k\dim \Res^{k,\infty}_0.
$$
Thus the problem reduces to understanding the spaces $\Res^{k,\infty}_0$ for $k=0,1,2,3,4$.
The proof of Theorem~\ref{thm:main} computes their dimensions, listed in the table below, from which the formulas for $m_{\mathrm R}(0)$ follow immediately. See
Theorem~\ref{t:hyperbolic} in~\S\ref{sec:reshyp} for the hyperbolic case
and Theorem~\ref{t:general-perturber} in~\S\ref{s:general-perturb}, as well as~\S\ref{s:main-proof}, for the case of generic perturbations.
\def\mystrut{\vrule height13pt depth7pt width0pt}
\begin{center}
\begin{tabular}{|l|l|l|}
\hline
\textbf{Dimension of} & \textbf{Hyperbolic} & \textbf{Perturbation} \\
\hline\hline
\mystrut $\Res^0_0=\Res^{0,\infty}_0$ & $1$ & $1$ \\
\hline
\mystrut $\Res^1_0=\Res^{1,\infty}_0$ & $2b_1(\Sigma)$ & $b_1(\Sigma)$ \\
\hline
\mystrut $\Res^2_0$ & $b_1(\Sigma)+2$ & $b_1(\Sigma)+2$ \\
\hline
\mystrut $\Res^{2,2}_0=\Res^{2,\infty}_0$ & $2b_1(\Sigma)+2$ & $b_1(\Sigma)+2$ \\
\hline
\mystrut $\Res^3_0=\Res^{3,\infty}_0$ & $2b_1(\Sigma)$ & $b_1(\Sigma)$ \\
\hline
\mystrut $\Res^4_0=\Res^{4,\infty}_0$ & $1$ & $1$ \\
\hline
\end{tabular}
\end{center}
Note that the semisimplicity condition holds for $k=0,1,3,4$ in both
the hyperbolic case and for generic perturbations. However, semisimplicity
fails for $k=2$ in the hyperbolic case (assuming $b_1(\Sigma)>0$), and it is restored for generic perturbations. Also, since $b_2(M) = b_1(\Sigma) + 1$ (see \eqref{e:bettiGysin}), we may interpret the dimension of $\Res_0^2$ in the perturbed case as the `topological part' coming from the bijection with $H^2(M)$ and the extra invariant form $d\alpha$.

The cases $k=0,4$ of the above table are well-known: the semisimplicity condition holds
and $\Res^0_0$, $\Res^4_0$ are spanned by $1$, $d\alpha\wedge d\alpha$,
see Lemma~\ref{l:0-forms}.
One can also see that the map $u\mapsto d\alpha\wedge u$ gives an isomorphism
from $\Res^{1,\ell}_0$ to $\Res^{3,\ell}_0$.
Thus it remains to understand the spaces $\Res^{k,\infty}_0$ for $k=1,2$ and this is where the situation gets more complicated.

The spaces $\Res^k_0\cap \ker d$ of resonant states that are closed forms play a distinguished
role in our argument. Similarly to~\cite{dyatlov-zworski-17} we introduce linear
maps $\pi_k$ from $\Res^k_0\cap\ker d$ to the de Rham cohomology groups
$H^k(M;\mathbb C)$, see~\eqref{e:pi-k-def}. We show that the map $\pi_1$ is an isomorphism,
see Lemma~\ref{l:1-forms-ish}. This gives the dimension of the space of \emph{closed} forms in $\Res^1_0$:
since $b_1(M)=b_1(\Sigma)$,
$$
\dim(\Res^1_0\cap\ker d)=b_1(\Sigma).
$$
In the hyperbolic case, the other $b_1(\Sigma)$-dimensional space of \emph{non-closed} forms in $\Res^1_0$ is obtained by rotating the closed forms by $\pi/2$ in the dual unstable space,
see~\S\ref{s:hyp-1-forms}. This rotation commutes with the geodesic flow because the geodesic flow is conformal on the stable/unstable spaces, see~\eqref{e:sasaki-conformal}. This additional symmetry, which is only present in the hyperbolic case, is related to the presence
of a closed 2-form $\psi\in C^\infty(M;\Omega^2)$ which is invariant under the geodesic flow
and is not a multiple of $d\alpha$, see~\S\ref{s:hyp-psi-2}. The space $\Res^2_0$ is spanned by $d\alpha$, $\psi$, and the differentials $du$ where $u$ are the non-closed forms
in $\Res^1_0$, see~\S\ref{s:hyp-2-forms}. We also show in~\S\ref{s:hyp-2-forms}
that each $du\in d(\Res^1_0)$ lies in the range of $\mathcal L_X$,
producing $b_1(\Sigma)$ Jordan blocks for the operator $P_{2,0}$.

In the case of the perturbation $g_\tau=e^{-2\tau\mathbf b}g_H$, we use first variation techniques and make the following
\emph{nondegeneracy assumption} (see~\S\ref{s:main-proof}): for the spaces $\Res^1_0,\Res^1_{0*}$
and the contact form $\alpha$ defined using the hyperbolic metric~$g_H$,
and denoting by $\pi_\Sigma:M=S\Sigma\to\Sigma$ the projection map, we assume that
\begin{equation}
  \label{e:pairing-intro}
\begin{gathered}
(du,du_*)\mapsto \int_M (\pi_\Sigma^*\mathbf b)\alpha\wedge du\wedge du_*\quad\text{defines a nondegenerate pairing}\\
\text{on}\quad d(\Res^1_0)\times d(\Res^1_{0*}).
\end{gathered}
\end{equation}
Under the assumption~\eqref{e:pairing-intro}, we show that the non-closed 1-forms in $\Res^1_0$ move away once $\tau$ becomes nonzero (i.e. they turn into generalized resonant
states for nonzero Pollicott--Ruelle resonances), see~\S\ref{s:general-perturber}.
Thus for $0<|\tau|<\varepsilon$ all the resonant 1-forms are closed and we get $\dim\Res^1_0=b_1(\Sigma)$. Further analysis shows that semisimplicity is restored for $k=2$ and $\dim\Res^2_0=b_1(\Sigma)+2$.

It remains to show that the nondegeneracy assumption~\eqref{e:pairing-intro}
holds for a generic choice of the conformal factor~$\mathbf b\in C^\infty(\Sigma;\mathbb R)$.
The difficulty here is that $\mathbf b$ can only depend on the point in $\Sigma$ and not on elements
of~$S\Sigma$ which is where $\alpha\wedge du\wedge du_*$ lives. We reduce~\eqref{e:pairing-intro} to the following statement on nontriviality of pushforwards (see Proposition~\ref{l:nondeg-metric}):
for each real-valued resonant 1-form for the hyperbolic metric $u\in \Res^1_0$ we have
\begin{equation}
  \label{e:pairing-intro-2}
du\neq 0\quad\Longrightarrow\quad
\pi_{\Sigma*}^{}(\alpha\wedge du\wedge \mathcal J^*(du))\neq 0.
\end{equation}
Here $\mathcal J:(x,v)\mapsto (x,-v)$ is the antipodal map on $M=S\Sigma$
and $\pi_{\Sigma*}^{}$ is the pushforward of differential $k$-forms on $M$
to $(k-2)$-forms on $\Sigma$ obtained by integrating along the fibers, see~\eqref{e:forms-pf}.

The statement~\eqref{e:pairing-intro-2} concerns resonant 1-forms for the hyperbolic metric~$g=g_H$,
which are relatively well-understood. However, it is complicated by the fact
that $\pi_{\Sigma*}^{}(\alpha\wedge du\wedge \mathcal J^*(du))$ is merely a distribution,
so we cannot hope to show it is nonzero by evaluating its value at some point.
Instead we pair it with functions in $C^\infty(\Sigma)$ which have to be chosen carefully so that we can compute the pairing. More precisely, we prove the following identity
(Theorem~\ref{t:main-identity} in~\S\ref{s:main-identity}):
\begin{equation}
  \label{e:pairing-intro-3}
Q_4F=-\textstyle{1\over 6}\Delta_g|\sigma|_g^2\quad\text{where}\quad
\pi_{\Sigma*}^{}(\alpha\wedge du\wedge\mathcal J^*(du))=F\,d\vol_g.
\end{equation}
Here $d\vol_g$ is the volume form on $(\Sigma,g)$,
$\Delta_g$ is the Laplace--Beltrami operator,
$Q_4:\mathcal D'(\Sigma)\to C^\infty(\Sigma)$ is a naturally defined
smoothing operator, and
$$
\sigma:=\pi_{\Sigma*}^{}(d\alpha\wedge u)\in C^\infty(\Sigma;T^*\Sigma)
$$
is proved to be a nonzero harmonic 1-form on $(\Sigma,g_H)$. The identity~\eqref{e:pairing-intro-3} implies the nontriviality statement~\eqref{e:pairing-intro-2}:
if $F=0$ then $|\sigma|_g^2$ is constant, but hyperbolic 3-manifolds
do not admit harmonic 1-forms of nonzero constant length as shown in Appendix~\ref{sec:appC}. This finishes the proof of Theorem~\ref{thm:main}.

If one is interested instead in conformal perturbations of the contact form $\alpha$, then
one needs to show that $\alpha\wedge du\wedge du_*$ is not identically~0
assuming that $u\in\Res^1_0$, $u_*\in\Res^1_{0*}$ and $du\neq 0$, $du_*\neq 0$.
The latter follows
from the full support property for Pollicott--Ruelle resonant states proved by Weich~\cite{weich-17}. See Theorem~\ref{t:pertnondeg-contact} in~\S\ref{s:general-perturb}
for details.

We finally note that it would have been possible to introduce a flat unitary twist in our discussion.
Namely, we can consider a Hermitian vector bundle over $\Sigma$ endowed with a unitary flat connection $A$.
Resonant spaces can be defined using the operator $d_{A}$ and the holonomy of $A$ provides a way to twist the Ruelle zeta function as well, we refer to~\cite{cekic-paternain-19} for details. We do not pursue this extension here in order to simplify the presentation.

\subsection{A conjecture}

Theorem~\ref{thm:main} can be interpreted as follows: the hyperbolic metric
has non-closed resonant states due to the extra symmetries, and by destroying these
symmetries we make all resonant states closed. We thus make the following
conjecture about generic contact Anosov flows:
\begin{conj}\label{conj:1}
Let $M$ be a compact $2n+1$ dimensional manifold and $\alpha$
a contact 1-form on $M$ such that the corresponding flow is Anosov
with orientable stable/unstable bundles.
Define the spaces $\Res^k_0$, $0\leq k\leq 2n$, by~\eqref{e:Res-k-intro}
and let $\pi_k:\Res^k_0\cap\ker d\to H^k(M;\mathbb C)$ be defined by~\eqref{e:pi-k-def}.
Then for a generic choice of $\alpha$ we have:
\begin{enumerate}
\item the semisimplicity condition holds in all degrees $k=0,\dots,2n$;
\item $d(\Res^k_0)=0$ for all $k=0,\dots,2n$;
\item for $k=0,\dots,n$ the map $\pi_k$ is onto, $\ker\pi_k=d\alpha\wedge \Res^{k-2}_0$,
and $\dim\ker\pi_k=\dim\Res^{k-2}_0$.
\end{enumerate}
Denoting by $b_k(M)$ the $k$-th Betti number of~$M$, we then have
\begin{equation}
  \label{e:Res-k-conj}
\dim\Res^k_0=\sum_{j=0}^{\lfloor k/2\rfloor} b_{k-2j}(M),\quad
0\leq k\leq n;\quad \dim\Res^{2n-k}_0=\dim\Res^k_0
\end{equation}
and the order of vanishing of the Ruelle zeta function at~0 is given by
(see~\cite[(2.5)]{dyatlov-zworski-16})
\begin{equation}
  \label{e:conj}
m_{\mathrm R}(0)=\sum_{k=0}^{2n}(-1)^{k+n}\dim\Res^k_0=\sum_{k=0}^n(-1)^{k+n}(n+1-k)b_k(M).
\end{equation}
\end{conj}
The proof of part~2 of Theorem~\ref{thm:main} (see Theorem~\ref{t:general-perturber} in~\S\ref{s:general-perturb}, as well as~\S\ref{s:main-proof}) shows that Conjecture~\ref{conj:1} holds for $n=2$
and geodesic flows of generic nearly hyperbolic metrics
(while the conjecture is stated for generic metrics that do not have to be nearly hyperbolic).
Moreover, \cite{dyatlov-zworski-17}
shows that Conjecture~\ref{conj:1} holds for $n=1$ and any contact Anosov flow.

Note that the conditions (1) and (2) of Conjecture~\ref{conj:1} imply (3). Indeed,
by the work of Dang--Rivi\`ere~\cite[Theorem~2.1]{Dang-Riviere-Topology} the cohomology of the complex
$(\Res^{k,\infty},d)$, with $\Res^{k,\infty}$ defined in~\eqref{e:res-k-infty} below
with $\lambda_0:=0$, is isomorphic to the de Rham cohomology of~$M$
(with the isomorphism mapping each closed form in $\Res^{k,\infty}$
to its cohomology class). By~\eqref{e:res-0-split}
and the semisimplicity condition~(1), we have
$\Res^{k,\infty}=\Res^k_0\oplus (\alpha\wedge \Res^{k-1}_0)$.
By condition~(2), we have $d(u+\alpha\wedge v)=d\alpha\wedge v$
for all $u\in\Res^k_0$, $v\in\Res^{k-1}_0$. If $k\leq n$,
then $d\alpha\wedge:\Res^{k-1}_0\to \Res^{k+1}_0$ is injective,
so $\Res^{k,\infty}\cap\ker d=\Res^k_0$ and $d(\Res^{k-1,\infty})=d\alpha\wedge\Res^{k-2}_0$.
This gives condition~(3).

Note also that for $n=2$ the set of contact forms satisfying Conjecture~\ref{conj:1}
is open in $C^\infty(M;TM)$.
Indeed, by the perturbation theory discussed in~\S\ref{s:general-perturber},
more specifically~\eqref{e:perprez-2}, if we take a sufficiently small perturbation
of a contact form satisfying Conjecture~\ref{conj:1}, then
$\dim\Res^{1,\infty}_0\leq b_1(M)$ and $\dim\Res^{2,\infty}_0\leq b_2(M)+1$.
By Lemma~\ref{l:1-forms-ish} we see that semisimplicity holds for $k=1$
and $d(\Res^1_0)=0$. Then Lemma~\ref{l:1-means-2} together with Lemma~\ref{l:0-forms}
give all the conclusions of Conjecture~\ref{conj:1}.
A similar argument might work in the case of higher $n$.
Thus the main task in proving the conjecture is to show that (1) and (2) hold on a dense set
of contact forms.


One can make a similar conjecture for geodesic flows of generic negatively curved
compact orientable $n+1$-dimensional Riemannian manifolds~$(\Sigma,g)$, with $M=S\Sigma$.
In particular, if $n=2m$ is even, then $\Sigma$ is odd dimensional and thus has Euler characteristic~0.
By the Gysin exact sequence we have $b_k(M)=b_k(\Sigma)$
for $0\leq k<n$ and $b_n(M)=b_n(\Sigma)+b_0(\Sigma)$. Moreover,
by Poincar\'e duality we have $b_k(\Sigma)=b_{n+1-k}(\Sigma)$. Thus~\eqref{e:conj} becomes
$$
m_{\mathrm R}(0)=b_0(\Sigma)+\sum_{k=0}^m (-1)^k(2m+1-2k)b_k(\Sigma).
$$
This is in contrast to the hyperbolic case, where by~\cite[Theorem~3]{Fried86}
$$
m_{\mathrm R}(0)=\sum_{k=0}^m (-1)^k(2m+2-2k)b_k(\Sigma).
$$
Note that we only expect Conjecture~\ref{conj:1} to hold for generic flows/metrics rather than, say,
all non-hyperbolic metrics: for $n=2$ the proof of Theorem~\ref{thm:main} uses first variation
which by the Implicit Function Theorem suggests that there is a `singular submanifold' of metrics passing
through the hyperbolic metric on which Conjecture~\ref{conj:1} fails.

\subsection{Previous work} 
\label{s:history}

The treatment of Pollicott--Ruelle resonances of an Anosov flow as eigenvalues of the generator of the flow
on anisotropic Banach and Hilbert spaces has been developed by many authors, including Baladi~\cite{baladi-05}, Baladi--Tsujii~\cite{baladi-tsuji-07},
Blank--Keller--Liverani~\cite{BKL_02},
Butterley--Liverani~\cite{butterley-car-07}, Gou\"ezel--Liverani~\cite{gou-car-06},
and Liverani~\cite{Liverani1,Liverani2}
(some of the above papers considered the related setting of Anosov maps).
In this paper we use the microlocal approach to dynamical resonances,
introduced by Faure--Sj\"ostrand~\cite{faure-sjostrand-11} and developed further by Dyatlov--Zworski~\cite{dyatlov-zworski-16};
see also Faure--Roy--Sj\"ostrand~\cite{faure-roy-sjostrand}, Dyatlov--Guillarmou~\cite{dyatlov-guillarmou}, as well as Dang--Rivi\`ere~\cite{Dang-Riviere-19} and Meddane~\cite{Meddane-21} for the treatment of Morse--Smale and Axiom A flows.

The study of the relation of the vanishing order $m_{\mathrm R}(0)$
to the topology of the underlying manifold~$M$ has a long history,
going back to the works of Fried~\cite{Fried0,Fried86} for geodesic flows
on hyperbolic manifolds. The paper~\cite{Fried86}
also related the leading coefficient of $\zeta_{\mathrm R}$ at~0
to \emph{Reidemeister torsion}, which is a topological invariant
of~$M$. It considered the more general setting of a twisted zeta function
corresponding to a unitary representation. One advantage of such twists
is that one can choose the representation so that the twisted de Rham complex
is \emph{acyclic}, i.e. has no cohomology, and then one expects $\zeta_{\mathrm R}$ to be holomorphic and nonvanishing at~0.

In~\cite[p.~66]{Fried87} Fried conjectured a formula relating the Reidemeister torsion
with the value $\zeta_{\mathrm R}(0)$ for geodesic flows on all compact locally homogeneous manifolds with acyclic representations. Fried's conjecture was proved by Shen~\cite{Shen-Fried} for compact locally symmetric reductive manifolds, following earlier contributions
by Bismut~\cite{Bismut} and Moscovici--Stanton~\cite{Moscovici-Stanton}.
The abovementioned works~\cite{Fried0,Fried86,Bismut,Moscovici-Stanton,Shen-Fried}
used representation theory and Selberg trace formulas, which do not extend beyond
the class of locally symmetric manifolds.

In recent years much progress has been made on understanding
the relation between the behavior of $\zeta_{\mathrm R}$
at~0, as well as the dimensions of $\Res^{k,\ell}_0$, with topological invariants for general (not locally symmetric) negatively curved Riemannian
manifolds and Anosov flows:
\begin{itemize}
\item Dyatlov--Zworski~\cite{dyatlov-zworski-17} computed $m_{\mathrm R}(0)$
for any contact Anosov flow in dimension~3 with orientable stable/unstable bundles,
including geodesic flows on compact oriented negatively curved surfaces;
\item Dang--Rivi\`ere~\cite[Theorem 2.1]{Dang-Riviere-Topology} showed that the chain complex $(\Res^{\bullet, \infty}, d)$, where $\Res^{k,\infty}=\Res^{k,\infty}(0)$ is defined in~\eqref{e:res-k-l} below, is homotopy equivalent to the usual de Rham complex and hence their cohomologies agree.
One can see that Conjecture~\ref{conj:1} is compatible with this result, using~\eqref{e:res-0-split} and the fact that $(d\alpha \wedge)^k: \Omega_0^{n-k} \to\Omega_0^{n + k}$ is a bundle isomorphism for $0\leq k\leq n$;
\item Hadfield~\cite{Hadfield} showed a result similar to~\cite{dyatlov-zworski-17}
for geodesic flows on negatively curved surfaces with boundary;
\item Dang--Guillarmou--Rivi\`ere--Shen~\cite{dang-guillarmou-riviere-shen-20}
computed $\dim\Res^{k,\infty}_0$ for hyperbolic 3-manifolds
and proved Fried's formula relating $\zeta_{\mathrm R}(0)$ to Reidemeister torsion
for nearly hyperbolic 3-manifolds in the acyclic case; see also Chaubet--Dang~\cite{chaubet-dang-19};
\item K\"uster--Weich~\cite{kuster-weich} obtained several results on geodesic
flows on compact hyperbolic manifolds and their perturbations, in particular showing that
$\dim\Res^1_0=b_1(\Sigma)$ when $\dim\Sigma\neq 3$;
\item Ceki\'c--Paternain~\cite{cekic-paternain-19} studied
volume preserving Anosov flows in dimension 3, giving the first example
of a situation where $m_{\mathrm R}(0)$ jumps under perturbations of the flow and
thus is not topologically invariant;
\item Borns-Weil--Shen~\cite{Borns-Weil-Shen-21} proved a result similar to~\cite{dyatlov-zworski-17}
for nonorientable stable/unstable bundles.
\end{itemize}
Our Theorem~\ref{thm:main} gives a jump in $m_{\mathrm R}(0)$ for
geodesic flows on 3-manifolds and indicates that the situation for the hyperbolic case
is different from that in the case of generic metrics.
We stress that it is more difficult to obtain results for generic metric perturbations
(such as Theorem~\ref{thm:main}) than for generic perturbations of contact forms (such as Theorem~\ref{t:pertnondeg-contact}
in~\S\ref{s:general-perturb}) due to the more restricted nature of metric perturbations.

One of our main technical results (Theorem~\ref{t:main-identity}) bears (limited) similarities to known pairing formulas for Patterson--Sullivan distributions such as those established by
Anantharaman--Zelditch~\cite{Anantharaman-Zelditch-07},
Hansen--Hilgert--Schr\"oder~\cite{Hansen-Hilgert-Schroder},
Dyatlov--Faure--Guillarmou~\cite{dyatlov-faure-guillarmou-15},
and Guillarmou--Hilgert--Weich~\cite{guillarmou-hilgert-weich-18}.
We briefly discuss this in the Remark after Theorem~\ref{t:main-identity}.

\subsection{Structure  of the paper}

\begin{itemize}
\item \S\ref{sec:contactflows} discusses contact Anosov flows on 5-manifolds
and sets up the scene for the rest of the paper.
In particular, it introduces Pollicott--Ruelle resonances, (co-)resonant states, dynamical
zeta functions, de Rham cohomology, and geodesic flows.
It also proves various general lemmas about the maps $\pi_k$ and semisimplicity.
\item \S\ref{sec:reshyp} gives a complete description of generalized resonant states at~0 for hyperbolic 3-manifolds, proving part~1 of Theorem~\ref{thm:main}.
The approach in this section is geometric, as opposed to the algebraic
route taken in~\cite{Fried86} and~\cite{dang-guillarmou-riviere-shen-20}.
\item \S\ref{s:general-perturb} discusses contact perturbations
of geodesic flows on hyperbolic 3-manifolds. It proves Theorem~\ref{t:general-perturber}
which is a general perturbation statement using the nondegeneracy
condition~\eqref{e:pairing-intro}, as well as Theorem~\ref{t:pertnondeg-contact}
on generic contact perturbations. It also gives the proof of part~2 of Theorem~\ref{thm:main},
relying on the key identity~\eqref{e:pairing-intro-3}.
\item \S\ref{s:main-identity} contains the proof of the identity~\eqref{e:pairing-intro-3}
(stated in Theorem~\ref{t:main-identity}),
using a change of variables, a regularization procedure, and the results
of~\S\ref{sec:reshyp}.
\item Finally, Appendix~\ref{sec:appC} gives a proof of the fact
that hyperbolic 3-manifolds have no nonzero harmonic 1-forms of constant length.
\end{itemize}

\section{Contact 5-dimensional flows}\label{sec:contactflows}

In this section we study general contact Anosov flows on 5-dimensional manifolds. Some of the statements below apply to non-contact Anosov flows and to other dimensions, however we use the setting of 5-dimensional contact flows for uniformity of presentation.

\subsection{Contact Anosov flows}
\label{s:contact-anosov}

Assume that $M$ is a compact connected 5-dimensional $C^\infty$ manifold 
and $\alpha\in C^\infty(M;T^*M)$ is a contact 1-form on~$M$, namely
$$
d\vol_\alpha:=\alpha\wedge d\alpha\wedge d\alpha\neq 0\quad\text{everywhere}.
$$ 
We fix the orientation on $M$ by requiring that $d\vol_\alpha$ be positively oriented.
Let $X\in C^\infty(M;TM)$ be the associated Reeb field, that is the unique vector field satisfying
\begin{equation}
  \label{e:reeb-field}
\iota_X \alpha=1,\quad
\iota_X d\alpha=0.
\end{equation}
Note that this immediately implies (where $\mathcal L_X$ denotes the Lie derivative)
$$
\mathcal L_X\alpha=d\iota_X\alpha+\iota_Xd\alpha=0.
$$
We assume that the flow generated by $X$,
$$
\varphi_t:=e^{tX}:M\to M,
$$
is an \emph{Anosov flow}, namely there exists a continuous flow/unstable/stable decomposition of the tangent spaces to $M$,
\begin{equation}
  \label{e:Anosov-split}
T_\rho M=E_0(\rho)\oplus E_u(\rho)\oplus E_s(\rho),\quad
\rho\in M,\quad
E_0(\rho):=\mathbb R X(\rho)
\end{equation}
and there exist constants $C,\theta>0$ and a smooth norm $|\bullet|$ on the fibers of $TM$ such that for
all $\rho\in M$, $\xi\in T_\rho M$, and $t$
\begin{equation}
  \label{e:Anosov}
|d\varphi_t(\rho)\xi|\leq Ce^{-\theta|t|}\cdot |\xi|\quad\text{if}\quad\begin{cases}
t\leq 0,& \xi\in E_u(\rho)\quad\text{or}\\
t\geq 0,& \xi\in E_s(\rho).\end{cases}
\end{equation}
The flow/unstable/stable decomposition gives rise to the dual decomposition of the cotangent spaces to $M$,
\begin{equation}
  \label{e:anosov-dual}
\begin{aligned}
T_\rho^* M=E_0^*(\rho)\oplus E_u^*(\rho)\oplus E_s^*(\rho),&\quad
E_0^*:=(E_u\oplus E_s)^\perp,\\
E_u^*:=(E_0\oplus E_u)^\perp,&\quad
E_s^*:=(E_0\oplus E_s)^\perp.
\end{aligned}
\end{equation}
Since $\mathcal L_X\alpha=0$, we see from~\eqref{e:Anosov} that $\alpha|_{E_u\oplus E_s}=0$ and thus
$$
E_0^*=\mathbb R\alpha.
$$
Since $\alpha$ is a contact form and $d\alpha$ vanishes on $E_u\times E_u$ and on $E_s\times E_s$
(as follows from~\eqref{e:Anosov} and the fact that $\mathcal L_Xd\alpha=0$),
we have $\dim E_u=\dim E_s=2$.

\subsubsection{Bundles of differential forms}
We define the vector bundles over $M$
\begin{equation}
  \label{e:Omega-k-def}
\Omega^k:=\wedge^k (T^*M),\quad
\Omega^k_0:=\{\omega\in\Omega^k\mid \iota_X\omega=0\}\simeq \wedge^k(E_u^*\oplus E_s^*).
\end{equation}
Note that smooth sections of $\Omega^k$ are differential $k$-forms on~$M$.

We use the de Rham cohomology groups
\begin{equation}
  \label{e:de-rham}
H^k(M;\mathbb C):={\{u\in C^\infty(M;\Omega^k)\mid du=0\}\over \{dv\mid v\in C^\infty(M;\Omega^{k-1})\}}.
\end{equation}
Unless otherwise stated, we will always take $\Omega^k$ to be complexified. We define the Betti numbers
$$
b_k(M):=\dim H^k(M;\mathbb C).
$$
Since $M$ is connected and by Poincar\'e duality we have
$$
b_0(M)=1,\quad
b_k(M)=b_{5-k}(M).
$$
The bundles $\Omega^k$ and $\Omega^k_0$ are related as follows:
$$
\Omega^k\simeq \Omega^k_0\oplus \Omega^{k-1}_0
$$
with the canonical isomorphism and its inverse given by
\begin{equation}
  \label{e:0-split}
u\mapsto (u-\alpha\wedge\iota_X u,\iota_X u),\quad
(v,w)\mapsto v+\alpha\wedge w.
\end{equation}
Denote by $d\alpha\wedge$ the map $u\mapsto d\alpha\wedge u$ and by $d\alpha\wedge^2$ the map
$u\mapsto d\alpha\wedge d\alpha\wedge u$, then we have linear isomorphisms (as both maps are injective and image and domain have the same dimension)
\begin{equation}
  \label{e:d-alpha-wedge}
d\alpha\wedge:\Omega^1_0\to \Omega^3_0,\quad
d\alpha\wedge^2:\Omega^0_0\to \Omega^4_0.
\end{equation}
We also have a nondegenerate bilinear pairing between sections of $\Omega^k_0$ and $\Omega^{4-k}_0$ given by
\begin{equation}\label{e:main-pairing}
u\in C^\infty(M;\Omega^k_0),\
u_*\in C^\infty(M;\Omega^{4-k}_0)\ \mapsto\
\llangle u,u_*\rrangle:=\int_M \alpha\wedge u\wedge u_*
\end{equation}
which in particular identifies the dual space to $L^2(M;\Omega^k_0)$ with $L^2(M;\Omega^{4-k}_0)$.
If $A:C^\infty(M;\Omega^k_0)\to \mathcal D'(M;\Omega^k_0)$ is a continuous operator,
where $\mathcal D'$ denotes the space of distributions, then its \emph{transpose operator} is
the unique operator $A^T:C^\infty(M;\Omega^{4-k}_0)\to \mathcal D'(M;\Omega^{4-k}_0)$ satisfying
\begin{equation}
\label{e:transpose}
\llangle Au,u_*\rrangle=\llangle u,A^Tu_*\rrangle\quad\text{for all}\quad
u\in C^\infty(M;\Omega^k_0),\ u_*\in C^\infty(M;\Omega^{4-k}_0).
\end{equation}

\subsection{Geodesic flows}
\label{s:geodesic-flows}
A large class of examples of contact Anosov flows is given by geodesic flows on negatively curved manifolds,
which is the setting of the main results of this paper. More precisely, assume that $(\Sigma,g)$ is a compact connected
oriented 3-dimensional Riemannian manifold. Define $M$
to be the sphere bundle of $\Sigma$ and let $\pi_\Sigma$ be the canonical projection:
$$
M:=S\Sigma=\{(x,v)\in T\Sigma\colon |v|_g=1\},\quad
\pi_\Sigma:M\to\Sigma.
$$
Define the \emph{canonical}, or \emph{tautological}, 1-form $\alpha$ on $M$ as follows: for all $\xi\in T_{(x,v)}M$,
\begin{equation}
  \label{e:geodesic-alpha}
\langle \alpha(x,v),\xi\rangle=\langle v,d\pi_\Sigma(x,v)\xi\rangle_{g}.
\end{equation}
Then $\alpha$ is a contact form, the corresponding flow $\varphi_t$ is the geodesic flow, and
$d\vol_\alpha$ is the standard Liouville volume form up to a constant, see for instance~\cite[\S1.3.3]{paternain-99}.
If the metric $g$ has negative sectional curvature, then the flow $\varphi_t$ is Anosov, see for instance~\cite[Theorem~3.9.1]{klingenberg-95}.

We have the time reversal involution
\begin{equation}
  \label{e:J-def}
\mathcal J:M\to M,\quad \mathcal J(x,v)=(x,-v)
\end{equation}
which is an orientation reversing diffeomorphism satisfying
\begin{equation}\label{eq:J-prop}
\mathcal J^*\alpha=-\alpha,\quad
\mathcal J^*X=-X,\quad
\varphi_t\circ \mathcal J=\mathcal J\circ\varphi_{-t}
\end{equation}
and the differential of $\mathcal J$ maps $E_0,E_u,E_s$ into $E_0,E_s,E_u$.

\subsubsection{Horizontal and vertical spaces}
\label{s:hv-spaces}

Recall from~\eqref{e:Anosov-split} that
an Anosov flow induces a splitting of the tangent bundle $TM$ into the flow, unstable, and stable subbundles.
For geodesic flows there is another splitting, into \emph{horizontal and vertical subbundles},
which we briefly review here. See~\cite[\S1.3.1]{paternain-99} for more details.

Let $(x,v)\in M=S\Sigma$. The vertical space at $(x,v)$ is the tangent space to the fiber~$S_x\Sigma$:
$$
\mathbf V(x,v):=\ker d\pi_\Sigma(x,v)\ \subset\ T_{(x,v)}M.
$$
To define a complementary horizontal subspace of $T_{(x,v)}M$, we use the metric. The
\emph{connection map} of the metric is the unique bundle homomorphism $\mathcal K:TM\to T\Sigma$
covering the map~$\pi_\Sigma$ such that for any curve on $M$ written as
$$
\rho(t)=(x(t),v(t)),\quad
x(t)\in \Sigma,\quad
v(t)\in S_{x(t)}\Sigma
$$
we have
\begin{equation}\label{e:connection-map}
	\mathcal K(\rho(t))\dot\rho(t)=\mathbf D_t v(t)\ \in\ T_{x(t)}\Sigma
\end{equation}
where $\mathbf D_tv(t)$ denotes the Levi--Civita covariant derivative of the vector
field $v(t)$ along the curve $x(t)$ (see e.g.\ \cite[Proposition 2.2]{do-Carmo-92} for a precise definition). Note that since $d_t\langle v(t),v(t)\rangle_g=0$, the range
of $\mathcal K(x,v)$ is $g$-orthogonal to $v$.

We now define the horizontal space as
$$
\mathbf H(x,v):=\ker \mathcal K(x,v)\ \subset\ T_{(x,v)}M.
$$
We have the splitting
$$
T_{(x,v)}M=\mathbf H(x,v)\oplus \mathbf V(x,v),\quad
\dim\mathbf H(x,v)=3,\quad
\dim\mathbf V(x,v)=2
$$
and the isomorphisms (here $\{v\}^\perp$ is the $g$-orthogonal complement of $v$ in $T_x\Sigma$)
$$
d\pi_\Sigma(x,v):\mathbf H(x,v)\to T_x\Sigma,\quad
\mathcal K(x,v):\mathbf V(x,v)\to \{v\}^\perp
$$
which together give the following isomorphism $T_{(x,v)}M\to T_x\Sigma\oplus \{v\}^\perp$:
\begin{equation}
  \label{e:hv-map}
\xi\mapsto (\xi_H,\xi_V),\quad
\xi_H=d\pi_\Sigma(x,v)\xi,\quad
\xi_V=\mathcal K(x,v)\xi.
\end{equation}
We use the map~\eqref{e:hv-map} to identify $T_{(x,v)}M$ with $T_x\Sigma\oplus\{v\}^\perp$.

Under the identification~\eqref{e:hv-map}, the contact form $\alpha$ and its differential satisfy
(see~\cite[Proposition~1.24]{paternain-99})
\begin{equation}
  \label{e:alpha-hv}
\begin{aligned}
\alpha(x,v)(\xi)&=\langle\xi_H,v\rangle_g,\\
d\alpha(x,v)(\xi,\eta)&=\langle \xi_V,\eta_H\rangle_g-\langle\xi_H,\eta_V\rangle_g.
\end{aligned}
\end{equation}
Using the splitting~\eqref{e:hv-map}, we define the \emph{Sasaki metric} $\langle \bullet,\bullet\rangle_S$ on $M$ as follows:
\begin{equation}\label{e:sasaki-metric}
\langle \xi,\eta\rangle_S:=\langle\xi_H,\eta_H\rangle_g+\langle\xi_V,\eta_V\rangle_g.
\end{equation}
We finally remark that the generator $X$ of the geodesic flow has the following form under the isomorphism~\eqref{e:hv-map}:
\begin{equation}
  \label{e:X-hv}
X(x,v)_H=v,\quad X(x,v)_V=0.
\end{equation}

\subsubsection{De Rham cohomology of the sphere bundle}
\label{s:de-rham-sphere}

We now describe the de Rham cohomology of $M=S\Sigma$ in terms of the cohomology of $\Sigma$.
To relate the two, we use the pullback operators
$$
\pi_\Sigma^*:C^\infty(\Sigma;\Omega^k)\to C^\infty(M;\Omega^k),\quad
0\leq k\leq 3
$$
and the pushforward operators defined by integrating along the fibers of $S\Sigma$
\begin{equation}
  \label{e:forms-pf}
\pi_{\Sigma*}^{}:C^\infty(M;\Omega^k)\to C^\infty(\Sigma;\Omega^{k-2}),\quad
2\leq k\leq 5.
\end{equation}
Here the orientation on each fiber $S_x\Sigma$ is induced by the orientation on~$\Sigma$:
if $v,v_1,v_2$ is a positively oriented orthonormal basis of $T_x\Sigma$,
then the vertical vectors corresponding to $v_1,v_2$ form a positively oriented basis of $T_v(S_x\Sigma)$.
The pushforward operation can be characterized as follows: if $X_1,\dots,X_{k-2}$
are vector fields on $\Sigma$ and $\widetilde X_1,\dots,\widetilde X_{k-2}$ are vector
fields on $M$ projecting to $X_1,\dots,X_{k-2}$ under $d\pi_\Sigma$, then
for any $\omega\in C^\infty(M;\Omega^k)$ and $x\in\Sigma$
$$
\pi_{\Sigma*}^{}\omega(x)(X_1,\dots,X_{k-2})=\int_{S_x\Sigma} \iota_{\widetilde X_{k-2}}\dots\iota_{\widetilde X_1}\omega.
$$
Another characterization of $\pi_{\Sigma*}^{}$ is that for any $\omega\in C^\infty(M;\Omega^k)$
and any compact $k-2$ dimensional oriented submanifold with boundary $Y\subset\Sigma$, we have
\begin{equation}
  \label{e:pushforward-integral}
\int_{\pi_\Sigma^{-1}(Y)}\omega=\int_Y\pi_{\Sigma*}^{}\omega.
\end{equation}
Here the orientation on $\pi_\Sigma^{-1}(Y)$ is induced by the orientation on~$Y$.
If $Y=\Sigma$ is the entire base manifold, then the orientation on $\pi_\Sigma^{-1}(\Sigma)=S\Sigma$ featured in~\eqref{e:pushforward-integral}
is \emph{opposite} to the usual orientation on $M=S\Sigma$, induced by $d\vol_\alpha=\alpha\wedge d\alpha\wedge d\alpha$.
In fact, using~\eqref{e:alpha-hv} we can compute that
\begin{equation}
  \label{e:pushforward-volume}
\pi_{\Sigma*}^{}d\vol_\alpha=-8\pi d\vol_g
\end{equation}
where $d\vol_g$ is the volume form on~$\Sigma$ induced by~$g$ and the choice of orientation,
by applying $d\vol_\alpha$ to the vectors $X=(v,0)$, $(v_1,0)$, $(v_2,0)$, $(0,v_1)$, $(0,v_2)$
written using the horizontal/vertical decomposition~\eqref{e:hv-map}, where
$v,v_1,v_2$ is a positively oriented $g$-orthonormal basis on~$\Sigma$.

The pushforward map has the following properties (see for instance~\cite[Propositions~6.14.1 and~6.15]{bott-tu-82} for the related case
of vector bundles):
\begin{align}
  \label{e:pfm-1}
d\pi_{\Sigma*}^{}&=\pi_{\Sigma*}^{}d,\\
  \label{e:pfm-2}
\pi_{\Sigma*}^{}\big(\omega_1\wedge (\pi_\Sigma^*\omega_2)\big)&=(\pi_{\Sigma*}^{}\omega_1)\wedge\omega_2.
\end{align}
Note that the maps $\pi_{\Sigma*}^{}$, $\pi_\Sigma^*$ can also be defined on distributional
forms. For $\pi_{\Sigma*}^{}$ this follows from the fact that pushforward is always well-defined
on distributions as long as the fibers are compact and for the pullback~$\pi_\Sigma^*$ this follows from the fact that
$\pi_\Sigma$ is a submersion~\cite[Theorem~6.1.2]{hoermander-03}.

Since the map $\mathcal J$ defined in~\eqref{e:J-def} is an orientation reversing
diffeomorphism of the fibers of $S\Sigma$, we also have
\begin{equation}
  \label{e:J-push-forward}
\pi_{\Sigma*}^{}(\mathcal J^*\omega)=-\pi_{\Sigma*}^{}\omega.
\end{equation}

Since pullbacks commute with the differential $d$, and by \eqref{e:pfm-1}, the operations $\pi_\Sigma^*,\pi_{\Sigma*}^{}$ induce maps on de Rham cohomology, which we denote
by the same letters:
$$
\pi_\Sigma^*:H^k(\Sigma;\mathbb C)\to H^k(M;\mathbb C),\quad
\pi_{\Sigma*}^{}:H^k(M;\mathbb C)\to H^{k-2}(\Sigma;\mathbb C).
$$
From the Gysin exact sequence (see for instance~\cite[Proposition~14.33]{bott-tu-82}, where
the Euler class is zero since $\Sigma$ is three-dimensional; alternatively one
can use K\"unneth formulas and the fact that every compact orientable 3-manifold is parallelizable)
we have isomorphisms
\begin{equation}
  \label{e:Gysin-iso}
\pi_\Sigma^*:H^1(\Sigma;\mathbb C)\to H^1(M;\mathbb C),\quad
\pi_{\Sigma*}^{}:H^4(M;\mathbb C)\to H^2(\Sigma;\mathbb C)
\end{equation}
and the exact sequences
\begin{gather}
  \label{e:exaseq-1}
0\to H^2(\Sigma;\mathbb C)\xrightarrow{\pi_\Sigma^*} H^2(M;\mathbb C)\xrightarrow{\pi_{\Sigma*}^{}} H^0(\Sigma;\mathbb C)\to 0,\\
  \label{e:exaseq-2}
0\to H^3(\Sigma;\mathbb C)\xrightarrow{\pi_\Sigma^*} H^3(M;\mathbb C)\xrightarrow{\pi_{\Sigma*}^{}} H^1(\Sigma;\mathbb C)\to 0.
\end{gather}
In particular, we get formulas for the Betti numbers of the sphere bundle~$M$:
\begin{equation}
  \label{e:bettiGysin}
b_0(M)=b_5(M)=1,\quad
b_1(M)=b_4(M)=b_1(\Sigma),\quad
b_2(M)=b_3(M)=b_1(\Sigma)+1.
\end{equation}

\subsection{Pollicott--Ruelle resonances}
\label{section:PRresonances}

We now review the theory of Pollicott--Ruelle resonances in the present setting.
Define the first order differential operators
$$
\begin{aligned}
P_k&\,:=-i\mathcal L_X:C^\infty(M;\Omega^k)\to C^\infty(M;\Omega^k),\\
P_{k,0}&\,:=-i\mathcal L_X:C^\infty(M;\Omega^k_0)\to C^\infty(M;\Omega^k_0).
\end{aligned}
$$
Note that $P_{k,0}$ is the restriction of $P_k$ to $C^\infty(M;\Omega^k_0)$ which is 
the space of all $u\in C^\infty(M;\Omega^k)$ which satisfy $\iota_Xu=0$.

For $\lambda\in\mathbb C$ with $\Im\lambda$ large enough, the integral
\begin{equation}
  \label{e:Rk-uhp}
R_k(\lambda):=i\int_0^\infty e^{i\lambda t}e^{-itP_k}\,dt:L^2(M;\Omega^k)\to L^2(M;\Omega^k)
\end{equation}
converges and defines a bounded operator on $L^2$ which is holomorphic in $\lambda$.
Here the evolution group $e^{-itP_k}$ is given by
$e^{-itP_k}u=\varphi_{-t}^*u$.
It is straightforward to check that $R_k(\lambda)$ is the $L^2$-resolvent of $P_k$:
\begin{equation}
  \label{e:resolver}
R_k(\lambda)=(P_k-\lambda)^{-1}:L^2(M;\Omega^k)\to L^2(M;\Omega^k),\quad
\Im\lambda\gg 1
\end{equation}
where we treat $P_k$ as an unbounded operator on $L^2$ with domain
$\{u\in L^2(M;\Omega^k)\mid P_ku\in L^2(M;\Omega^k)\}$
and $P_ku$ is defined in the sense of distributions.

\subsubsection{Meromorphic continuation}

Since $\varphi_t$ is an Anosov flow, the resolvent $R_k(\lambda)$ admits a meromorphic continuation
$$
R_k(\lambda):C^\infty(M;\Omega^k)\to \mathcal D'(M;\Omega^k),\quad
\lambda\in\mathbb C,
$$
see for instance~\cite[\S3.2]{dyatlov-zworski-16} and~\cite[Theorems~1.4, 1.5]{faure-sjostrand-11}.
The proof of this continuation shows that $R_k(\lambda)$ acts on certain anisotropic Sobolev spaces
adapted to the stable/unstable decompositions, see e.g.~\cite[\S3.1]{dyatlov-zworski-16};
this makes it possible to compose the operator $R_k(\lambda)$ with itself.
Instead of introducing these spaces here, we use the spaces of distributions
\begin{equation}
  \label{e:D-prime-Gamma}
\mathcal D'_\Gamma(M;\Omega^k):=\{u\in\mathcal D'(M;\Omega^k)\mid \WF(u)\subset \Gamma\}
\end{equation}
where $\Gamma\subset T^*M\setminus 0$ is a closed conic set and $\WF(u)$ denotes the
wavefront set of a distribution~$u$. These spaces come with a natural sequential topology,
see~\cite[Definition~8.2.2]{hoermander-03}.

We have the
wavefront set property of~$R_k(\lambda)$ proved
in~\cite[(3.7)]{dyatlov-zworski-16}:
\begin{equation}
  \label{e:R-k-WFset}
\WF'(R_k(\lambda))\ \subset\ \mathscr W:=\Delta(T^*M)\cup \Upsilon_+\cup(E_u^*\times E_s^*)
\end{equation}
where $\Delta(T^*M)\subset T^*M\times T^*M$ is the diagonal and
$\Upsilon_+=\{(\varphi_t(x),d\varphi_t(x)^{-T}\xi,x,\xi)\mid t\geq 0,\ \xi(X(x))=0\}$; for an operator $B:C^\infty(M) \to \mathcal{D}'(M)$ with Schwartz kernel $K_B \in \mathcal{D}'(M \times M)$, we denote $\WF'(B) = \{(x, \xi, y, -\eta) \mid (x, \xi, y, \eta) \in \WF(K_B)\} \subset T^*(M \times M)$. The Schwartz kernel of~$R_k(\lambda)$ is meromorphic in~$\lambda$ with values in $\mathcal D'_{\mathscr W'}$ where $\mathscr W':=\{(x,\xi,y,-\eta)\mid (x,\xi,y,\eta)\in\mathscr W\}$. By the wavefront set calculus~\cite[Theorem~8.2.13]{hoermander-03}
and since $E_u^*\cap E_s^*=0$, $R_k(\lambda)$ defines
a meromorphic family of continuous operators 
\begin{equation}
  \label{e:mapper}
R_k(\lambda):\mathcal D'_{E_u^*}(M;\Omega^k)\to\mathcal D'_{E_u^*}(M;\Omega^k)
\end{equation}
where we view $E_u^*\subset T^*M$ as a closed conic subset
and define $\mathcal D'_{E_u^*}$ by~\eqref{e:D-prime-Gamma}.

Note that differential operators (in particular, $d,\iota_X,\mathcal L_X$) define continuous maps
on the regularity classes~$\mathcal D'_{E_u^*}$. We have
\begin{equation}
  \label{e:invertor}
R_k(\lambda)(P_k-\lambda)u=(P_k-\lambda)R_k(\lambda)u=u\quad\text{for all }
u\in\mathcal D'_{E_u^*}(M;\Omega^k).
\end{equation}
For $\Im\lambda\gg 1$ and $u\in C^\infty(M;\Omega^k)$ this follows from~\eqref{e:resolver};
the general case follows from here by analytic continuation and since $C^\infty$ is dense
in $\mathcal D'_{E_u^*}$.

We also have the commutation relations
\begin{equation}
  \label{e:commuter}
dR_k(\lambda)u=R_{k+1}(\lambda)du,\quad
\iota_X R_k(\lambda)u=R_{k-1}(\lambda)\iota_Xu\quad\text{for all}\quad u\in\mathcal D'_{E_u^*}(M;\Omega^k).
\end{equation}
As with~\eqref{e:invertor} it suffices to consider the case
$\Im\lambda\gg 1$ and $u\in C^\infty(M;\Omega^k)$, in which~\eqref{e:commuter}
follows from~\eqref{e:Rk-uhp} and the fact that $d$ and $\iota_X$ commute with $\varphi_{-t}^*$.

The poles of the family of operators $R_k(\lambda)$ are called \emph{Pollicott--Ruelle resonances}
on $k$-forms. At each pole $\lambda_0\in\mathbb C$ we have an expansion
(see for instance~\cite[(3.6)]{dyatlov-zworski-16})
\begin{equation}
  \label{e:merex}
R_k(\lambda)=R^H_k(\lambda;\lambda_0)-\sum_{j=1}^{J_k(\lambda_0)}{(P_k-\lambda_0)^{j-1}\Pi_k(\lambda_0)\over (\lambda-\lambda_0)^j}
\end{equation}
where $R^H_k(\lambda;\lambda_0):\mathcal D'_{E_u^*}(M;\Omega^k)\to \mathcal D'_{E_u^*}(M;\Omega^k)$ is a family of operators holomorphic in a neighborhood of~$\lambda_0$,
$J_k(\lambda_0)\geq 1$ is an integer, 
and $\Pi_k(\lambda_0):\mathcal D'_{E_u^*}(M;\Omega^k)\to \mathcal D'_{E_u^*}(M;\Omega^k)$
is a finite rank operator commuting with $P_k$ and such that $(P_k-\lambda_0)^{J_k(\lambda_0)}\Pi_k(\lambda_0)=0$.

Taking the expansions of~\eqref{e:commuter} at $\lambda_0$ we see that
\begin{equation}
  \label{e:commuter-Pi}
d\Pi_k(\lambda_0)=\Pi_{k+1}(\lambda_0)d,\quad
\iota_X \Pi_k(\lambda_0)=\Pi_{k-1}(\lambda_0)\iota_X.
\end{equation}

\subsubsection{Resonant states}
\label{s:resonant-states}

The range of the operator $\Pi_k(\lambda_0)$ is equal to the space of \emph{generalised resonant states} (see for instance~\cite[Proposition~3.3]{dyatlov-zworski-16})
\begin{equation}
  \label{e:res-k-infty}
\Res^{k,\infty}(\lambda_0):=\bigcup_{\ell\geq 1}\Res^{k,\ell}(\lambda_0)
\end{equation}
where we define
\begin{equation}
  \label{e:res-k-l}
\Res^{k,\ell}(\lambda_0):=\{u\in \mathcal D'_{E_u^*}(M;\Omega^k)\mid (P_k-\lambda_0)^\ell u=0\}.
\end{equation}
We define the \emph{algebraic multiplicity} of $\lambda_0$ as a resonance on $k$-forms by
\begin{equation}
  \label{e:m-k-def}
m_k(\lambda_0):=\rank \Pi_k(\lambda_0)=\dim\Res^{k,\infty}(\lambda_0).
\end{equation}
The \emph{geometric multiplicity} is the dimension of the space
of \emph{resonant states}
$$
\Res^k(\lambda_0):=\Res^{k,1}(\lambda_0)=\{u\in\mathcal D'_{E_u^*}(M;\Omega^k)\mid (P_k-\lambda_0)u=0\}.
$$
We say a resonance $\lambda_0$ of $P_k$ is \emph{semisimple} if the algebraic and geometric multiplicities coincide,
that is $\Res^{k,\infty}(\lambda_0)=\Res^k(\lambda_0)$. This is equivalent to saying
that $J_k(\lambda_0)=1$ in~\eqref{e:merex}. Another equivalent definition of semisimplicity is
\begin{equation}
  \label{e:semi-simple}
u\in\mathcal D'_{E_u^*}(M;\Omega^k),\ (P_k-\lambda_0)^2u=0\quad\Longrightarrow\quad
(P_k-\lambda_0)u=0.
\end{equation}
We note that the operators $\Pi_k(\lambda_0)$ are idempotent. In fact,
applying the Laurent expansion~\eqref{e:merex} at $\lambda_0$ to $u\in\Res^{k,\ell}(\lambda_1)$ and using the identity $R_k(\lambda)u=-\sum_{j=0}^{\ell-1}(\lambda-\lambda_1)^{-j-1}(P_k-\lambda_1)^ju$ we see that
\begin{equation}
  \label{e:Pi-k-idem}
\Pi_k(\lambda_0)\Pi_k(\lambda_1)=\begin{cases}
\Pi_k(\lambda_0)&\text{if }\lambda_1=\lambda_0,\\
0&\text{if }\lambda_1\neq \lambda_0.
\end{cases}
\end{equation}

\subsubsection{Operators on the bundles $\Omega^k_0$}
\label{s:omega-k-0}

The above constructions apply equally as well to the operators $P_{k,0}$
(except that the operator $d$ does not preserve sections of $\Omega^k_0$,
so the first commutation relation in~\eqref{e:commuter-Pi} does not hold, and the
second one is trivial); we denote the resulting objects
by
$$
R_{k,0}(\lambda),\
J_{k,0}(\lambda_0),\
R^H_{k,0}(\lambda;\lambda_0),\
\Pi_{k,0}(\lambda_0),\
\Res^{k,\ell}_0(\lambda_0),\
m_{k,0}(\lambda_0).
$$
Under the isomorphism~\eqref{e:0-split} the operator $P_k$ is conjugated to $P_{k,0}\oplus P_{k-1,0}$.
Therefore~\eqref{e:0-split} gives an isomorphism
\begin{equation}
  \label{e:res-0-split}
\Res^{k,\ell}(\lambda_0)\simeq \Res^{k,\ell}_0(\lambda_0)\oplus \Res^{k-1,\ell}_0(\lambda_0).
\end{equation}
Moreover, we get for all $u\in\mathcal D'_{E_u^*}(M;\Omega^k)$
\begin{equation}
\label{e:Pi-k-split}
\Pi_k(\lambda_0)u=\Pi_{k,0}(\lambda_0)(u-\alpha\wedge \iota_X u)
+\alpha\wedge \Pi_{k-1,0}(\lambda_0)\iota_X u.
\end{equation}
Since $\mathcal L_Xd\alpha=0$,
the operations~\eqref{e:d-alpha-wedge} give rise to linear isomorphisms
\begin{equation}
\label{e:isom-1}
d\alpha\wedge:\Res^{1,\ell}_0(\lambda_0)\to \Res^{3,\ell}_0(\lambda_0),\quad
d\alpha\wedge^2:\Res^{0,\ell}_0(\lambda_0)\to \Res^{4,\ell}_0(\lambda_0)
\end{equation}
which in particular give the equalities
\begin{equation}
  \label{e:isom-2}
m_{1,0}(\lambda_0)=m_{3,0}(\lambda_0),\quad
m_{0,0}(\lambda_0)=m_{4,0}(\lambda_0).
\end{equation}

\subsubsection{Transposes and coresonant states}
\label{s:transposes}

Since $\mathcal L_X\alpha=0$ and $\int_M \mathcal L_X\omega=0$ for any 5-form~$\omega$, we have
\begin{equation}\label{e:P-transpose}
	(P_{k,0})^T=-P_{4-k,0}, \quad k = 0, 1, 2, 3, 4,
\end{equation}
where the transpose is defined using the pairing $\llangle\bullet,\bullet\rrangle$,
see~\eqref{e:transpose}.
Thus the transpose of the resolvent $(R_{k,0}(\lambda))^T$ is the meromorphic continuation
of the resolvent corresponding to the vector field $-X$; the latter generates an Anosov flow
with the unstable and stable spaces switching roles compared to the ones for~$X$. Similarly to~\eqref{e:mapper} we have
\begin{equation}\label{eq:restranspose}
(R_{k,0}(\lambda))^T:\mathcal D'_{E_s^*}(M;\Omega^{4-k}_0)\to \mathcal D'_{E_s^*}(M;\Omega^{4-k}_0)
\end{equation}
where $\mathcal D'_{E_s^*}$ is the space of distributional sections with wavefront set contained in $E_s^*$.
Same applies to the transposes of the operators $R^H_{k,0}(\lambda;\lambda_0)$ and $\Pi_{k,0}(\lambda_0)$
appearing in~\eqref{e:merex}. The range of $(\Pi_{k,0}(\lambda_0))^T$ is the space
of \emph{generalised coresonant states}
$\Res^{4-k,\infty}_{0*}(\lambda_0)$ where
$$
\begin{aligned}
\Res^{k,\infty}_{0*}(\lambda_0)&\,:=\bigcup_{\ell\geq 1}\Res^{k,\ell}_{0*}(\lambda_0),\\
\Res^{k,\ell}_{0*}(\lambda_0)&\,:=\{u_*\in \mathcal D'_{E_s^*}(M;\Omega^k_0)\mid (P_{k,0}+\lambda_0)^\ell u_*=0\}.
\end{aligned}
$$
The space of \emph{coresonant states} is defined as
$$
\Res^k_{0*}(\lambda_0):=\Res^{k,1}_{0*}(\lambda_0)=\{u_*\in \mathcal D'_{E_s^*}(M;\Omega^k_0)\mid
(P_{k,0}+\lambda_0)u_*=0\}.
$$
Similarly to~\eqref{e:isom-1} we have the isomorphisms
\begin{equation}
  \label{e:isom-1*}
d\alpha\wedge:\Res^{1,\ell}_{0*}(\lambda_0)\to \Res^{3,\ell}_{0*}(\lambda_0),\quad
d\alpha\wedge^2:\Res^{0,\ell}_{0*}(\lambda_0)\to \Res^{4,\ell}_{0*}(\lambda_0).
\end{equation}
In the special case when $\varphi_t$ is a geodesic flow with the time reversal map $\mathcal J$ defined in~\eqref{e:J-def}, 
the pullback operator $\mathcal J^*$ gives an isomorphism between
$\mathcal D'_{E_u^*}(M;\Omega^k_0)$ and $\mathcal D'_{E_s^*}(M;\Omega^k_0)$.
Moreover, $\mathcal J^*P_{k,0}=-P_{k,0}\mathcal J^*$.
This gives rise to isomorphisms between the spaces of generalised resonant and coresonant states
\begin{equation}
\label{e:J-star-def}
\mathcal J^*:\Res^{k,\ell}_0(\lambda_0)\to \Res^{k,\ell}_{0*}(\lambda_0).
\end{equation}

\subsubsection{Coresonant states and pairing}
Since $E_u^*$ and $E_s^*$ intersect only at the zero section, we can define the product
$u\wedge u_*\in \mathcal D'(M;\Omega^4_0)$ and thus the pairing
$\llangle u,u_*\rrangle$ for any $u\in\mathcal D'_{E_u^*}(M;\Omega^k_0)$,
$u_*\in \mathcal D'_{E_s^*}(M;\Omega^{4-k}_0)$, see~\cite[Theorem~8.2.10]{hoermander-03}. Note that this pairing is nondegenerate since both $\mathcal D'_{E_u^*}$ and $\mathcal D'_{E_s^*}$ contain $C^\infty$,
and the transpose formula~\eqref{e:transpose} still holds
since $C^\infty$ is dense in $\mathcal D'_{E_u^*}$ and in $\mathcal D'_{E_s^*}$. In particular, we have
a pairing
\begin{equation}
  \label{e:pairing-res}
u\in\Res^{k,\infty}_0(\lambda_0),\
u_*\in\Res^{4-k,\infty}_{0*}(\lambda_0)
\quad\mapsto\quad
\llangle u,u_*\rrangle\in\mathbb C.
\end{equation}
This pairing is nondegenerate. Indeed, assume that $u\in\Res^{k,\infty}_0(\lambda_0)$
and $\llangle u,u_*\rrangle=0$ for all $u_*\in\Res^{4-k,\infty}_{0*}(\lambda_0)$.
Since $\Res^{4-k,\infty}_{0*}(\lambda_0)$ is the range of $(\Pi_{k,0}(\lambda_0))^T$,
we have
$$
0=\llangle u,(\Pi_{k,0}(\lambda_0))^T\varphi\rrangle=\llangle\Pi_{k,0}(\lambda_0)u,\varphi\rrangle
=\llangle u,\varphi\rrangle\quad\text{for all }
\varphi\in C^\infty(M;\Omega^{4-k}_0)
$$
where the last equality follows from the fact that $\Pi_{k,0}(\lambda_0)^2=\Pi_{k,0}(\lambda_0)$
and $u$ is in the range of $\Pi_{k,0}(\lambda_0)$. It follows that $u=0$.
Similarly one can show that if $\llangle u,u_*\rrangle=0$ for some $u_*\in\Res^{4-k,\infty}_{0*}(\lambda_0)$
and all $u\in\Res^{k,\infty}_0(\lambda_0)$, then $u_*=0$.

Consider the operators on finite dimensional spaces
\begin{align}
  \label{e:fdop-1}
P_{k,0}-\lambda_0&:\Res^{k,\infty}_0(\lambda_0)\to \Res^{k,\infty}_0(\lambda_0),\\
  \label{e:fdop-2}
-P_{4-k,0}-\lambda_0&:\Res^{4-k,\infty}_{0*}(\lambda_0)\to \Res^{4-k,\infty}_{0*}(\lambda_0)
\end{align}
which are transposes of each other with respect to the pairing~\eqref{e:pairing-res}.
The kernels of $\ell$-th powers of these operators are $\Res^{k,\ell}_0(\lambda_0)$ and $\Res^{4-k,\ell}_{0*}(\lambda_0)$, thus
(using the isomorphisms~\eqref{e:isom-1*})
\begin{equation}
  \label{e:same-dim}
\dim\Res^{k,\ell}_0(\lambda_0)=\dim\Res^{4-k,\ell}_{0*}(\lambda_0)=\dim\Res^{k,\ell}_{0*}(\lambda_0).
\end{equation}
We now give a solvability result for the operators~$P_{k,0}$. It follows from
the Fredholm property of these operators on anisotropic Sobolev spaces
but we present instead a proof using the Laurent expansion~\eqref{e:merex}.
\begin{lemm}
    \label{l:solver}
Assume that $w\in \mathcal D'_{E_u^*}(M;\Omega^k_0)$. Then the equation
\begin{equation}
  \label{e:solvereq}
(P_{k,0}-\lambda_0)u=w,\quad
u\in\mathcal D'_{E_u^*}(M;\Omega^k_0)
\end{equation}
has a solution if and only if $w$ satisfies the condition
\begin{equation}
\label{e:solver}
\llangle w,u_*\rrangle=0\quad\text{for all}\quad u_*\in\Res^{4-k}_{0*}(\lambda_0).
\end{equation}
\end{lemm}
\begin{proof}
First of all, if~\eqref{e:solvereq} has a solution $u$, then for each
$u_*\in\Res^{4-k}_{0*}(\lambda_0)$ we have
$$
\llangle w,u_*\rrangle=\llangle (P_{k,0}-\lambda_0)u,u_*\rrangle=-\llangle u,(P_{4-k,0}+\lambda_0)u_*\rrangle=0,
$$
that is the condition~\eqref{e:solver} is satisfied.

Now, assume that $w$ satisfies the condition~\eqref{e:solver};
we show that~\eqref{e:solvereq} has a solution.
We start with the special case when $w\in\Res^{k,\infty}_0(\lambda_0)$.
We use the pairing~\eqref{e:pairing-res} to identify the
dual space to $\Res^{k,\infty}_0(\lambda_0)$ with $\Res^{4-k,\infty}_{0*}(\lambda_0)$.
By~\eqref{e:solver}, $w$ is annihilated by the kernel
of the operator~\eqref{e:fdop-2}. Therefore
$w$ is in the range of the operator~\eqref{e:fdop-1},
that is~\eqref{e:solvereq} has a solution $u\in\Res^{k,\infty}_0(\lambda_0)$.

We now consider the case of general $w$ satisfying~\eqref{e:solver}. Taking the constant term in the Laurent expansion of the identity~\eqref{e:invertor} at $\lambda=\lambda_0$, we obtain
\begin{equation}\label{eq:resid}
(P_{k,0}-\lambda_0)R_{k,0}^H(\lambda_0;\lambda_0)w=w-\Pi_{k,0}(\lambda_0)w.
\end{equation}
We have $\Pi_{k,0}(\lambda_0)w\in \Res^{k,\infty}_0(\lambda_0)$ and it satisfies~\eqref{e:solver},
thus~\eqref{e:solvereq} has a solution with this right-hand side. Writing $w = \Pi_{k, 0}(\lambda_0)w + \big(\Id - \Pi_{k, 0}(\lambda_0)\big)w$, we may take as $u$ the sum of
this solution and $R_{k,0}^H(\lambda_0;\lambda_0)w$.
\end{proof}
Lemma~\ref{l:solver} implies the following criterion for semisimplicity:
\begin{lemm}
  \label{l:ss-cores}
The semisimplicity condition~\eqref{e:semi-simple} holds for the operator~$P_{k,0}$ if and only if
the restriction of the pairing~\eqref{e:pairing-res} to~$\Res^k_0(\lambda_0)\times\Res^{4-k}_{0*}(\lambda_0)$
is nondegenerate.
\end{lemm}
\begin{proof}
The condition~\eqref{e:semi-simple} is equivalent to saying that the intersection
of $\Res^k_0(\lambda_0)$ with the range of the operator $P_{k,0}-\lambda_0:\mathcal D'_{E_u^*}(M;\Omega^k_0)\to\mathcal D'_{E_u^*}(M;\Omega^k_0)$ is trivial; that is, for each $w\in\Res^k_0(\lambda_0)\setminus \{0\}$
the equation~\eqref{e:solvereq} has no solution. By Lemma~\ref{l:solver}, this is equivalent to
saying that $w$ does not satisfy the condition~\eqref{e:solver}, i.e. there exists
$v\in \Res^{4-k}_{0*}(\lambda_0)$ such that $\llangle w,v\rrangle\neq 0$. This is equivalent
to the nondegeneracy condition of the present lemma.
\end{proof}

\subsubsection{Zeta functions}
We now discuss dynamical zeta functions. We assume that the unstable/stable bundles $E_u,E_s$
are orientable (the non-orientable case is covered by \cite{Borns-Weil-Shen-21}); this is true for the case of geodesic flows on orientable manifolds
as follows from the fact that the vertical bundle trivially intersects the weak unstable bundle $\mathbb{R}X \oplus E_u$ (see \cite[Lemma B.1]{giuletti-liverani-pollicott-13}).

We say $\gamma:[0,T_\gamma]\to M$ is a \emph{closed trajectory} of the flow~$\varphi_t$ of period~$T_\gamma>0$ if $\gamma(t)=\varphi_t(\gamma(0))$ and $\gamma(T_\gamma)=\gamma(0)$. We identify closed trajectories
obtained by shifting $t$.
The \emph{primitive period} of a closed trajectory, denoted by $T_\gamma^\sharp$, is
the smallest positive $t>0$ such that $\gamma(t)=\gamma(0)$.
We say $\gamma$ is a \emph{primitive} closed trajectory if $T_\gamma=T_\gamma^\sharp$.

Define the \emph{linearised Poincar\'e map} $\mathcal P_\gamma:=d\varphi_{-T_\gamma}(\gamma(0))|_{E_u\oplus E_s}$.
We have $\det\mathcal P_\gamma=1$ since the restriction of $d\alpha\wedge d\alpha$ to $E_u\oplus E_s$ is a $\varphi_t$-invariant
nonvanishing 4-form.
Since $\varphi_t$ is an Anosov flow, the map $I-\mathcal P_\gamma$ is invertible
(in fact $\mathcal P_\gamma$ has no eigenvalues on the unit circle). 

For $0\leq k\leq 4$, define the zeta function
\begin{equation}
  \label{e:zeta-k}
\zeta_k(\lambda):=\exp\bigg(-\sum_{\gamma}{T_\gamma^\sharp \tr(\wedge^k\mathcal P_\gamma)e^{i\lambda T_\gamma}
\over T_\gamma \det(I-\mathcal P_\gamma)}\bigg),\quad
\Im\lambda\gg 1
\end{equation}
where the sum is over all the closed trajectories $\gamma$.
The series in~\eqref{e:zeta-k} converges for sufficiently large $\Im\lambda$, see e.g.~\cite[\S2.2]{dyatlov-zworski-16}.

The zeta function $\zeta_k$ continues holomorphically to $\lambda\in\mathbb C$ and for each $\lambda_0\in\mathbb C$,
the multiplicity of $\lambda_0$ as a zero of $\zeta_k$ is equal to $m_{k,0}(\lambda_0)$, the algebraic multiplicity
of $\lambda_0$ as a resonance of the operator $P_{k,0}$ defined similarly to~\eqref{e:m-k-def}~-- see~\cite[\S4]{dyatlov-zworski-16}
for the proof.

By Ruelle's identity (see e.g.~\cite[(2.5)]{dyatlov-zworski-16}) the Ruelle zeta function defined in~\eqref{e:ruelle-zeta}
factorizes as follows:
$$
\zeta_{\mathrm R}(\lambda)={\zeta_0(\lambda)\zeta_2(\lambda)\zeta_4(\lambda)\over \zeta_1(\lambda)\zeta_3(\lambda)}.
$$
Using~\eqref{e:isom-2} we see that the order of vanishing of the function $\zeta_{\mathrm R}$ at $\lambda_0$ is equal to
\begin{equation}
  \label{e:alternator}
m_{\mathrm R}(\lambda_0)=
\sum_{k=0}^4 (-1)^k m_{k,0}(\lambda_0)=2m_{0,0}(\lambda_0)-2m_{1,0}(\lambda_0)+m_{2,0}(\lambda_0).
\end{equation}

\subsection{Resonance at~0}
\label{subsec:resatzero}

This paper focuses on the resonance at~$0$, which is why we henceforth put $\lambda_0:=0$ unless
stated otherwise. For instance we write
$$
R^H_{k,0}(\lambda):=R^H_{k,0}(\lambda;0),\quad
\Pi_{k,0}:=\Pi_{k,0}(0),\quad
\Res^{k,\ell}_0:=\Res^{k,\ell}_0(0).
$$
Our main goal is to study the order of vanishing of the Ruelle zeta function at~0, which by~\eqref{e:alternator}
is equal to
$$
m_{\mathrm R}(0)=2m_{0,0}(0)-2m_{1,0}(0)+m_{2,0}(0),\quad
m_{k,0}(0)=\dim\Res^{k,\infty}_0.
$$
Since $\mathcal L_X=d\iota_X+\iota_Xd$, the space of resonant states at 0 for the operator $P_{k,0}$ is
\begin{equation}
  \label{e:res-k-0}
\Res^k_0=\{u\in\mathcal D'_{E_u^*}(M;\Omega^k)\mid \iota_X u=0,\ \iota_X du=0\}.
\end{equation}
In particular, the exterior derivative defines an operator
$d:\Res^k_0\to \Res^{k+1}_0$. (Unfortunately this is no longer true for the spaces
of generalised resonant states $\Res^{k,\ell}_0$ with $\ell\geq 2$,
since $d$ does not necessarily map these to the kernel of $\iota_X$.)

\subsubsection{0-forms and 4-forms}
We first analyze the resonance at~0 for the operators $P_{0,0}$ and $P_{4,0}$. The following
regularity result is a special case of~\cite[Lemma~2.3]{dyatlov-zworski-17} (see also \cite[Lemma 4]{faure-roy-sjostrand} for a similar statement in the case of Anosov maps):
\begin{lemm}
\label{l:regres}
Assume that
$$
u\in\mathcal D'_{E_u^*}(M;\mathbb C),\quad
Xu\in C^\infty(M;\mathbb C),\quad
\Re\langle Xu,u\rangle_{L^2(M;d\vol_\alpha)}\leq 0.
$$
Then $u\in C^\infty(M;\mathbb C)$.
\end{lemm}
Using Lemma~\ref{l:regres} we show the following statement similar to~\cite[Lemma~3.2]{dyatlov-zworski-17} (we note that it straightforwardly generalizes to other dimensions, which was known already to \cite[Corollary 2.11]{Liverani1}):
\begin{lemm}
  \label{l:0-forms}
The semisimplicity condition~\eqref{e:semi-simple} holds at $\lambda_0=0$ for the operators $P_{0,0}$, $P_{4,0}$ and
$$
m_{0,0}(0)=m_{4,0}(0)=1.
$$
Moreover, $\Res^0_0=\Res^0_{0*}$ is spanned by the constant function 1 and $\Res^4_0=\Res^4_{0*}$ is spanned
by the form $d\alpha\wedge d\alpha$.
\end{lemm}
\begin{proof}
We only give the proof for 0-forms (i.e. functions); the case
of 4-forms follows from here using the isomorphisms~\eqref{e:isom-1}, \eqref{e:isom-1*}.

Assume that $u\in\Res^0_0$. Then $Xu=0$, so Lemma~\ref{l:regres} implies that $u\in C^\infty(M;\mathbb C)$.
Thus the differential $du\in C^\infty(M;\Omega^1)$ is invariant under the flow $\varphi_t$;
the stable/unstable decomposition~\eqref{e:anosov-dual} gives that $du\in E_0^*$ at every point. Together
with the equation $Xu=0$, this implies that $du=0$ and thus (since $M$ is connected) $u$ is constant.
We have shown that $\Res^0_0$ is spanned by the function~1; applying the above argument to~$-X$
we see that $\Res^0_{0*}$ is spanned by~1 as well.

To show the semisimplicity condition~\eqref{e:semi-simple}, assume
that $u\in\mathcal D'_{E_u^*}(M;\mathbb C)$ satisfies $X^2u=0$.
Then $Xu\in\Res^0_0$, so $Xu$ is constant.
Together with the identity $\int_M (Xu)\,d\vol_\alpha=0$
this gives $Xu=0$ as needed.
\end{proof}

\subsubsection{Closed forms}
We now study resonant states which are closed, that is elements of the space
$$
\Res^k_0\cap\ker d=\{u\in\mathcal D'_{E_u^*}(M;\Omega^k)\mid \iota_X u=0,\ du=0\}.
$$
We use a special case of~\cite[Lemma~2.1]{dyatlov-zworski-17} which shows that de Rham cohomology
in the spaces $\mathcal D'_{E_u^*}(M;\Omega^k)$ is the same as the usual de Rham cohomology
defined in~\eqref{e:de-rham}:
\begin{lemm}
  \label{l:hodge}
Assume that $u\in\mathcal D'_{E_u^*}(M;\Omega^k)$ and $du\in C^\infty(M;\Omega^{k+1})$.
Then there exist $v\in C^\infty(M;\Omega^k)$, $w\in \mathcal D'_{E_u^*}(M;\Omega^{k-1})$
such that $u=v+dw$.
\end{lemm}
Similarly to~\cite[\S3.3]{dyatlov-zworski-17} we introduce the linear map
\begin{equation}
  \label{e:pi-k-def}
\begin{gathered}
\pi_k:\Res^k_0\cap\ker d\to H^k(M;\mathbb C),\quad
\pi_k(u)=[v]_{H^k}
\\\text{where}\quad u=v+dw,\quad
v\in C^\infty(M;\Omega^k),\quad
w\in \mathcal D'_{E_u^*}(M;\Omega^{k-1}).
\end{gathered}
\end{equation}
Here $v,w$ exist by Lemma~\ref{l:hodge}. To show that the map $\pi_k$ is well-defined,
assume that $u=v+dw=v'+dw'$ where $v,v'\in C^\infty(M;\Omega^k)$ and $w,w'\in\mathcal D'_{E_u^*}(M;\Omega^{k-1})$.
Then $d(w-w')=v'-v\in C^\infty(M;\Omega^k)$, thus by Lemma~\ref{l:hodge}
we may write $w-w'=w_1+dw_2$ where $w_1\in C^\infty(M;\Omega^{k-1})$,
$w_2\in \mathcal D'_{E_u^*}(M;\Omega^{k-2})$.
Then $v'-v=dw_1$ where $w_1$ is smooth, so $[v]_{H^k}=[v']_{H^k}$.

Similar arguments apply to the spaces $\Res^{k}_{0*}\cap \ker d$ of closed coresonant $k$-forms;
we denote the corresponding maps by
$$
\pi_{k*}:\Res^{k}_{0*}\cap\ker d\to H^k(M;\mathbb C).
$$
From Lemma~\ref{l:0-forms} we see that $\pi_0$ is an isomorphism
and hence by \eqref{e:isom-1} that $\pi_4=0$. 

We now establish several properties of the spaces $\Res^k_0\cap\ker d$ and the
maps $\pi_k$; some of these are extensions of the results of~\cite[\S3.3]{dyatlov-zworski-17}.
\begin{lemm}
  \label{l:ker-pi-k}
The kernel of $\pi_k$ satisfies
$$
d(\Res^{k-1}_0)\subset\ker \pi_k\subset d(\Res^{k-1,\infty}).
$$
\end{lemm}
\begin{proof}
The first containment is immediate. For the second one, assume that $u\in\Res^k_0\cap \ker d$
and $\pi_k(u)=0$. Then $u=v+dw$ where $v\in C^\infty(M;\Omega^k)$ satisfies $[v]_{H^k}=0$
and $w\in\mathcal D'_{E_u^*}(M;\Omega^{k-1})$. We have $v=d\zeta$ for some $\zeta\in C^\infty(M;\Omega^{k-1})$
and by~\eqref{e:commuter-Pi}
$$
u=\Pi_ku=\Pi_k d(\zeta+w)=d\Pi_{k-1}(\zeta+w).
$$
Therefore $u\in d(\Res^{k-1,\infty})$.
\end{proof}
%
We note that the case $k = 0$ of the following lemma holds trivially.
\begin{lemm}
  \label{l:exact-onto}
Assume that for some $k$ all the coresonant states in $\Res^{5-k}_{0*}$ are exact forms. Then the map $\pi_k$ is onto.
\end{lemm}
\begin{proof}
Take arbitrary $v\in C^\infty(M;\Omega^k)$ such that $dv=0$.
We will construct $u\in\Res^k_0\cap\ker d$ such that $\pi_k(u)=[v]_{H^k}$ by putting
$$
u:=v+dw\quad\text{for some}\quad w\in\mathcal D'_{E_u^*}(M;\Omega^{k-1}_0).
$$
Such $u$ is automatically closed, so we only need to choose $w$ so that
$\iota_Xu=0$, that is
\begin{equation}
  \label{e:solvent}
\iota_Xdw=\mathcal L_Xw=-\iota_X v
\end{equation}
where the first equality is immediate because $\iota_X w=0$.

To solve~\eqref{e:solvent}, we use Lemma~\ref{l:solver}. It suffices to check that
the condition~\eqref{e:solver} holds:
$$
\llangle \iota_X v,u_*\rrangle=0\quad\text{for all}\quad u_*\in\Res^{5-k}_{0*}.
$$
We compute
$$
\llangle\iota_X v,u_*\rrangle=\int_M \alpha\wedge(\iota_X v)\wedge u_*
=\int_M v\wedge u_*=0.
$$
Here in the second equality we used that $\iota_X u_*=0$ (thus $\iota_X$ of the 5-forms on both
sides are the same) and in the last equality
we used that $v$ is closed and, by the assumption of the lemma, $u_*$ is exact.
\end{proof}
%
\begin{lemm}
\label{l:1-forms-ish}
The maps $\pi_1,\pi_{1*}$ are isomorphisms, in particular
$$
\dim(\Res^1_0\cap\ker d)=\dim(\Res^1_{0*}\cap\ker d)=b_1(M).
$$
\end{lemm}
\begin{proof}
We only consider the case of $\pi_1$, with $\pi_{1*}$ handled similarly.
To show that $\pi_1$ is one-to-one, we use Lemma~\ref{l:ker-pi-k}
and the fact that $\Res^{0,\infty}=\Res^0_0$ consists of constant functions
by Lemma~\ref{l:0-forms}.
To show that $\pi_1$ is onto, it suffices to use Lemma~\ref{l:exact-onto}:
by Lemma~\ref{l:0-forms}, the space $\Res^4_{0*}$ is spanned by $d\alpha\wedge d\alpha=d(\alpha\wedge d\alpha)$.
\end{proof}
%
\begin{lemm}
  \label{l:res-3-closed}
We have $d(\Res^3_0)=d(\Res^3_{0*})=0$.
\end{lemm}
\begin{proof}
We only consider the case of $\Res^3_0$, with $\Res^3_{0*}$ handled similarly.
Assume that $u\in\Res^3_0$. Then $du\in\Res^4_0$, so by Lemma~\ref{l:0-forms} we have
$du=cd\alpha\wedge d\alpha$ for some constant $c$. It remains to use that
$$
c\int_M d\vol_\alpha=\int_M \alpha\wedge du=\int_M d\alpha\wedge u=0
$$
where in the second equality we integrated by parts and in the third equality we used
that $\iota_X(d\alpha\wedge u)=0$, thus $d\alpha\wedge u=0$.
\end{proof}
We also have the following nondegeneracy result for the pairing between closed resonant and coresonant forms
when $k=1$:
\begin{lemm}
  \label{l:pairing-k-1}
The pairing induced by $\llangle\bullet,\bullet\rrangle$ on $(\Res^1_0\cap\ker d)\times(d\alpha\wedge(\Res^1_{0*}\cap\ker d))$ is nondegenerate.
\end{lemm}
\begin{proof}
We show the following stronger statement: for each closed but not exact $v\in C^\infty(M;\Omega^1)$,
\begin{equation}
  \label{e:pairing-k-1}
\Re \llangle \pi_1^{-1}([v]_{H^1}),d\alpha\wedge\pi_{1*}^{-1}([\overline v]_{H^1})\rrangle < 0.
\end{equation}
Here we used that the map $\pi_1$ is an isomorphism, as shown in Lemma \ref{l:1-forms-ish}. We have
$$
\pi_1^{-1}([v]_{H^1})=v+df,\quad
\pi_{1*}^{-1}([\overline v]_{H^1})=\overline{v+dg}
$$
where $f\in\mathcal D'_{E_u^*}(M;\mathbb C),g\in\mathcal D'_{E_s^*}(M;\mathbb C)$
satisfy
\begin{equation}
  \label{e:idator}
Xf=Xg=-\iota_X v.
\end{equation}
We compute
$$
\begin{aligned}
\Re\llangle \pi_1^{-1}([v]_{H^1}),d\alpha\wedge\pi_{1*}^{-1}([\overline v]_{H^1})\rrangle&=\Re\int_M \alpha\wedge d\alpha\wedge(v+df)\wedge (\overline{v+dg})\\
&=\Re\int_M \alpha\wedge d\alpha\wedge(df\wedge \overline v+v\wedge d\overline g+df\wedge d\overline g)\\
&=\Re\int_M d\alpha\wedge d\alpha\wedge(f\overline v-\overline g v-\overline g df)\\
&=\Re\int_M \big(f\iota_X \overline v-\overline g\iota_X v-(Xf)\overline{g}\big)d\vol_\alpha\\
&=-\Re\langle Xf,f\rangle_{L^2(M;d\vol_\alpha)}.
\end{aligned}
$$
Here in the second line we used that $\Re (v\wedge \overline v)=0$.
In the third line we integrated by parts and used that $dv=0$.
In the fourth line we used that $\iota_X d\alpha=0$
(the 5-forms under the integral are equal as can be seen by taking $\iota_X$ of both sides).
In the last line we used the identity~\eqref{e:idator}.

Thus, if~\eqref{e:pairing-k-1} fails, we have
$\Re\langle Xf,f\rangle_{L^2(M;d\vol_\alpha)}\leq 0$ which by Lemma~\ref{l:regres}
implies that $f\in C^\infty(M;\mathbb C)$ and thus $u:=\pi_1^{-1}([v]_{H^1})$ lies in
$\Res^1_0\cap C^\infty(M;\Omega^1)$.
Now the fact that $u$ is invariant under the flow $\varphi_t$
and the stable/unstable decomposition~\eqref{e:anosov-dual} imply
that $u\in E_0^*$ at each point, and the fact that $\iota_X u=0$ then gives $u=0$.
This shows that $v$ is exact, giving a contradiction.
\end{proof}
We finally give the following result in the case when all forms in $\Res^1_0$ are closed:
\begin{lemm}
  \label{l:1-means-2}
Assume that $\Res^1_0$ consists of closed forms, i.e. $d(\Res^1_0)=0$.
Then:

1. The semisimplicity condition~\eqref{e:semi-simple} holds at $\lambda_0=0$ for the operators $P_{1,0}$ and~$P_{3,0}$.

2. $d(\Res^2_0)=0$, $\pi_2$ is onto, and $\ker \pi_2$ is spanned by $d\alpha$.

3. $m_{1,0}(0)=m_{3,0}(0)=b_1(M)$, $\dim \Res^2_0=b_2(M)+1$, and $\pi_3=0$.
\end{lemm}
\Remark Lemma~\ref{l:1-means-2} does not provide full information on the resonance at~0
since it does not prove the semisimplicity condition for the operator $P_{2,0}$, and only \emph{assumes} that resonant forms $\Res_0^1$ are closed (in fact we will see that $d(\Res_0^1) \neq 0$ and $P_{2, 0}$ is not semisimple in the hyperbolic case when $b_1(M)>0$, see \S \ref{sec:reshyp}).
\begin{proof}
1. Since $\dim(\Res^1_0\cap \ker d)=\dim(\Res^1_{0*}\cap\ker d)$ by
Lemma~\ref{l:1-forms-ish},
and $\dim \Res^1_0=\dim \Res^1_{0*}$ by~\eqref{e:same-dim}, we have $d(\Res^1_{0*})=0$.
By~\eqref{e:isom-1*} we have $\Res^3_{0*}=d\alpha\wedge \Res^1_{0*}$.
Now Lemma~\ref{l:pairing-k-1} shows that $\llangle\bullet,\bullet\rrangle$ defines
a nondegenerate pairing on $\Res^1_0\times \Res^3_{0*}$, which by Lemma~\ref{l:ss-cores}
shows that
the semisimplicity condition~\eqref{e:semi-simple} 
holds at $\lambda_0=0$ for the operator
$P_{1,0}$. By~\eqref{e:isom-1} semisimplicity holds for $P_{3,0}$ as well.

\noindent 2. We first show that $\Res^2_0$ consists of closed forms. Assume that $\zeta\in \Res^2_0$,
then $d\zeta\in\Res^3_0$. By~\eqref{e:isom-1}, $d\zeta=d\alpha\wedge u$ for some $u\in \Res^1_0$.
Take arbitrary $u_*\in\Res^1_{0*}$. Then
\begin{equation}
  \label{e:mclean}
\llangle u,d\alpha\wedge u_*\rrangle=\int_M \alpha\wedge d\zeta\wedge u_*\\
=\int_M d\alpha\wedge \zeta\wedge u_*=0 
\end{equation}
Here in the second equality we integrate by parts and use that~$du_*=0$;
in the last equality we use that $\iota_X$ applied to the 5-form under the integral is equal to~0.
Now by Lemma~\ref{l:pairing-k-1} we have $u=0$, which means that $d\zeta=0$ as needed.

Next, by Lemma~\ref{l:ker-pi-k} we have $\ker \pi_2\subset d(\Res^{1,\infty})$. By~\eqref{e:res-0-split}, Lemma~\ref{l:0-forms}, and the fact
that $\Res^{1,\infty}_0=\Res^1_0$ we have $\Res^{1,\infty}=\Res^1_0\oplus \mathbb C\alpha$.
Since $d(\Res^1_0)=0$ and $d\alpha\in\ker \pi_2$, we see that $\ker \pi_2$ is spanned by $d\alpha$.

Finally, to show that $\pi_2$ is onto, it suffices to use Lemma~\ref{l:exact-onto}:
since all elements of $\Res^1_{0*}$ are closed, all elements of $\Res^3_{0*}=d\alpha\wedge \Res^1_{0*}$
are exact.

\noindent 3. This follows immediately from the above statements and Lemma~\ref{l:1-forms-ish}. To show that $\pi_3=0$
we note that $\Res^3_0=d\alpha\wedge \Res^1_0$ consists of exact forms.
\end{proof}

\subsubsection{Summary} We now briefly summarize the contents of this section. Lemma \ref{l:ss-cores} will often be used to interpret the semisimplicity condition \eqref{e:semi-simple} via the more tractable nondegeneracy of the pairing \eqref{e:main-pairing}. Next, Lemma \ref{l:0-forms} provides us with a definitive understanding of $\Res_0^{0, \infty}$ and $\Res_0^{4, \infty}$, which by the isomorphisms \eqref{e:isom-1*} reduces the problem to studying $\Res_0^{1, \infty}$ and $\Res_0^{2, \infty}$. As Theorem \ref{thm:main} shows, this is a complicated question, but Lemma \ref{l:1-forms-ish} says that $\Res_0^1 \cap \ker d$ is `stably topological', that is, it is always mapped isomorphically by~$\pi_1$ to $H^1(M)$. Moreover, if one can show $d(\Res_0^1) = 0$, Lemma \ref{l:1-means-2} shows that semisimplicity for $1$-forms is valid, which will be used in the perturbed picture in \S \ref{s:general-perturb}. Under the same assumption, we also know that $\Res_0^2$ is spanned by the `topological part' $\pi_2^{-1}(H^2(M))$ and the form $d\alpha$. Thus, to compute \eqref{e:alternator} it suffices to study conditions under which forms in $\Res_0^1$ are closed, and semisimplicity conditions for $P_{2, 0}$. This will be done in two steps: in \S \ref{sec:reshyp} we will first develop a detailed understanding when $\varphi_t$ is the geodesic flow of a hyperbolic $3$-manifold, and later in \S \ref{s:general-perturb} we will study the perturbed picture.

\section{Resonant states for hyperbolic 3-manifolds}
\label{sec:reshyp}

In this section we study in detail the Pollicott--Ruelle resonant states at~0 for geodesic flows on hyperbolic 3-manifolds. The theorem below summarizes the main results. Here
$\Res^k_0=\Res^{k,1}_0$ are the spaces of resonant $k$-forms, $\Res^{k,\ell}_0$
are the spaces of generalized resonant $k$-forms (see~\S\ref{subsec:resatzero}),
and $\pi_k:\Res^k_0\cap\ker d\to H^k(M;\mathbb C)$ are the maps defined in~\eqref{e:pi-k-def}.
The maps $\pi_\Sigma^*$, $\pi_{\Sigma*}^{}$ are defined in~\S\ref{s:de-rham-sphere}.
\begin{theo}
\label{t:hyperbolic}
Let $M=S\Sigma$ where $\Sigma$ is a hyperbolic 3-manifold
and $\varphi_t$ be the geodesic flow on $\Sigma$. Then:

1. There exists a 2-form $\psi\in C^\infty(M;\Omega^2_0)$ which is closed
but not exact, $\pi_{\Sigma*}^{}(\psi)=-4\pi$, and $\psi$ is invariant under
$\varphi_t$.

2. $\Res^1_0=\mathcal C\oplus\mathcal C_\psi$ is $2b_1(\Sigma)$-dimensional where $\mathcal C:=\Res^1_0\cap\ker d$
is $b_1(\Sigma)$-dimensional and $\mathcal C_\psi$
is another $b_1(\Sigma)$-dimensional space characterized by the identity
$d\alpha\wedge\mathcal C_\psi=\psi\wedge\mathcal C$.

3. The semisimplicity
condition~\eqref{e:semi-simple} holds at $\lambda_0=0$
for the operators $P_{1,0}$ and~$P_{3,0}$.

4. $\Res^2_0=\mathbb Cd\alpha\oplus \mathbb C\psi\oplus d\mathcal C_\psi$
is $b_1(\Sigma)+2$-dimensional and consists of closed forms. The
map $\pi_2$ has kernel $\mathbb Cd\alpha\oplus d\mathcal C_\psi$ and range
$\mathbb C[\psi]_{H^2}$.

5. $\Res^{2,\infty}_0=\Res^{2,2}_0$ is $2b_1(\Sigma)+2$-dimensional.
The range of the map $\mathcal L_X:\Res^{2,2}_0\to \Res^2_0$ is equal
to $d\mathcal C_\psi$.

6. $\Res^3_0=d\alpha\wedge\Res^1_0$ is $2b_1(\Sigma)$-dimensional
and consists of closed forms. The map $\pi_3$ has kernel
$d\alpha\wedge\mathcal C$ and its range is a codimension 1 subspace
of $H^3(M;\mathbb C)$ not containing $[\pi_\Sigma^*d\vol_g]_{H^3}$.

7. The map $\pi_{\Sigma*}^{}$ annihilates $d\alpha \wedge \mathcal{C}$ and is an isomorphism from $d\alpha \wedge \mathcal{C}_\psi$ onto the space of harmonic $1$-forms on $\Sigma$.
\end{theo}
Theorem~\ref{t:hyperbolic} together with Lemma~\ref{l:0-forms}
and~\eqref{e:alternator} give part~1 of Theorem~\ref{thm:main}:
\begin{corr}
  \label{l:hyperbolic-corr}
Under the assumptions of Theorem~\ref{t:hyperbolic},
the algebraic multiplicities of 0 as a resonance of the operators $P_{k,0}$ are
\begin{equation}
  \label{e:hyperbolic-alg}
m_{0,0}(0)=m_{4,0}(0)=1,\quad
m_{1,0}(0)=m_{3,0}(0)=2b_1(\Sigma),\quad
m_{2,0}(0)=2b_1(\Sigma)+2
\end{equation}
and the order of vanishing of the Ruelle zeta function $\zeta_{\mathrm R}$ at~0 is equal to
$$
m_{\mathrm R}(0)=2m_{0,0}(0)-2m_{1,0}(0)+m_{2,0}(0)=4-2b_1(\Sigma).
$$
\end{corr}
Previously~\eqref{e:hyperbolic-alg} was proved in~\cite[Proposition 7.7]{dang-guillarmou-riviere-shen-20} using different methods. Here we give a more refined description: we construct the resonant forms, prove pairing formulas, and study the existence of Jordan blocks. We emphasize that these properties are of crucial importance for the perturbation arguments in \S \ref{s:general-perturb} and were not known prior to this work.

This section is structured as follows: in~\S\ref{s:hyp-3-basics} we review the geometric
features of hyperbolic 3-manifolds used here. In~\S\ref{s:hyp-psi} we construct the smooth
invariant 2-form $\psi$ and study its properties, proving part~1 of Theorem~\ref{t:hyperbolic}. In~\S\ref{s:hyp-1-forms} we study the resonant 1-forms and 3-forms, proving parts~2, 3, and~6 of Theorem~\ref{t:hyperbolic}. In~\S\ref{s:hyp-2-forms} we study the resonant
2-forms, proving parts~4 and~5 of Theorem~\ref{t:hyperbolic}. Finally, in~\S\ref{s:hyp-harmonic} we show that the pushforward operator $\pi_{\Sigma*}^{}$
maps elements of $\Res^3_0$ to harmonic 1-forms on~$(\Sigma,g)$, proving part 7 of Theorem~\ref{t:hyperbolic}.

\subsection{Hyperbolic 3-manifolds}
\label{s:hyp-3-basics}

We first review the geometry of hyperbolic 3-manifolds, following~\cite[\S3]{dyatlov-faure-guillarmou-15}.
We define a hyperbolic 3-manifold to be a nonempty compact connected oriented 3-dimensional Riemannian manifold~$\Sigma$
with constant sectional curvature~$-1$. Each such manifold can be written as a quotient
$$
\Sigma=\Gamma\backslash \mathbb H^3
$$
where $\mathbb H^3$ is the 3-dimensional hyperbolic space and $\Gamma\subset \SO_+(1,3)$ is a discrete torsion-free co-compact subgroup.
We will use the \emph{hyperboloid model}
$$
\begin{gathered}
\mathbb H^3=\{x\in\mathbb R^{1,3}\mid \langle x,x\rangle_{1,3}=1,\ x_0>0\}
\end{gathered}
$$
where $\mathbb R^{1,3}=\mathbb R^4$ is the Minkowski space, with points denoted by $x=(x_0,x_1,x_2,x_3)$ and the Lorentzian
inner product
$$
\langle x,x\rangle_{1,3}:=x_0^2-x_1^2-x_2^2-x_3^2.
$$
The group $\SO_+(1,3)$ is the group of linear transformations on $\mathbb R^{1,3}$ (that is, $4\times 4$ real matrices) which preserve the inner product $\langle\bullet,\bullet\rangle_{1,3}$, have determinant~1, and preserve the sign of $x_0$
on elements of $\mathbb H^3$.
The Riemannian metric on $\mathbb H^3$ is the restriction of $-\langle\bullet,\bullet\rangle_{1,3}$;
the group $\SO_+(1,3)$ acts on $\mathbb H^3$ by isometries, so the metric descends to the quotient $\Sigma$. Note that we may write $\mathbb{H}^3 \simeq \SO_+(1,3)/\SO(3)$ as a homogeneous space for the $\SO_+(1,3)$-action, since $\SO(3)$ is the stabilizer of the point $(1, 0, 0, 0) \in \mathbb{H}^3$.

\subsubsection{Geodesic flow}
We now study the geodesic flow on $\Sigma$, using the notation
of~\S\ref{s:geodesic-flows}. The sphere bundle $S\Sigma$ is the quotient
\begin{equation}
  \label{e:S-Sigma-Gamma}
S\Sigma=\Gamma\backslash S\mathbb H^3
\end{equation}
where the sphere bundle $S\mathbb H^3\subset \mathbb R^{1,3}\times \mathbb R^{1,3}$ has the form
$$
S\mathbb H^3=\{(x,v) \in \mathbb R^{1,3}\times \mathbb R^{1,3} \mid \langle x,x\rangle_{1,3}=1,\ \langle v,v\rangle_{1,3}=-1,\ \langle x,v\rangle_{1,3}=0\}.
$$
Note that we may write $S\mathbb{H}^3 \simeq \SO_+(1,3)/\SO(2)$ as a homogeneous space for the $\SO_+(1, 3)$-action, since $\SO(2)$ is the stabilizer of the point $(1, 0, 0, 0, 0, 1, 0, 0) \in S\mathbb{H}^3$. The contact form $\alpha$, defined in~\eqref{e:geodesic-alpha}, and the generator
$X$ of the geodesic flow are
\begin{equation}
  \label{e:X-hyperbolic}
\alpha=-\langle v,dx\rangle_{1,3},\quad
X=v\cdot \partial_x+x\cdot\partial_v
\end{equation}
where `$\cdot$' denotes the (positive definite) Euclidean inner product on $\mathbb R^{1,3}$.
The geodesic flow is then given by
$$
\varphi_t(x,v)=(x\cosh t+v\sinh t,x\sinh t+v\cosh t).
$$
As a corollary, the distance function on $\mathbb H^3$ with respect to the hyperbolic metric is given by
\begin{equation}
  \label{e:dist-h3}
\cosh d_{\mathbb H^3}(x,y)=\langle x,y\rangle_{1,3}\quad\text{for all}\quad x,y\in\mathbb H^3.
\end{equation}
The tangent space $T_{(x,v)}(S\mathbb H^3)$ consists of vectors $(\xi_x,\xi_v)\in \mathbb R^{1,3}\oplus \mathbb R^{1,3}$ such that
$$
\langle x,\xi_x\rangle_{1,3}=\langle v,\xi_v\rangle_{1,3}=\langle x,\xi_v\rangle_{1,3}+\langle v,\xi_x\rangle_{1,3}=0.
$$
The connection map \eqref{e:connection-map} is given by
$$
\mathcal K(x,v)(\xi_x,\xi_v)=\xi_v-\langle x,\xi_v\rangle_{1,3} \,x
=\xi_v+\langle v,\xi_x\rangle_{1,3} x.
$$
Here and throughout we note that the addition of points $x$ and vectors $\xi_v$ (or $\xi_x$) has to be understood in $\mathbb{R}^{1, 3}$. The horizontal and vertical spaces $\mathbf H(x,v),\mathbf V(x,v)\subset T_{(x,v)}(S\mathbb H^3)$ are then
$$
\begin{aligned}
\mathbf H(x,v)&=\{(\xi_x,\xi_v)\mid \langle x,\xi_x\rangle_{1,3}=0,\ \xi_v=-\langle v,\xi_x\rangle_{1,3}\, x\},\\
\mathbf V(x,v)&=\{(0,\xi_v)\mid\langle x,\xi_v\rangle_{1,3}=\langle v,\xi_v\rangle_{1,3}=0\}
\end{aligned}
$$
and the horizontal-vertical splitting map~\eqref{e:hv-map} takes for $\xi = (\xi_x,\xi_v) \in T_{(x, v)}(S\mathbb{H}^3) \subset \mathbb R^{1,3}\oplus \mathbb R^{1,3}$ the form
$$
\xi_H=\xi_x,\quad
\xi_V=\xi_v+\langle v,\xi_x\rangle_{1,3}\,x.
$$
The Sasaki metric \eqref{e:sasaki-metric} is for $\xi, \eta \in T_{(x, v)}(S\mathbb{H}^3)$ given by
$$
\langle \xi,\eta\rangle_S=-\langle\xi_x,\eta_x\rangle_{1,3}-\langle\xi_v,\eta_v\rangle_{1,3}+\langle v,\xi_x\rangle_{1,3}\langle v,\eta_x\rangle_{1,3}.
$$
The unstable/stable subspaces $E_u,E_s$ from~\eqref{e:Anosov-split} on $S\mathbb H^3$ are given by
\begin{equation}
  \label{e:stun-identify}
\begin{aligned}
E_u(x,v)&=\{(w,w)\mid w\in \mathbb R^{1,3},\ \langle w,x\rangle_{1,3}=\langle w,v\rangle_{1,3}=0\},\\
E_s(x,v)&=\{(w,-w)\mid w\in \mathbb R^{1,3},\ \langle w,x\rangle_{1,3}=\langle w,v\rangle_{1,3}=0\}.
\end{aligned}
\end{equation}
In terms of the horizontal-vertical splitting~\eqref{e:hv-map} they can be characterized as follows:
\begin{equation}
  \label{e:stun-hv}
E_u=\{\xi_V=\xi_H\},\quad
E_s=\{\xi_V=-\xi_H\}.
\end{equation}
A distinguished feature of hyperbolic manifolds is that the restriction of the differential of the geodesic flow
to the unstable/stable spaces is conformal with respect to the Sasaki metric:
\begin{equation}
  \label{e:sasaki-conformal}
|d\varphi_t(x,v)\xi|_S=\begin{cases} e^t|\xi|_S,& \xi\in E_u(x,v);\\
e^{-t}|\xi|_S,& \xi\in E_s(x,v).\end{cases}
\end{equation}
The objects discussed above are invariant under the action of $\SO_+(1,3)$ and thus
descend naturally to the quotients $\Sigma,S\Sigma$.

\subsubsection{The frame bundle and canonical vector fields}
\label{s:canonical-vf}

A convenient tool for computations on $M=S\Sigma$ is the \emph{frame bundle}
$\mathcal F\Sigma$, consisting of quadruples $(x,v_1,v_2,v_3)$ where
$x\in \Sigma$ and $v_1,v_2,v_3\in T_x\Sigma$ form a positively oriented orthonormal basis.
We have
$$
\mathcal F\Sigma=\Gamma\backslash \mathcal F\mathbb H^3,\quad
\mathcal F\mathbb H^3\simeq \SO_+(1,3)
$$
where the frame bundle $\mathcal F\mathbb H^3$ is identified with the group
$\SO_+(1,3)$ by the following map (where $e_0=(1,0,0,0),e_1=(0,1,0,0),\dots$)
\begin{equation}
  \label{e:frame-identifier}
\gamma\in \SO_+(1,3)\mapsto (\gamma e_0,\gamma e_1,\gamma e_2,\gamma e_3).
\end{equation}
Under this identification, the action of $\SO_+(1,3)$ on $\mathcal F\mathbb H^3$ corresponds
to the action of this group on itself by left multiplications. Therefore, $\SO_+(1,3)$-invariant 
vector fields on $\mathcal F\mathbb H^3$ correspond to left-invariant vector fields on the group $\SO_+(1,3)$,
that is to elements of its Lie algebra $\so(1,3)$. We define the basis of left-invariant vector fields on $\SO_+(1,3)$
corresponding to the following matrices in $\so(1,3)$:
\begin{align*}
X&=\begin{pmatrix}
0 & 1 & 0 & 0\\
1 & 0 & 0 & 0\\
0 & 0 & 0 & 0\\
0 & 0 & 0 & 0\\
\end{pmatrix},& R&=\begin{pmatrix}
0 & 0 & 0 & 0\\
0 & 0 & 0 & 0\\
0 & 0 & 0 & 1\\
0 & 0 & -1 & 0\\
\end{pmatrix},& U^+_1&=\begin{pmatrix}
0 & 0 & -1 & 0\\
0 & 0 & -1 & 0\\
-1 & 1 & 0 & 0\\
0 & 0 & 0 & 0\\
\end{pmatrix},
\end{align*}\begin{align*} 
U^+_2&=\begin{pmatrix}
0 & 0 & 0 & -1\\
0 & 0 & 0 & -1\\
0 & 0 & 0 & 0\\
-1 & 1 & 0 & 0\\
\end{pmatrix},& U^-_1&=\begin{pmatrix}
0 & 0 & -1 & 0\\
0 & 0 & 1 & 0\\
-1 & -1 & 0 & 0\\
0 & 0 & 0 & 0\\
\end{pmatrix},& U^-_2&=\begin{pmatrix}
0 & 0 & 0 & -1\\
0 & 0 & 0 & 1\\
0 & 0 & 0 & 0\\
-1 & -1 & 0 & 0\\
\end{pmatrix}.
\end{align*} 
Under the identification~\eqref{e:frame-identifier}, and considering $\mathcal F\mathbb H^3$
as a submanifold of~$(\mathbb R^{1,3})^4$, we can write using coordinates $(x, v_1, v_2, v_3) \in (\mathbb{R}^{1, 3})^4$ and writing `$\cdot$' for the Euclidean inner product
$$
\begin{aligned}
X=v_1\cdot\partial_x+x\cdot \partial_{v_1},&\quad
R=v_2\cdot\partial_{v_3}-v_3\cdot\partial_{v_2},\\
U_1^\pm=-v_2\cdot \partial_x-x\cdot \partial_{v_2}\pm(v_2\cdot \partial_{v_1}-v_1\cdot \partial_{v_2}),
&\quad
U_2^\pm=-v_3\cdot \partial_x-x\cdot \partial_{v_3}\pm(v_3\cdot \partial_{v_1}-v_1\cdot \partial_{v_3}).
\end{aligned}
$$
Since the vector fields above are invariant under the action of $\SO_+(1,3)$, they descend to the frame
bundle of the quotient, $\mathcal F\Sigma$.

The commutation relations between these fields are
(as can be seen by computing the commutators of the corresponding matrices, or by using the explicit formulas above)
\begin{equation}
\begin{alignedat}{6}\label{commutation_relations_XUR0}
[X,U_i^\pm]&=\pm U_i^\pm, &\qquad [U_i^+,U_i^-]&=2X, &\qquad [U_1^\pm,U_2^\mp]&=2R, \\
[X,R]&=[U_1^\pm,U_2^\pm]=0, &[R,U_1^\pm]&=-U_2^\pm, & [R,U_2^\pm]&=U_1^\pm.
\end{alignedat}
\end{equation}
The map
$$
\pi_{\mathcal F}:(x,v_1,v_2,v_3)\in\mathcal F\Sigma\mapsto (x,v_1)\in S\Sigma
$$
is a submersion, with one-dimensional fibers whose tangent spaces are spanned by the field~$R$.
Thus, if a vector field on $\mathcal F\Sigma$ commutes with $R$ then this vector field
descends to the sphere bundle $S\Sigma$.
In particular, the vector field~$X$
descends to the generator of the geodesic flow (which we also denote by~$X$).

The vector fields $U_i^\pm$ do not commute with $R$ and thus do not descend to $S\Sigma$.
However, the vector space $\Span(U_1^+,U_2^+)$ is $R$-invariant and descends to the
stable space $E_s$ on $S\Sigma$. Similarly, the space $\Span(U_1^-,U_2^-)$ descends to $E_u$.
Because of this we think of $U_1^+,U_2^+$ as \emph{stable vector fields}
and $U_1^-,U_2^-$ as \emph{unstable vector fields}.

\subsubsection{Canonical differential forms}
\label{s:canonical-forms}

We next introduce the
frame of \emph{canonical differential 1-forms} on $\mathcal F\Sigma$
$$
\alpha,\ R^*,\ U_1^{\pm*},\ U_2^{\pm*}
$$
which is defined as a dual frame for the vector fields $X,R,U_1^\mp,U_2^\mp$, in the sense compatible with the definition of the dual stable/unstable bundles \eqref{e:anosov-dual}, as follows:
\begin{equation}
  \label{e:forms-dual}
\langle \alpha,X\rangle=\langle R^*,R\rangle=\langle U_1^{\pm *},U_1^{\mp}\rangle=\langle U_2^{\pm *},U_2^\mp\rangle=1
\end{equation}
and all the other pairings between the 1-forms and the vector fields in question are equal to~0.
In particular, $\langle U_i^{\pm*},U_i^\pm\rangle=0$.

Using the following identity valid for any 1-form $\beta$ and any two vector fields $Y,Z$
\begin{equation}
  \label{e:d-form}
d\beta(Y,Z)=Y\beta(Z)-Z\beta(Y)-\beta([Y,Z]),
\end{equation}
the commutation relations~\eqref{commutation_relations_XUR0}, and the duality relations~\eqref{e:forms-dual}, we compute the differentials of the canonical forms:
\begin{equation}
  \label{e:canonical-diff}
\begin{aligned}
d\alpha=2(U_1^{+*}\wedge U_1^{-*}+U_2^{+*}\wedge U_2^{-*}),&\quad
dR^*=2(U_2^{-*}\wedge U_1^{+*}+U_2^{+*}\wedge U_1^{-*}),\\
dU^{\pm*}_1=\pm \alpha\wedge U^{\pm*}_1-R^*\wedge U^{\pm*}_2,&\quad
dU^{\pm*}_2=\pm \alpha\wedge U^{\pm*}_2+R^*\wedge U^{\pm*}_1.
\end{aligned}
\end{equation}
It follows that
\begin{equation}
  \label{e:canonical-Lx}
\mathcal L_X U^{\pm*}_j=\pm U^{\pm*}_j,\quad
\mathcal L_R U^{\pm*}_1=-U^{\pm*}_2,\quad
\mathcal L_R U^{\pm*}_2=U^{\pm*}_1.
\end{equation}
If $\omega$ is a differential form on $\mathcal F\Sigma$, then $\omega$
descends to $S\Sigma$ (i.e. it is a pullback by~$\pi_{\mathcal F}$ of a form on~$S\Sigma$) if and only
if $\iota_R\omega=0$, $\mathcal L_R\omega=0$.
In particular the form $\alpha$ on $\mathcal F\Sigma$ descends to the contact form
on $S\Sigma$, which we also denote by~$\alpha$.

\subsubsection{Conformal infinity}

Following~\cite[\S3.4]{dyatlov-faure-guillarmou-15} we consider the maps
\begin{equation}
  \label{e:B-pm}
\Phi_\pm:S\mathbb H^3\to (0,\infty),\quad
B_\pm:S\mathbb H^3\to \mathbb S^2,
\end{equation}
where $\mathbb S^2$ is the unit sphere in $\mathbb R^3$, defined by the identities
\begin{equation}
  \label{e:B-pm-def}
x\pm v=\Phi_\pm(x,v)(1,B_\pm(x,v))\quad\text{for all}\quad (x,v)\in S\mathbb H^3.
\end{equation}
Note that $B_\pm(x,v)$ is the limit as $t\to\pm\infty$ of the projection to $\mathbb H^3$ of the geodesic $\varphi_t(x,v)$
in the compactification of the Poincar\'e ball model of $\mathbb H^3$. Let
$$
(\mathbb S^2\times\mathbb S^2)_-:=
\{(\nu_-,\nu_+)\in\mathbb S^2\times \mathbb S^2\mid \nu_-\neq \nu_+\}.
$$
In fact, the maps $B_\pm$ yield the following diffeomorphism of $S\mathbb{H}^3$ (see~\cite[(3.24)]{dyatlov-faure-guillarmou-15}):
\begin{equation}
  \label{e:Xi-def}
\begin{gathered}
\Xi:S\mathbb H^3 \ni (y, v) \mapsto (\nu_-, \nu_+, t) \in (\mathbb S^2\times\mathbb S^2)_-\times\mathbb R\\
\text{with}\quad
\nu_\pm=B_\pm(y,v),\quad
t={1\over 2}\log\Big({\Phi_+(y,v)\over\Phi_-(y,v)}\Big).
\end{gathered}
\end{equation}
The geometric interpretation of $\Xi$ is as follows: $\nu_\pm$ are the limits on the conformal boundary~$\mathbb S^2$ of the geodesic $\varphi_s(y,v)$ as $s\to \pm\infty$ and $t$ is chosen so that $\varphi_{-t}(y,v)$ is the closest point to $e_0$ on that
geodesic (as can be seen from~\eqref{e:Xi-2} below and noting that $Xt=1$ by~\eqref{e:X-Phi-pm}).
 
We have the identity~\cite[(3.23)]{dyatlov-faure-guillarmou-15}
\begin{equation}
  \label{e:Phi-pm-B-pm-id}
\Phi_-(x,v)\Phi_+(x,v)\big|B_-(x,v)-B_+(x,v)\big|^2=4
\end{equation}
where $|\bullet|$ denotes the Euclidean distance on $\mathbb R^3\supset\mathbb S^2$.

We also introduce the Poisson kernel
\begin{equation}
  \label{e:Poisson-def}
P(x,\nu)=\big(\langle x,(1,\nu)\rangle_{1,3}\big)^{-1}>0,\quad
x\in\mathbb H^3,\quad
\nu\in\mathbb S^2\subset\mathbb R^3.
\end{equation}
The following relations hold~\cite[(3.21)]{dyatlov-faure-guillarmou-15}:
\begin{equation}\label{eq:Phi_pmPoisson}
	\Phi_\pm(x, v) = P(x, B_\pm(x,v)).
\end{equation}
If we fix $x\in\mathbb H^3$, then the maps $v\mapsto B_\pm(x,v)$ are diffeomorphisms
from the fiber $S_x\mathbb H^3$ onto~$\mathbb S^2$. The inverse maps are given by
$\nu\mapsto v_\pm(x,\nu)$ where~\cite[(3.20)]{dyatlov-faure-guillarmou-15}
\begin{equation}
\label{eq:xi_pmdef}
	v_\pm(x, \nu) = \mp x \pm P(x, \nu) (1, \nu) \in S_x \mathbb{H}^3,\quad
	B_\pm(x,v_\pm(x,\nu))=\nu.
\end{equation}
The diffeomorphisms $v \mapsto B_\pm(x, v)$ are conformal
with respect to the induced metric on~$S_x\mathbb H^3$ and the canonical metric
$|\bullet|_{\mathbb S^2}$: by~\cite[(3.22)]{dyatlov-faure-guillarmou-15}) we have
\begin{equation}\label{eq:jacobianxi_pm}
|\partial_v B_\pm(x,v) \eta|_{\mathbb S^2}={|\eta|_{g}\over \Phi_\pm(x,v)}
\quad\text{for all}\quad
\eta\in T_v(S_x\mathbb H^3).
\end{equation}

Next, we have
by~\eqref{e:X-hyperbolic} and~\eqref{e:stun-identify}
\begin{equation}
  \label{e:X-Phi-pm}
X\Phi_\pm=\pm\Phi_\pm,\quad
d\Phi_-|_{E_u}=d\Phi_+|_{E_s}=0.
\end{equation}
The maps $B_\pm$ are submersions with connected fibers, the tangent spaces to which
are described in terms of the stable/unstable decomposition~\eqref{e:Anosov-split} as follows:
for each $\nu\in\mathbb S^2$ 
\begin{equation}
  \label{e:B-pm-fibers}
T (B_+^{-1}(\nu))=(E_0\oplus E_s)|_{B_+^{-1}(\nu)},\quad
T (B_-^{-1}(\nu))=(E_0\oplus E_u)|_{B_-^{-1}(\nu)}.
\end{equation}
This can be checked using~\eqref{e:stun-identify}, see~\cite[(3.25)]{dyatlov-faure-guillarmou-15}.
The action of the differential $dB_+$ on~$E_u$, and of $dB_-$ on~$E_s$, can be described as follows:
for any $(x,v)\in S\mathbb H^3$ and $w\in \mathbb R^{1,3}$ such that $\langle x,w\rangle_{1,3}=\langle v,w\rangle_{1,3}=0$,
\begin{equation}
  \label{e:B-pm-diffs}
dB_\pm(x,v)(w,\pm w)={2(w'-w_0B_\pm(x,v))\over \Phi_\pm(x,v)}\quad\text{where}\quad w=(w_0,w').
\end{equation}

We next briefly discuss the action of the group $\SO_+(1,3)$ on the conformal infinity $\mathbb S^2$, referring to~\cite[\S3.5]{dyatlov-faure-guillarmou-15} for details. For any $\gamma\in\SO_+(1,3)$, define
$$
N_\gamma:\mathbb S^2\to (0,\infty),\quad
L_\gamma:\mathbb S^2\to \mathbb S^2
$$
by the identity (where on the left is the linear action of $\gamma$ on $(1,\nu)\in\mathbb R^{1,3}$)
$$
\gamma\cdot(1,\nu)=N_\gamma(\nu)(1,L_\gamma(\nu))\quad\text{for all}\quad \nu\in\mathbb S^2.
$$
The maps $L_\gamma$ define an action of $\SO_+(1,3)$ on $\mathbb S^2$. This action is transitive and the stabilizer of $e_1\in\mathbb S^2$ is the group of matrices $A\in\SO_+(1,3)$ such that
$A(1,1,0,0)^T=\tau (1,1,0,0)^T$ for some $\tau>0$, which may be shown to be isomorphic to the group of similarities of the plane $\mathrm{Sim}(2)$, giving $\mathbb{S}^2 \simeq \SO_+(1, 3)/\mathrm{Sim}(2)$ the structure of a homogeneous space.

This action is by orientation preserving conformal transformations, more precisely
\begin{equation}
  \label{e:gamma-s2-conformal}
|dL_\gamma(\nu)\zeta|_{\mathbb S^2}={|\zeta|_{\mathbb S^2}\over N_\gamma(\nu)}\quad\text{for all}\quad
(\nu,\zeta)\in T\mathbb S^2.
\end{equation}
Moreover, the maps $B_\pm$ have the equivariance property
\begin{equation}
  \label{e:B-pm-equivariant}
B_\pm(\gamma\cdot(x,v))=L_\gamma(B_\pm(x,v))\quad\text{for all}\quad
(x,v)\in S\mathbb H^3.
\end{equation}

We finally use the maps $B_\pm$ to describe a special class of differential forms on $S\Sigma$ defined as follows (c.f.\ \cite{kuster-weich, dyatlov-faure-guillarmou-15}):
\begin{defi}
\label{d:totally-stun}
We call a $k$-form $u\in\mathcal D'(S\Sigma;\Omega^k_0)$ \textbf{stable}
if it is a section of $\wedge^k E_s^*\subset \Omega^k_0$ where $E_s^*\subset T^*(S\Sigma)$ is the annihilator
of $E_0\oplus E_s$ (see~\eqref{e:anosov-dual}). We call $u$ \textbf{unstable}
if it is a section of $\wedge^k E_u^*$ where $E_u^*$ is the annihilator of $E_0\oplus E_u$.

We call a form $u$ \textbf{totally (un)stable} if both $u$ and $du$ are (un)stable.
\end{defi}
The lemma below (see also \cite[\S\S2.3--2.4]{kuster-weich}) shows that totally (un)stable $k$-forms on~$S\Sigma$, $\Sigma=\Gamma\backslash\mathbb H^3$,
correspond to $\Gamma$-invariant $k$-forms on $\mathbb S^2$. Denote by $\pi_\Gamma:S\mathbb H^3\to S\Sigma$ the
covering map.
\begin{lemm}
  \label{l:lifted-forms}
Let $u\in \mathcal D'(S\Sigma;\Omega^k_0)$ be totally stable. Then the lift $\pi_\Gamma^* u$ has the form
\begin{equation}
  \label{e:lifted-forms-1}
\pi_\Gamma^*u=B_+^* w\quad\text{where}\quad w\in \mathcal D'(\mathbb S^2;\Omega^k),\quad
L_\gamma^*w=w\quad\text{for all}\quad \gamma\in\Gamma.
\end{equation}
Conversely, each form $B_+^*w$, where $w$ satisfies~\eqref{e:lifted-forms-1}, is the lift
of a totally stable $k$-form on~$S\Sigma$.
A similar statement holds for totally unstable forms, with $B_+$ replaced by $B_-$.
\end{lemm}
\begin{proof}
We only consider the case of totally stable forms, with totally unstable forms handled similarly.
First of all, note that lifts of totally stable $k$-forms on $S\Sigma$ are exactly the $\Gamma$-invariant totally stable $k$-forms
on $S\mathbb H^3$.
Next, by~\eqref{e:B-pm-fibers}, a $k$-form $\zeta\in\mathcal D'(S\mathbb H^3;\Omega^k)$ is totally stable if and only if 
$\iota_Y \zeta=0$, $\mathcal L_Y \zeta=0$ for any vector field $Y$ tangent to the fibers
of the map~$B_+$, which is equivalent to saying that $\zeta=B_+^* w$
for some $w\in\mathcal D'(\mathbb S^2;\Omega^k)$. Finally, by~\eqref{e:B-pm-equivariant},
$\Gamma$-invariance of~$\zeta$ is equivalent to $\Gamma$-invariance of $w$.
\end{proof}
Lemma~\ref{l:lifted-forms} implies that
\begin{equation}
  \label{e:totally-wf}
\begin{aligned}
\text{every totally stable }u\in\mathcal D'(S\Sigma;\Omega^k_0)&\quad \text{lies in }\mathcal D'_{E_s^*}(S\Sigma;\Omega^k_0),\\
\text{every totally unstable }u\in\mathcal D'(S\Sigma;\Omega^k_0)&\quad \text{lies in }\mathcal D'_{E_u^*}(S\Sigma;\Omega^k_0).
\end{aligned}
\end{equation}
Indeed, assume that $u$ is totally stable. Write $\pi_{\Gamma}^*u = B_+^*w$ for some $w \in \mathcal{D}'(\mathbb{S}^2; \Omega^k)$, then we have $\WF(\pi_\Gamma^*u) = \pi_\Gamma^*\WF(u)$ (as $\pi_\Gamma$ is a local diffeomorphism).
From the behavior of wavefront sets under pullbacks~\cite[Theorem~8.2.4]{hoermander-03},
we know that $\WF(\pi_\Gamma^*u)$ is contained in the conormal bundle of the fibers
of the submersion~$B_+$. From~\eqref{e:B-pm-fibers} and~\eqref{e:anosov-dual}
we then have $\WF(u)\subset E_s^*$.
A similar argument works for the totally unstable case.

\subsection{Additional invariant 2-form}
  \label{s:hyp-psi}

The space of smooth flow invariant $2$-forms on $S\Sigma$ is known to be $2$-dimensional, see Lemma~\ref{l:inv2forms} below, \cite[Claim 3.3]{kanai-93} or \cite{hamenstaedt-95}, thus there exists a smooth invariant 2-form which
is not a multiple of $d\alpha$. In this section we introduce such a 2-form $\psi$ and study its properties; these are crucial for the study of Pollicott--Ruelle resonances at zero in~\S\S\ref{s:hyp-1-forms}--\ref{s:hyp-2-forms} below.

\subsubsection{A rotation on $E_u\oplus E_s$}
\label{s:hyp-psi-1}

Let $x\in\Sigma$. For any two $v,w\in T_x\Sigma$, we may define their \emph{cross product}
$v\times w\in T_x\Sigma$, which is uniquely determined by the following properties:
$v\times w$ is $g$-orthogonal to $v$ and $w$; the length of $v\times w$ is the area
of the parallelogram spanned by $v,w$ in $T_x\Sigma$; and $v,w,v\times w$
is a positively oriented basis of $T_x\Sigma$ whenever $v\times w\neq 0$.

For future use we record here an identity true for any $v,w_1,w_2,w_3,w_4\in T_x\Sigma$
such that $|v|_g=1$ and $w_1,w_2,w_3,w_4$ are $g$-orthogonal to~$v$:
\begin{equation}
  \label{e:cp-identity}
\langle v\times w_1,w_2\rangle_g\langle v\times w_3,w_4\rangle_g=
\langle w_1,w_3\rangle_g\langle w_2,w_4\rangle_g-\langle w_2,w_3\rangle_g\langle w_1,w_4\rangle_g.
\end{equation}
Using the horizontal/vertical decomposition~\eqref{e:hv-map},
we define the bundle homomorphism
\begin{equation}
  \label{e:I-def}
\mathcal I:TS\Sigma\to TS\Sigma,\quad
\mathcal I(x,v)(\xi_H,\xi_V)=(v\times \xi_V,v\times \xi_H).
\end{equation}
From~\eqref{e:X-hv} and~\eqref{e:stun-hv} we see that $\mathcal I$ preserves the flow/stable/unstable decomposition~\eqref{e:Anosov-split}.
Moreover, it annihilates $E_0=\mathbb RX$ and 
it is a rotation by $\pi/2$ on $E_u$ and on $E_s$ (with respect to the Sasaki metric), so in particular it satisfies $\mathcal{I}^2 = -\Id$ on $\ker \alpha = E_u \oplus E_s$;
however, the direction of the rotation is opposite on $E_u$ and on $E_s$ if we identify them by~\eqref{e:stun-identify}.

The map $\mathcal I$ is invariant under the geodesic flow $\varphi_t=e^{tX}$:
\begin{equation}
  \label{e:I-flow-invariant}
\mathcal L_X\mathcal I=0.
\end{equation}
This follows from the conformal property of the geodesic flow~\eqref{e:sasaki-conformal}
and the description of the action of $\mathcal I$ on $E_0,E_u,E_s$ in the previous paragraph.

For any point $(x,v_1,v_2,v_3)$ in the frame bundle $\mathcal F\Sigma$, we have
(using the horizontal/vertical decomposition)
\begin{equation}
  \label{e:rot-frame}
\mathcal I(x,v_1)(v_2,\pm v_2)=\pm (v_3,\pm v_3),\quad
\mathcal I(x,v_1)(v_3,\pm v_3)=\mp (v_2,\pm v_2).
\end{equation}
It follows that (see~\S\ref{s:canonical-vf} for the definition of the vector fields $U_i^\pm$ on $\mathcal F\Sigma$)
\begin{equation}
  \label{e:canonical-rot}
\begin{aligned}
\mathcal I(x,v_1)(d\pi_{\mathcal F}U^\pm_1(x,v_1,v_2,v_3))&=\mp d\pi_{\mathcal F}U^\pm_2(x,v_1,v_2,v_3),\\
\mathcal I(x,v_1)(d\pi_{\mathcal F}U^\pm_2(x,v_1,v_2,v_3))&=\pm d\pi_{\mathcal F}U^\pm_1(x,v_1,v_2,v_3).
\end{aligned}
\end{equation}

\subsubsection{Relation to conformal infinity}

The homomorphism $\mathcal I$ lifts to $TS\mathbb H^3$. If
$B_\pm:S\mathbb H^3\to \mathbb S^2$ are the maps defined in~\eqref{e:B-pm}
and `$\times$' denotes the cross product on $\mathbb R^3$, then for all $(x,v)\in S\mathbb H^3$
and $\xi\in T_{(x,v)}S\mathbb H^3$ we have
\begin{equation}
  \label{e:B-pm-I}
dB_\pm(x,v)(\mathcal I(x,v)\xi)=B_\pm(x,v)\times dB_\pm(x,v)(\xi).
\end{equation}
To see this, we use~\eqref{e:B-pm-fibers}, and the fact that $\mathcal I$ preserves
the flow/stable/unstable decomposition, to reduce to the case
$\xi=(w,\pm w)$, where $x,v,w$ is an orthonormal set in $\mathbb R^{1,3}$.
By the equivariance~\eqref{e:B-pm-equivariant} of $B_\pm$ under $\SO_+(1,3)$, the fact that
the action $L_\gamma$ of any $\gamma\in \SO_+(1,3)$ on $\mathbb S^2$ is by orientation preserving
conformal maps, and the equivariance of $\mathcal I$ under $\SO_+(1,3)$
we can reduce to the case $x=e_0,v=e_1,w=e_2$, where $e_0,e_1,e_2,e_3$ is the canonical
basis of $\mathbb R^{1,3}$. In the latter case~\eqref{e:B-pm-I} is verified directly
using~\eqref{e:B-pm-diffs} and~\eqref{e:rot-frame}.

Let $\star$ be the Hodge star operator on 1-forms on the round sphere $\mathbb S^2$. It may be expressed as follows: for any $w\in C^\infty(\mathbb S^2;\Omega^1)$ and $(\nu,\zeta)\in T\mathbb S^2$ we have
$$
\langle(\star w)(\nu),\zeta\rangle=-\langle w(\nu),\nu\times\zeta\rangle.
$$
From~\eqref{e:B-pm-I} we get the following relation of $\mathcal I$ to~$\star$:
for any 1-form $w$ on $\mathbb S^2$ we have
\begin{equation}
  \label{e:hodge-star}
(B_\pm^*w)\circ\mathcal I=-B_\pm^*(\star w)
\end{equation}
where for any 1-form $\beta$ on $S\mathbb H^3$
the 1-form $\beta\circ\mathcal I$ on $S\mathbb H^3$ is defined by
\begin{equation}
  \label{e:I-circ-def}
\langle (\beta\circ\mathcal I)(x,v),\xi\rangle=\langle\beta(x,v),\mathcal I(x,v)\xi\rangle.
\end{equation}

\subsubsection{The new invariant 2-form}
\label{s:hyp-psi-2}

We next define the 2-form $\psi\in C^\infty(S\Sigma;\Omega^2)$ as follows:
for all $(x,v)\in S\Sigma$ and $\xi,\eta\in T_{(x,v)}S\Sigma$,
\begin{equation}
  \label{e:psi-def}
\psi(x,v)(\xi,\eta)=d\alpha(x,v)(\mathcal I(x,v)\xi,\eta).
\end{equation}
To see that $\psi$ is indeed an antisymmetric form, we may use~\eqref{e:alpha-hv} and \eqref{e:I-def} to write
it in terms of the horizontal/vertical decomposition of $\xi,\eta$:
\begin{equation}
  \label{e:psi-hv}
\psi(x,v)(\xi,\eta)=\langle v\times\xi_H,\eta_H\rangle_g-\langle v\times\xi_V,\eta_V\rangle_g.
\end{equation}
Using~\eqref{e:canonical-diff}, \eqref{e:canonical-rot} we may also compute the lift of $\psi$ to the frame bundle $\mathcal F\Sigma$, which we still denote by $\psi$: 
\begin{equation}
  \label{e:canonical-psi}
\psi=2(U_1^{+*}\wedge U_2^{-*}+U_1^{-*}\wedge U_2^{+*}).
\end{equation}
We have
\begin{equation}
  \label{e:psi-flow-invariant}
\iota_X\psi=0,\quad
\mathcal L_X\psi=0.
\end{equation}
The first of these statements is checked directly using~\eqref{e:X-hv}. The second statement can be verified using~\eqref{e:canonical-Lx}
and~\eqref{e:canonical-psi}, or using that $\mathcal L_X \mathcal I=0$ and $\mathcal L_Xd\alpha=0$.

We now establish several properties of the form $\psi$. We will use the following corollaries
of~\eqref{e:alpha-hv}, \eqref{e:psi-hv}:
\begin{equation}
  \label{e:alpha-psi-zero}
d\alpha|_{\mathbf H\times\mathbf H}=0,\quad
d\alpha|_{\mathbf V\times\mathbf V}=0,\quad
\psi|_{\mathbf H\times\mathbf V}=0
\end{equation}
where the horizontal/vertical spaces $\mathbf H,\mathbf V$ are defined in~\S\ref{s:hv-spaces}.
\begin{lemm}
We have
\begin{align}
  \label{e:psi-basic-1}
d\psi&=0,\\
  \label{e:psi-basic-2}
\psi\wedge\psi&=d\alpha\wedge d\alpha,\\
  \label{e:psi-basic-3}
d(\alpha\wedge\psi)&=0.
\end{align}
\end{lemm}
\begin{proof}
By~\eqref{e:psi-flow-invariant} we have $\iota_X d\psi=0$.
Therefore, $d\psi(x,v)(\xi_1,\xi_2,\xi_3)=0$ for $\xi_1,\xi_2,\xi_3\in T_{(x,v)}S\Sigma$
such that one of these vectors lies in $E_0$. Next, $\mathcal L_X d\psi=0$, that is $d\psi$
is invariant under the geodesic flow.
Using this invariance for time $t\to\pm\infty$ together with~\eqref{e:sasaki-conformal} and the
fact that 3 is an odd number, we see that
$d\psi(x,v)(\xi_1,\xi_2,\xi_3)=0$ also when each of the vectors $\xi_1,\xi_2,\xi_3$
lies in either $E_u(x,v)$ or $E_s(x,v)$. It follows that~\eqref{e:psi-basic-1} holds.

To check~\eqref{e:psi-basic-2}, we first note that $\iota_X$ of both sides is zero.
Thus it suffices to check that
\begin{equation}
  \label{e:psib-int-1}
\psi\wedge\psi(x,v)(\xi_1,\xi_2,\xi_3,\xi_4)=d\alpha\wedge d\alpha(x,v)(\xi_1,\xi_2,\xi_3,\xi_4)
\end{equation}
for some choice of basis $\xi_1,\xi_2,\xi_3,\xi_4\in T_{(x,v)}S\Sigma$ of the kernel of~$\alpha$.
We take
$$
\xi_1=(w_1,0),\quad \xi_2=(w_2,0),\quad \xi_3=(0,w_3),\quad \xi_4=(0,w_4)
$$
under the horizontal/vertical decomposition~\eqref{e:hv-map}, where each $w_j\in T_x\Sigma$ is orthogonal to~$v$.
By~\eqref{e:alpha-hv}, \eqref{e:psi-hv}
$$
\begin{aligned}
\psi\wedge\psi(x,v)(\xi_1,\xi_2,\xi_3,\xi_4)&=-2\langle v\times w_1,w_2\rangle_g\langle v\times w_3,w_4\rangle_g,\\
d\alpha\wedge d\alpha(x,v)(\xi_1,\xi_2,\xi_3,\xi_4)&=2(\langle w_2,w_3\rangle_g\langle w_1,w_4\rangle_g
-\langle w_1,w_3\rangle_g\langle w_2,w_4\rangle_g)
\end{aligned}
$$
and~\eqref{e:psib-int-1} follows from~\eqref{e:cp-identity}.

Finally, to show~\eqref{e:psi-basic-3} it suffices to prove that $d\alpha\wedge \psi=0$. To show this we
may argue similarly to the proof of~\eqref{e:psi-basic-2} above, using~\eqref{e:alpha-psi-zero}.

Alternatively, \eqref{e:psi-basic-1}--\eqref{e:psi-basic-3}
can be checked by lifting to the frame bundle $\mathcal F\Sigma$ and using~\eqref{e:canonical-diff} and~\eqref{e:canonical-psi}.
\end{proof}
The next lemma studies the relation of $\psi$ to the de Rham cohomology of $M=S\Sigma$; in particular, its first item and \eqref{e:psi-flow-invariant} give the first item of Theorem \ref{t:hyperbolic}.
Recall the pullback and pushforward operators $\pi_\Sigma^*,\pi_{\Sigma*}^{}$ defined
in~\S\ref{s:de-rham-sphere} and denote by $d\vol_g$
the volume 3-form on $\Sigma$ induced by~$g$ and the choice of orientation.
\begin{lemm}
\label{l:psi-de-rham}
We have:

1. $\pi_{\Sigma*}^{}(\psi)=-4\pi$. In particular, $[\psi]_{H^2}\neq 0$.

2. $\pi_{\Sigma*}^{}(\alpha\wedge\psi)=0$.

3. $\pi_{\Sigma*}^{}(\alpha\wedge d\alpha)=0$.

4. $\alpha \wedge d\alpha \wedge d\alpha = 2 \psi \wedge \pi_\Sigma^*(d\vol_g)$.

5. $[\alpha \wedge \psi]_{H^3} = 2[\pi_\Sigma^*(d\vol_g)]_{H^3}$.
\end{lemm}
\begin{proof}
1. Let $(x,v)\in S\Sigma$ and $v_2,v_3$ be a positively oriented $g$-orthonormal basis
of the tangent space to the fiber $T_v(S_x\Sigma)$.
We consider $v_2,v_3$ as vertical vectors in $T_{(x,v)}S\Sigma$,
as well as vectors in $T_x\Sigma$.
The triple $v,v_2,v_3$
is a positively oriented $g$-orthonormal basis of $T_x\Sigma$, so by~\eqref{e:psi-hv}
$$
\psi(x,v)(v_2,v_3)=-\langle v\times v_2,v_3\rangle_g=-1.
$$
Thus the restriction of $\psi$ to each fiber $S_x\Sigma$ is $-1$ times the standard volume form
on $S_x\Sigma\simeq \mathbb S^2$, which implies that $\pi_{\Sigma*}^{}(\psi)=-4\pi$. It now follows from \eqref{e:pfm-1} that $[\psi]_{H^2} \neq 0$.

2. Fix $x\in\Sigma$, $v_1\in T_x\Sigma$. Let $v\in S_x\Sigma$ and $v_2,v_3$ be a positively oriented
$g$-orthonormal basis of the tangent space $T_v(S_x\Sigma)$ as in part~1 of this proof. Let
$\xi_1=(v_1,0)$ be the horizontal lift of $v_1$ to $T_{(x,v)}(S\Sigma)$.
By~\eqref{e:alpha-hv} and~\eqref{e:psi-hv} we compute
$$
\alpha\wedge \psi(x,v)(\xi_1,v_2,v_3)=-\langle v_1,v\rangle_g\langle v\times v_2,v_3\rangle_g=-\langle v_1,v\rangle_g.
$$
Since $v\mapsto \langle v_1,v\rangle_g$ is an odd function on $S_x\Sigma$, we have
$$
(\pi_{\Sigma*}^{}(\alpha\wedge\psi))(x)(v_1)=\int_{S_x\Sigma} -\langle v_1,v\rangle_g \,d\vol(v)=0.
$$

3. If $\xi_1,\xi_2,\xi_3\in T_{(x,v)}(S\Sigma)$ and $\xi_2,\xi_3$ are vertical, then by~\eqref{e:alpha-hv}
we have
$$
\alpha\wedge d\alpha(x,v)(\xi_1,\xi_2,\xi_3)=0
$$
which implies that $\pi_{\Sigma*}^{}(\alpha\wedge d\alpha)=0$.

4. Let $x\in\Sigma$ and $v,v_2,v_3$ be a positively oriented $g$-orthonormal basis of
$T_x\Sigma$. Let $\xi=X(x,v),\xi_2,\xi_3$ be the horizontal lifts of $v,v_2,v_3$
to $T_{(x,v)}S\Sigma$; we treat $v_2,v_3$ as vertical vectors in $T_{(x,v)}S\Sigma$.
Using~\eqref{e:alpha-hv} and~\eqref{e:psi-hv}, we compute
$$
\alpha\wedge d\alpha\wedge d\alpha(x,v)(\xi,\xi_2,\xi_3,v_2,v_3)=-2
=2\psi\wedge \pi_\Sigma^*(d\vol_g)(x,v)(\xi,\xi_2,\xi_3,v_2,v_3).
$$

5. Using the exact sequence~\eqref{e:exaseq-2} and the fact that $\pi_{\Sigma*}^{}(\alpha\wedge\psi)=0$,
we see that
$$
[\alpha\wedge\psi]_{H^3}=c[\pi_{\Sigma}^*(d\vol_g)]_{H^3}
$$
for some constant $c$. To determine $c$, note that $\alpha\wedge\psi\wedge\psi$ has the same integral over
$S\Sigma$ as $c\psi\wedge \pi_{\Sigma}^*(d\vol_g)$. Since $\alpha\wedge\psi\wedge\psi=\alpha\wedge d\alpha\wedge d\alpha
=2\psi\wedge\pi_\Sigma^*(d\vol_g)$,
we get $c=2$.
\end{proof}
We also have the following identity relating the operators $d\alpha\wedge$ and $\psi\wedge$
on 1-forms in $\Omega^1_0$:
\begin{lemm}
  \label{l:dalpha-psi-1}
For any 1-form $\beta$ on $S\Sigma$ such that $\iota_X\beta=0$, we have
\begin{equation}
  \label{e:dalpha-psi-1}
d\alpha\wedge\beta=\psi\wedge(\beta\circ \mathcal I)
\end{equation}
where the 1-form $\beta\circ\mathcal I$ is defined by~\eqref{e:I-circ-def}.
\end{lemm}
\begin{proof}
It is easy to see that $\iota_X$ of both sides of~\eqref{e:dalpha-psi-1} is equal to~0.
It is thus enough to check that
\begin{equation}
  \label{e:dalpha-psint-1}
d\alpha\wedge\beta(x,v)(\xi_1,\xi_2,\xi_3)=\psi\wedge (\beta\circ\mathcal I)(x,v)(\xi_1,\xi_2,\xi_3)
\end{equation}
for any three vectors $\xi_1,\xi_2,\xi_3$, each of which is either horizontal or vertical under the decomposition~\eqref{e:hv-map}.
Moreover, we may assume that the horizontal components of these vectors lie in the orthogonal complement $\{v\}^\perp$ to $v$ in $T_x\Sigma$.
It suffices to consider the following two cases:

\textbf{Case 1:} $\beta(x,v)(\xi)=\langle\xi_H,w_4\rangle_g$ for some $w_4\in\{v\}^\perp$. By~\eqref{e:I-def} and~\eqref{e:alpha-psi-zero}, both sides of~\eqref{e:dalpha-psint-1} are equal to~0 unless two of $\xi_1,\xi_2,\xi_3$ are horizontal and one is vertical; we write
$$
\xi_1=(w_1,0),\quad \xi_2=(w_2,0),\quad \xi_3=(0,w_3)
$$
where $w_j\in \{v\}^\perp$. We compute using~\eqref{e:alpha-hv}, \eqref{e:I-def}, and~\eqref{e:psi-hv}
$$
\begin{aligned}
d\alpha\wedge\beta(x,v)(\xi_1,\xi_2,\xi_3)&=\langle w_1,w_3\rangle_g\langle w_2,w_4\rangle_g-\langle w_2,w_3\rangle_g \langle w_1,w_4\rangle_g,\\
\psi\wedge(\beta\circ\mathcal I)(x,v)(\xi_1,\xi_2,\xi_3)&=\langle v\times w_1,w_2\rangle_g\langle v\times w_3,w_4\rangle_g
\end{aligned}
$$
and~\eqref{e:dalpha-psint-1} follows from~\eqref{e:cp-identity}.

\textbf{Case 2:} $\beta(x,v)(\xi)=\langle\xi_V,w_4\rangle_g$ for some $w_4\in \{v\}^\perp$. By~\eqref{e:I-def} and~\eqref{e:alpha-psi-zero}, both sides of~\eqref{e:dalpha-psint-1}
are equal to~0 unless two of $\xi_1,\xi_2,\xi_3$ are vertical and one is horizontal; we write
$$
\xi_1=(0,w_1),\quad
\xi_2=(0,w_2),\quad
\xi_3=(w_3,0)
$$
where $w_j\in \{v\}^\perp$. We compute using~\eqref{e:alpha-hv}, \eqref{e:I-def}, and~\eqref{e:psi-hv}
$$
\begin{aligned}
d\alpha\wedge\beta(x,v)(\xi_1,\xi_2,\xi_3)&=\langle w_2,w_3\rangle_g \langle w_1,w_4\rangle_g-\langle w_1,w_3\rangle_g\langle w_2,w_4\rangle_g,\\
\psi\wedge(\beta\circ\mathcal I)(x,v)(\xi_1,\xi_2,\xi_3)&=-\langle v\times w_1,w_2\rangle_g\langle v\times w_3,w_4\rangle_g
\end{aligned}
$$
and~\eqref{e:dalpha-psint-1} again follows from~\eqref{e:cp-identity}.

Alternatively, we may lift both sides of~\eqref{e:dalpha-psi-1} to the frame bundle $\mathcal F\Sigma$:
it suffices to consider the cases when $\beta$ is replaced by one of the forms $U^{\pm*}_i$,
in which case~\eqref{e:dalpha-psi-1} is checked by a direct calculation using~\eqref{e:canonical-diff}, \eqref{e:canonical-rot}, and~\eqref{e:canonical-psi}.
\end{proof}

\subsubsection{Characterization of all smooth flow-invariant 2-forms}

We finally give
\begin{lemm}
  \label{l:inv2forms}
Assume that $u\in C^\infty(S\Sigma;\Omega^2)$ satisfies $\mathcal L_Xu=0$.
Then $u$ is a linear combination of $d\alpha$ and~$\psi$.
\end{lemm}
\begin{proof}
Without loss of generality we assume that $u$ is real valued.
Since $d\alpha\wedge\psi=0$ and $\psi\wedge\psi=d\alpha\wedge d\alpha$ by~\eqref{e:psi-basic-2}--\eqref{e:psi-basic-3}, we may subtract from $u$ a linear combination of $d\alpha$ and~$\psi$ to make
\begin{equation}
  \label{e:inv2fi-0}
\int_M \alpha\wedge d\alpha\wedge u=\int_M \alpha\wedge\psi\wedge u=0.
\end{equation}
We will show that under the condition~\eqref{e:inv2fi-0} we have $u=0$.

Since $\alpha\wedge d\alpha\wedge u,\alpha\wedge \psi\wedge u,\alpha\wedge u\wedge u$ are smooth 5-forms on $S\Sigma$ invariant under the geodesic flow, by Lemma~\ref{l:0-forms} (we identify $\Omega^0$ and $\Omega^5$ via the volume form $d\vol_{\alpha}$) we have
\begin{equation}
  \label{e:inv2fi-1}
\alpha\wedge d\alpha\wedge u=\alpha\wedge \psi\wedge u=0,\quad
\alpha\wedge u\wedge u=c\,d\vol_\alpha
\end{equation}
for some constant $c\in\mathbb R$.

Next, $\iota_X u\in C^\infty(S\Sigma;\Omega^1_0)$ and $\mathcal L_X\iota_X u=0$, so by~\eqref{e:Anosov} (similarly to the last
step of the proof of Lemma~\ref{l:pairing-k-1}) we get $\iota_X u = 0$.
Also by~\eqref{e:Anosov} we obtain $u|_{E_u \times E_u} = 0$ and $u|_{E_s \times E_s} = 0$.
Therefore, it is enough to show that $u|_{E_s\times E_u}=0$.

Since $d\alpha$ is nondegenerate on~$E_s\times E_u$ (as follows for instance from~\eqref{e:alpha-hv} and~\eqref{e:stun-hv}), there exists unique smooth bundle homomorphism $A:E_s\to E_s$ such that
$$
u(x,v)(\xi,\eta)=d\alpha(A(x,v)\xi,\eta)\quad\text{for all}\quad
(x,v)\in S\Sigma,\
\xi\in E_s(x,v),\
\eta\in E_u(x,v).
$$
It remains to show that $A=0$.

Take any $(x,v)\in S\Sigma$, assume
that $v,w_1,w_2$ is a positively oriented orthonormal basis of~$T_x\Sigma$, and
define using the horizontal/vertical decomposition and~\eqref{e:stun-hv}
$$
\xi_j=(w_j,-w_j)\in E_s(x,v),\quad \eta_j=(w_j,w_j)\in E_u(x,v),\quad j=1,2.
$$
Applying~\eqref{e:inv2fi-1} to
the vectors $X(x,v),\xi_1,\xi_2,\eta_1,\eta_2$ and using~\eqref{e:alpha-hv}, \eqref{e:rot-frame}, and~\eqref{e:psi-def}, we get
\begin{equation}
  \label{e:inv2fi-2}
\tr A(x,v)=0,\quad
A(x,v)^T=A(x,v),\quad
\det A(x,v)=c
\end{equation}
where the transpose is with respect to the restriction of the Sasaki metric to $E_s(x,v)$.

If $c=0$, then~\eqref{e:inv2fi-2} implies that $A=0$. Assume that $c\neq 0$, then by~\eqref{e:inv2fi-2}
we have $c<0$ and $A$ has eigenvalues $\pm \sqrt{-c}$. The eigenspace of $A(x,v)$ corresponding to the eigenvalue $\sqrt{-c}$
is a one-dimensional subspace of $E_s(x,v)$ depending continuously on $(x,v)$. This is impossible since by restricting
to a single fiber $S_x\Sigma\subset S\Sigma$ and projecting $E_s$ onto the vertical space $\mathbf V$ we would obtain
a continuous one-dimensional subbundle of the tangent space to the 2-sphere.
\end{proof}

\subsection{Resonant 1-forms and 3-forms}
  \label{s:hyp-1-forms}

In this section we apply the properties of the $2$-form $\psi$ defined in~\eqref{e:psi-def}
to determine the precise structure of resonant $1$-forms on $M=S\Sigma$. Let us introduce some notation for (co-)resonant $1$-forms
(see~\eqref{e:I-circ-def} for the definition of $u\circ\mathcal I$)
\[
\mathcal{C}_{(*)} := \Res_{0(*)}^1 \cap \ker d, \quad \mathcal{C}_{\psi(*)} := \{u\circ\mathcal I\mid u\in \mathcal C_{(*)}\}
\]
where the subscript $(*)$ means we either suppress the star or we include it, respectively corresponding to resonances or co-resonances; we apply this convention to other notions appearing in this section. We remark that the use of subscript $\psi$ in $\mathcal{C}_\psi$ is motivated by the property $d\alpha\wedge\mathcal C_\psi=\psi\wedge\mathcal C$ demonstrated in \eqref{e:psi-wedge-def} below; in fact we initially used this relation as the definition of $\mathcal{C}_\psi$, before coming to the interpretation via the map $\mathcal{I}$.

Since $\mathcal I$ is invariant under the geodesic flow by~\eqref{e:I-flow-invariant} and annihilates $X$, we have
$$
\mathcal C_{\psi(*)}\subset \Res_{0(*)}^1.
$$
By Lemma~\ref{l:1-forms-ish} and~\eqref{e:bettiGysin} we have
\begin{equation}
  \label{e:c-psi-dim}
\dim \mathcal{C}_{(*)}=\dim\mathcal C_{\psi(*)} = b_1(\Sigma).
\end{equation}
We next show that \emph{all} resonant 1-forms lie in the direct sum $\mathcal C\oplus \mathcal C_\psi$.
This is done in Lemma~\ref{l:c-psi-2} below but first we need
\begin{lemm}
  \label{l:c-psi-1.5}
Assume that $u\in \Res^1_0$. Then $u$ is totally unstable in the sense of Definition~\ref{d:totally-stun}.
Similarly, if $u\in\Res^1_{0*}$, then $u$ is totally stable.
\end{lemm}
\Remark Lemma~\ref{l:c-psi-1.5} was previously proved by K\"uster--Weich~\cite[\S2.6]{kuster-weich}.
\begin{proof}
We consider the case $u\in\Res^1_0$, with the case $u\in\Res^1_{0*}$ handled in the same way.

We first show that $u$ is unstable in the sense of Definition~\ref{d:totally-stun}.
For that it is enough to prove that $u(Y)=0$ for any $Y\in C^\infty(M;E_0\oplus E_u)$.
Since $\iota_X u=0$, we may assume that $Y\in C^\infty(M;E_u)$. By the integral formula~\eqref{e:Rk-uhp}
for the Pollicott--Ruelle resolvent $R_{k,0}(\lambda)$, we have for $\Im\lambda\gg 1$ and any
$w\in C^\infty(M;\Omega^1_0)$, $\rho\in M$
$$
\langle R_{1,0}(\lambda)w,Y\rangle(\rho)=i\int_0^\infty e^{i\lambda t}\langle w(\varphi_{-t}(\rho)),d\varphi_{-t}(\rho)Y(\rho)\rangle\,dt.
$$
Since $Y$ is a section of the unstable bundle, by~\eqref{e:sasaki-conformal} we have
$|\langle w(\varphi_{-t}(\rho)),d\varphi_{-t}(\rho)Y(\rho)\rangle|\leq Ce^{-t}$ for some constant $C$ and all~$t\geq 0$,
$\rho\in M$.
Therefore, the integral above converges uniformly in~$\rho$ for $\Im\lambda>-1$, which implies that $\lambda\mapsto \langle R_{1,0}(\lambda)w,Y\rangle$ is holomorphic in $\Im\lambda>-1$. If $\Pi_{1,0}$ is the projector appearing in the Laurent expansion of~$R_{1,0}(\lambda)$
at $\lambda=0$, defined in~\eqref{e:merex}, then $\iota_Y\Pi_{1,0}=0$. Since $\Res^1_0$ is contained in the range
of~$\Pi_{1,0}$, we get $u(Y)=0$ as needed.

We now analyze $du$. First of all, $\iota_Xdu=0$ since $u\in \Res^1_0$. Next, we have $du|_{E_u\times E_u}=0$.
This can be seen by following the argument above, or using that $u(Y)=0$ for any $Y\in C^\infty(M;E_0\oplus E_u)$,
the identity~\eqref{e:d-form}, and the fact that the class $C^\infty(M;E_0\oplus E_u)$ is closed under Lie brackets
(as follows from~\eqref{e:B-pm-fibers}).

It remains to show that $du|_{E_u\times E_s}=0$.
Let $\zeta$ be the restriction of $du$ to $E_u\times E_s$,
considered as a section in $\mathcal D'_{E_u^*}(M;E_s^*\otimes E_u^*)$. 
(Here $E_s^*,E_u^*$ are dual to $E_u,E_s$ as in~\eqref{e:anosov-dual}.)
We endow
$E_s^*\otimes E_u^*$ with the inner product which is the tensor product of the dual Sasaski metrics
on $E_s^*$ and $E_u^*$. The operator
$$
P:=-i\mathcal L_X:C^\infty(M;E_s^*\otimes E_u^*)\to C^\infty(M;E_s^*\otimes E_u^*)
$$
is formally self-adjoint as follows from~\eqref{e:sasaki-conformal}, and $P\zeta=0$.
Then by~\cite[Lemma~2.3]{dyatlov-zworski-17} the section $\zeta$ is in~$C^\infty$.

Let us now consider $\zeta=du|_{E_u\times E_s}$ as a smooth 2-form on $M$
(i.e. $\iota_X \zeta=0$, $\zeta|_{E_u\times E_u}=\zeta|_{E_s\times E_s}=0$, and $\zeta|_{E_u\times E_s}=du|_{E_u\times E_s}$), then $\mathcal L_X\zeta=0$ and by Lemma~\ref{l:inv2forms} we see that $\zeta=a\,d\alpha+b\,\psi$ for some constants $a,b$.
We claim that $a=b=0$. This follows from~\eqref{e:psi-basic-2}--\eqref{e:psi-basic-3} and the identities
\begin{align}
  \label{e:roo-1}
\int_M \alpha\wedge d\alpha\wedge \zeta&=\int_M \alpha\wedge d\alpha \wedge du=0,\\
  \label{e:roo-2}
\int_M \alpha\wedge \psi\wedge \zeta&=\int_M \alpha\wedge\psi\wedge du=0.
\end{align}
Here the first identity in each line follows from the fact that $d\alpha|_{E_u\times E_u}=\psi|_{E_u\times E_u}=0$
(which can be verified using~\eqref{e:alpha-hv}, \eqref{e:stun-hv}, and~\eqref{e:psi-def}). More precisely, it suffices to observe that $\alpha \wedge d\alpha \wedge (du - \zeta)$ and $\alpha \wedge d\psi \wedge (du - \zeta)$ are pointwise zero, as $du - \zeta$ is supported on $E_s \times E_s$ by definition.
The second identity in each line follows by integration by parts and the fact that $d\alpha\wedge d\alpha\wedge u=d\alpha\wedge\psi\wedge u=0$
(as $\iota_X$ of both of these 5-forms is equal to~0). Now, $a=b=0$ implies that $\zeta=0$, that is $du|_{E_u\times E_s}=0$ as needed.
\end{proof}
We are now ready to prove
\begin{lemm}
  \label{l:c-psi-2}
We have $\mathcal C_{(*)} \cap  \mathcal C_{\psi(*)} = \{0\}$ and $\Res^1_{0(*)}=\mathcal C_{(*)}\oplus \mathcal C_{\psi(*)}$.
\end{lemm}
\begin{proof}
We consider the case of~$\Res^1_0$, with $\Res^1_{0*}$ handled similarly. We need to prove that each $u\in\Res^1_0$ can be expressed uniquely as a sum of elements in $\mathcal C$ and $\mathcal C_\psi$. By Lemma~\ref{l:c-psi-1.5}, $u$ is totally unstable.
By Lemma~\ref{l:lifted-forms}, the lift of $u$ to $S\mathbb H^3$ has the form
$$
\pi_\Gamma^*u=B_-^*w\quad\text{for some $\Gamma$-invariant}\quad w\in \mathcal D'(\mathbb S^2;\Omega^1),
$$
where $\Gamma\subset \SO_+(1,3)$ is the discrete subgroup such that $\Sigma=\Gamma\backslash\mathbb H^3$.
Take the Hodge decomposition of $w$:
\begin{equation}
  \label{e:cpsi2-dec}
w=w_1+\star w_2\quad\text{where}\quad w_1,w_2\in\mathcal D'(\mathbb S^2;\Omega^1),\quad
dw_1=dw_2=0.
\end{equation}
Since $\Gamma$ acts on $\mathbb S^2$ by orientation preserving conformal transformations $L_\gamma$ (see~\eqref{e:gamma-s2-conformal}),
its action commutes with the Hodge star $\star$. Since $H^1(\mathbb S^2)=0$, the Hodge decomposition above is unique,
which implies that $w_1,w_2$ are $\Gamma$-invariant. Applying Lemma~\ref{l:lifted-forms} again
and using~\eqref{e:totally-wf}, we see that
$$
B_-^*w_j=\pi_\Gamma^*u_j\quad\text{for some}\quad u_1,u_2\in \mathcal D'_{E_u^*}(M;\Omega^1_0).
$$
Since $dw_j=0$, we have $du_j=0$, which together with the fact that $\iota_X u_j=0$ shows that
$u_1,u_2\in \mathcal C$. Finally, by~\eqref{e:hodge-star} and~\eqref{e:cpsi2-dec}
we may express $u$ uniquely as
$$
u=u_1-u_2\circ\mathcal I,\quad
u_1\in\mathcal C,\quad
u_2\circ\mathcal I\in \mathcal C_\psi,
$$
finishing the proof.
\end{proof}
The next lemma establishes semisimplicity on resonant 1-forms:
\begin{lemm}
  \label{l:c-psi-3}
The semisimplicity condition~\eqref{e:semi-simple} holds at $\lambda_0=0$ for the operators~$P_{1,0}$ and~$P_{3,0}$.
\end{lemm}
\begin{proof}
By~\eqref{e:isom-1} it suffices to establish semisimplicity for~$P_{1,0}$.
By Lemma~\ref{l:ss-cores} it suffices to show that the pairing $\llangle\bullet,\bullet\rrangle$
on $\Res^1_0\times \Res^3_{0*}$ is nondegenerate. Recall from~\eqref{e:isom-1*} that $\Res^3_{0*}=d\alpha\wedge\Res^1_{0*}$.
By Lemma~\ref{l:pairing-k-1}
the pairing $\llangle \bullet,\bullet\rrangle$ is nondegenerate on $\mathcal C\times (d\alpha\wedge\mathcal C_*)$.
Therefore, it suffices to show the following diagonal structure of the pairing with respect to the decompositions
$\Res^1_{0(*)}=\mathcal C_{(*)}\oplus\mathcal C_{\psi(*)}$ established in Lemma~\ref{l:c-psi-2}:
\begin{align}
  \label{e:cpsint3-1}
\llangle u,d\alpha\wedge u_*\rrangle=0&\quad\text{for all}\quad u\in\mathcal C,\ u_*\in\mathcal C_{\psi *}\\
  \label{e:cpsint3-2}
\llangle u,d\alpha\wedge u_*\rrangle=0&\quad\text{for all}\quad u\in\mathcal C_{\psi},\ u_*\in\mathcal C_*\\
  \label{e:cpsint3-3}
\llangle u,d\alpha\wedge u_*\rrangle=-\llangle u\circ\mathcal I,d\alpha\wedge(u_*\circ\mathcal I)\rrangle&\quad\text{for all}\quad
u\in\mathcal C,\ u_*\in\mathcal C_*.
\end{align}
We first show~\eqref{e:cpsint3-1}. By Lemma~\ref{l:hodge} and~\eqref{e:Gysin-iso}
we may write
$$
\begin{aligned}
u=\pi_\Sigma^*w+df &\quad\text{for some}\quad w\in C^\infty(\Sigma;\Omega^1),\ dw=0,\ f\in\mathcal D'_{E_u^*}(M;\mathbb C),\\
u_*\circ\mathcal I=\pi_\Sigma^*w_*+df_* &\quad\text{for some}\quad w_*\in C^\infty(\Sigma;\Omega^1),\ dw_*=0,\ f_*\in\mathcal D'_{E_s^*}(M;\mathbb C).
\end{aligned}
$$
We now compute
$$
\begin{aligned}
\llangle u,d\alpha\wedge u_*\rrangle&=\llangle u,\psi\wedge (u_*\circ\mathcal I)\rrangle
=\int_M \alpha\wedge\psi\wedge (\pi_\Sigma^*w+df)\wedge (\pi_\Sigma^* w_*+df_*)\\
&=\int_M \alpha\wedge\psi\wedge \pi_\Sigma^*(w\wedge w_*)
=-\int_\Sigma \pi_{\Sigma*}^{}(\alpha\wedge\psi)\wedge w\wedge w_*=0.
\end{aligned}
$$
Here the first equality used Lemma~\ref{l:dalpha-psi-1}.
The third equality used integration by parts and~\eqref{e:psi-basic-3}.
The fourth equality used~\eqref{e:pushforward-integral} and~\eqref{e:pfm-2},
with the negative sign explained in the paragraph following~\eqref{e:pushforward-integral}.
The fifth equality used part~2 of Lemma~\ref{l:psi-de-rham}.
A similar argument proves~\eqref{e:cpsint3-2}.

Finally, to show~\eqref{e:cpsint3-3} we compute
$$
\llangle u,d\alpha\wedge u_*\rrangle
=\llangle u,\psi\wedge (u_*\circ\mathcal I)\rrangle
=\llangle \psi\wedge u,u_*\circ\mathcal I\rrangle
=-\llangle d\alpha\wedge (u\circ\mathcal I),u_*\circ\mathcal I\rrangle
$$
using Lemma~\ref{l:dalpha-psi-1} and the fact that $u\circ\mathcal I\circ\mathcal I=-u$.
\end{proof} 

We finally discuss the properties of the maps $\pi_{3(*)}:\Res^3_{0(*)}\to H^3(M;\mathbb C)$; as explained at the top of \S \ref{s:hyp-1-forms}, recall that the subscript $(*)$ denotes the corresponding resonance or co-resonance space, so we can include both in the discussion. Recall that all forms in $\Res^3_{0(*)}$ are closed by Lemma~\ref{l:res-3-closed}
and $\Res^3_{0(*)}=d\alpha\wedge\Res^1_{0(*)}$ by~\eqref{e:isom-1},
\eqref{e:isom-1*}.
Moreover, by Lemma~\ref{l:dalpha-psi-1} and the definition of~$\mathcal C_{\psi(*)}$
\begin{equation}\label{e:psi-wedge-def}
d\alpha\wedge \mathcal C_{\psi(*)}=\psi\wedge\mathcal C_{(*)}.
\end{equation}
We have
$\pi_{3(*)}(d\alpha\wedge\mathcal C_{(*)})=0$. Assume now that $u\in\mathcal C_\psi$, then $u\circ\mathcal I\in\mathcal C$, and by Lemma~\ref{l:hodge} and~\eqref{e:Gysin-iso}
we may write
$$
u\circ\mathcal I=\pi_\Sigma^* w+df\quad\text{for some}\quad
w\in C^\infty(\Sigma;\Omega^1),\
dw=0,\
f\in\mathcal D'_{E_u^*}(M;\mathbb C).
$$
Wedging with $\psi$, taking $\pi_{\Sigma*}^{}$, and using~\eqref{e:pfm-1}--\eqref{e:pfm-2}, part~1 of Lemma~\ref{l:psi-de-rham}, and Lemma~\ref{l:dalpha-psi-1} we get
$$
\pi_{\Sigma*}^{}\pi_3(d\alpha \wedge u) = \pi_{\Sigma*}^{} (\psi \wedge \pi_{\Sigma}^*w) = -4\pi w,
$$
which (together with a similar argument for coresonant states) immediately shows that
\begin{equation}
  \label{e:pi-3-first}
\pi_{\Sigma*}^{}\pi_{3(*)}:d\alpha\wedge \mathcal C_{\psi(*)}\to H^1(\Sigma;\mathbb C)\text{ is an isomorphism}.
\end{equation}
This implies that
\begin{equation}
  \label{e:pi-3-ker}
\ker\pi_{3(*)}=d\alpha\wedge \mathcal C_{(*)}
\end{equation}
and so by \eqref{e:exaseq-2} the range of $\pi_{3(*)}$ is a codimension 1 subspace of $H^3(M;\mathbb C)$
which does not contain $[\pi_\Sigma^*d\vol_g]_{H^3}$.

Summarizing the contents of \S \ref{s:hyp-1-forms}, we note that the second item of Theorem \ref{t:hyperbolic} follows from \eqref{e:psi-wedge-def}, Lemma \ref{l:c-psi-2},
Lemma~\ref{l:1-forms-ish}, and~\eqref{e:bettiGysin}, the third item by Lemma \ref{l:c-psi-3}, and the sixth item by the discussion in the preceding paragraph.

\subsection{Resonant 2-forms}
\label{s:hyp-2-forms}

We next study resonant 2-forms. We start with
\begin{lemm}
  \label{l:2-forms-1}
We have $d(\Res^2_{0(*)})=0$ and $\ker\pi_{2(*)}=\mathbb C d\alpha\oplus d\mathcal C_{\psi(*)}$
has dimension $b_1(\Sigma)+1$.
\end{lemm}
\begin{proof}
We consider the case of resonant 2-forms, with the case of coresonant 2-forms handled similarly.
We first show that $d(\Res^2_0)=0$, arguing similarly to the proof of Lemma~\ref{l:1-means-2}.
Take $\zeta\in \Res^2_0$, then by the definition \eqref{e:pi-k-def} of $\pi_3$ we have $d\zeta\in\ker\pi_3$.
Thus by~\eqref{e:pi-3-ker}, $d\zeta=d\alpha\wedge u$ for some $u\in\mathcal C$.
Take arbitrary $u_*\in\mathcal C_*$, then precisely as in ~\eqref{e:mclean}
$$
\llangle u,d\alpha\wedge u_*\rrangle=\int_M\alpha\wedge d\zeta\wedge u_*=\int_M d\alpha\wedge\zeta\wedge u_*=0.
$$
Now Lemma~\ref{l:pairing-k-1} implies that $u=0$ and thus $d\zeta=0$ as needed.

Next, if $u\in\mathcal C_\psi$, then using the same argument of integration by parts as in~\eqref{e:roo-1} yields
$$
\int_M \alpha\wedge d\alpha\wedge du=0.
$$
Therefore, $du$ cannot be a nonzero multiple of $d\alpha$, which means that $\mathbb Cd\alpha\cap d\mathcal C_\psi=\{0\}$.
We have $d\alpha\in \ker\pi_2$ and by Lemma~\ref{l:ker-pi-k} we have $d\mathcal C_\psi\subset\ker \pi_2$ as well.

It remains to show that $\ker\pi_2\subset \mathbb Cd\alpha\oplus d\mathcal C_\psi$. By Lemma~\ref{l:ker-pi-k},
$\ker\pi_2$ is contained in $d(\Res^{1,\infty})$. By~\eqref{e:res-0-split} and Lemmas~\ref{l:0-forms}, \ref{l:c-psi-2},
and~\ref{l:c-psi-3}, we have $\Res^{1,\infty}=\mathbb C\alpha\oplus \mathcal C\oplus\mathcal C_\psi$.
Then $d(\Res^{1,\infty})=\mathbb Cd\alpha\oplus d\mathcal C_\psi$, which finishes the proof.
\end{proof}
We next establish the following auxiliary result:
\begin{lemm}
  \label{l:2-forms-aux}
Assume that $\eta\in C^\infty(\Sigma;\Omega^2)$, $d\eta=0$, and $w\in \mathcal D'_{E_u^*}(M;\Omega^1)$
satisfy
\begin{equation}
  \label{e:2-forms-aux}
\iota_X(\pi_\Sigma^*\eta+dw)=0.
\end{equation}
Then $\eta$ is exact.
\end{lemm}
\Remark
The proof of Lemma~\ref{l:2-forms-aux} uses the 2-form $\psi$ which
is only available in the case of constant curvature. By contrast, Lemma~\ref{l:2-forms-aux}
is false if $\Res^1_{0*}$ consists of closed forms and $b_1(\Sigma)>0$; in fact the equation~\eqref{e:2-forms-aux}
then has a solution $w\in\mathcal D'_{E_u^*}(M;\Omega^1_0)$ for any closed~$\eta$. Indeed,
in this case $\llangle \iota_X\pi_\Sigma^*\eta,d\alpha\wedge u_*\rrangle=\int_M\pi_\Sigma^*\eta\wedge d\alpha\wedge u_*=0$
for any $u_*\in\Res^1_{0*}$ by integration by parts, and the existence of $w$ now follows from Lemma~\ref{l:solver}.
\begin{proof}
Put $\zeta:=\pi_\Sigma^*\eta+dw$, then $\iota_X\zeta=0$. Take arbitrary
closed $\eta_*\in C^\infty(\Sigma;\Omega^1)$ and put $u_*:=\pi_{1*}^{-1}([\pi_\Sigma^*\eta_*]_{H^1})\in\mathcal C_*$.
Then $u_*=\pi_\Sigma^*\eta_*+dw_*$ for some $w_*\in\mathcal D'_{E_s^*}(M;\mathbb C)$.
We compute
$$
\begin{aligned}
0&=\int_M \psi\wedge \zeta\wedge u_*=\int_M \psi\wedge\pi_\Sigma^*\eta \wedge\pi_\Sigma^* \eta_*\\
&=-\int_\Sigma (\pi_{\Sigma*}^{}\psi) \eta\wedge \eta_*=4\pi\int_\Sigma \eta\wedge \eta_*.
\end{aligned}
$$
Here the first equality follows since the 5-form under the integral lies in the kernel of $\iota_X$.
The second equality follows by integration by parts, using that $\psi,\eta,\eta_*$ are closed.
The third equality follows from~\eqref{e:pushforward-integral} and~\eqref{e:pfm-2}. The fourth
equality follows from part~1 of Lemma~\ref{l:psi-de-rham}.

We see that $\eta\wedge \eta_*$ integrates to~0 on $\Sigma$ for any closed smooth 1-form $\eta_*$. This implies that
$\eta$ is exact; indeed, we can reduce to the case when $\eta$ is harmonic and take $\eta_*$ to be the Hodge star of $\eta$ (we note that this final argument is just a form of Poincar\'e duality).
\end{proof}
We now describe the space of resonant 2-forms (recalling the convention $(*)$ at the top of~\S \ref{s:hyp-1-forms}):
\begin{lemm}
  \label{l:2-forms-2}
The range of $\pi_{2(*)}$ is equal to $\mathbb C[\psi]_{H^2}$, and
$\Res^2_{0(*)}=\mathbb C d\alpha\oplus \mathbb C\psi\oplus d\mathcal C_{\psi(*)}$.
In particular, $\dim\Res^2_{0(*)}=b_1(\Sigma)+2$.
\end{lemm}
\begin{proof}
We consider the case of resonant 2-forms, with the case of coresonant 2-forms handled similarly.
First of all, $\psi\in\Res^2_0$, thus $[\psi]_{H^2}=\pi_2(\psi)$ is in the range of $\pi_2$.
Next, by~\eqref{e:exaseq-1} and part~1 of Lemma~\ref{l:psi-de-rham} we have
$$
H^2(M;\mathbb C)=\pi_\Sigma^*H^2(\Sigma;\mathbb C)\oplus \mathbb C[\psi]_{H^2}.
$$
To show that the range of $\pi_2$ is equal to $\mathbb C[\psi]_{H^2}$, it remains
to prove that the intersection of this range with $\pi_\Sigma^*H^2(\Sigma;\mathbb C)$
is trivial. Take $u\in \Res^2_0$ and assume that $\pi_2(u)=[\pi_\Sigma^* \eta]_{H^2}$
for some $\eta\in C^\infty(\Sigma;\Omega^2)$, $d\eta=0$. Then $u=\pi_\Sigma^*\eta+dw$
for some $w\in\mathcal D'_{E_u^*}(M;\Omega^1)$. Since $\iota_Xu=0$, Lemma~\ref{l:2-forms-aux}
implies that $\eta$ is exact, that is $\pi_2(u)=0$ as needed.

Finally, the statement that $\Res^2_0=\mathbb Cd\alpha\oplus \mathbb C\psi\oplus d\mathcal C_\psi$
follows from the first statement of this lemma together with Lemma~\ref{l:2-forms-1}.
\end{proof}
The next lemma describes the space of generalized resonant
states $\Res^{2,2}_0$ (see~\eqref{e:res-k-l} and~\S\ref{s:omega-k-0}). It implies in particular
that the operator $P_{2,0}$ does not satisfy the semisimplicity condition~\eqref{e:semi-simple},
assuming that $b_1(\Sigma)>0$:
\begin{lemm}
\label{l:res-2-pairing}
1. The pairing $\llangle\bullet,\bullet\rrangle$ on $\Res^2_0\times \Res^2_{0*}$
has the following form in the decomposition of Lemma~\ref{l:2-forms-2}:
\begin{gather}
  \label{e:res2p-1}
\llangle d\alpha,d\alpha\rrangle=\llangle \psi,\psi\rrangle=\vol_\alpha(M)>0,\quad
\llangle d\alpha,\psi\rrangle=\llangle\psi,d\alpha\rrangle=0,\\
  \label{e:res2p-2}
\llangle \zeta,\zeta_*\rrangle=0\quad\text{for all}\quad \zeta\in d\mathcal C_\psi,\ \zeta_*\in \Res^2_{0*},\\
  \label{e:res2p-3}
\llangle \zeta,\zeta_*\rrangle=0\quad\text{for all}\quad \zeta\in \Res^2_0,\ \zeta_*\in d\mathcal C_{\psi*}.
\end{gather}

2. The range of the map
\begin{equation}
  \label{e:L-x-res2}
\mathcal L_X:\Res^{2,2}_{0(*)}\to \Res^2_{0(*)}
\end{equation}
is equal to $d\mathcal C_{\psi(*)}$. We have $\dim\Res^{2,2}_{0(*)}=2b_1(\Sigma)+2$.
\end{lemm}
\begin{proof}
1. The identities~\eqref{e:res2p-1} follow immediately from~\eqref{e:psi-basic-2} and~\eqref{e:psi-basic-3}.
We next show~\eqref{e:res2p-2}, with~\eqref{e:res2p-3} proved similarly. Let $\zeta=du$ where $u\in\mathcal C_\psi$.
We compute
$$
\llangle \zeta,\zeta_*\rrangle=\int_M d\alpha\wedge u\wedge \zeta_*=0
$$
Here in the first equality we integrate by parts and use that $d\zeta_*=0$ by Lemma~\ref{l:2-forms-1}.
The second equality follows from the fact that $\iota_X(d\alpha\wedge u\wedge\zeta_*)=0$.

\noindent 2. We consider generalized resonant states, with generalized coresonant states handled similarly.
First, assume that $\zeta\in\Res^{2,2}_0$, then $\mathcal L_X\zeta\in\Res^2_0$. Moreover,
since the transpose of $\mathcal L_X$ is equal to $-\mathcal L_X$ (see~\S\ref{s:transposes})
we compute
\begin{equation}
  \label{e:jambon}
\llangle \mathcal L_X\zeta,\zeta_*\rrangle=-\llangle \zeta,\mathcal L_X\zeta_*\rrangle=0\quad\text{for all}\quad
\zeta_*\in\Res^2_{0*}.
\end{equation}
Using this for $\zeta_*=d\alpha$ and $\zeta_*=\psi$ together with~\eqref{e:res2p-1}--\eqref{e:res2p-2}, we see that 
$\mathcal L_X\zeta\in d\mathcal C_\psi$. That is, the range of the map~\eqref{e:L-x-res2} is contained in $d\mathcal C_\psi$.

Now, take arbitrary $\eta\in d\mathcal C_\psi$. By~\eqref{e:res2p-2}, we have $\llangle\eta,\zeta_*\rrangle=0$
for all $\zeta_*\in\Res^2_{0*}$. Then by Lemma~\ref{l:solver} there exists $\zeta\in\mathcal D'_{E_u^*}(M;\Omega^2_0)$
such that $\mathcal L_X\zeta=\eta$. Since $\eta\in\Res^2_0$, we see that $\zeta\in \Res^{2,2}_0$.
This shows that the range of the map~\eqref{e:L-x-res2} contains $d\mathcal C_\psi$.

Finally, the equality $\dim\Res^{2,2}_{0}=2b_1(\Sigma)+2$ follows from Lemma~\ref{l:2-forms-2}
and the fact that the kernel of the map~\eqref{e:L-x-res2} is given by~$\Res^2_0$.
\end{proof}
We finally show that there are no higher degree Jordan blocks, completing the analysis
of the generalized resonant states of $P_{2,0}$ at~0:
\begin{lemm}
We have $\Res^{2,\infty}_{0(*)}=\Res^{2,2}_{0(*)}$.
\end{lemm}
\begin{proof}
We consider the case of generalized resonant states, with generalized coresonant
states handled similarly. It suffices to prove that $\Res^{2,3}_0\subset\Res^{2,2}_0$.
Take $\eta\in\Res^{2,3}_0$ and put $\zeta:=\mathcal L_X\eta\in\Res^{2,2}_0$.
Exactly as in ~\eqref{e:jambon}, the pairing of $\zeta$ with any element
of $\Res^2_{0*}$ is equal to~$0$. In particular
$$
\llangle \zeta,du_*\rrangle=0\quad\text{for all}\quad u_*\in\Res^1_{0*}.
$$
By part~2 of Lemma~\ref{l:res-2-pairing}, we have $\mathcal L_X\zeta=du$ for some~$u\in\mathcal C_\psi$. Put
$$
\omega:=d(\zeta+\alpha\wedge u)\in\mathcal D'_{E_u^*}(M;\Omega^3).
$$
Then $\iota_X\omega=\iota_Xd\zeta-du=0$. Since $\omega$ is exact we have $\mathcal{L}_X \omega = 0$ and moreover that $\omega\in\ker\pi_3\subset\Res^3_0$. By~\eqref{e:pi-3-ker},
we then have $\omega\in d\alpha\wedge\mathcal C$.

We now compute
$$
\begin{aligned}
0&=\llangle \zeta,du_*\rrangle=-\int_M\alpha\wedge d\zeta\wedge u_*
=\llangle u,d\alpha\wedge u_*\rrangle-\llangle \omega,u_*\rrangle.
\end{aligned}
$$
Here in the second equality we integrated by parts and used that the 5-form
$d\alpha\wedge\zeta\wedge u_*$ lies in the kernel of $\iota_X$ and thus equals 0.
Using the identities~\eqref{e:cpsint3-1}--\eqref{e:cpsint3-3} and Lemma~\ref{l:pairing-k-1}, recalling
that $u\in\mathcal C_\psi$, $\omega\in d\alpha\wedge\mathcal C$, and using that $u_*$
can be chosen as an arbitrary element of $\mathcal C_*$ or $\mathcal C_{\psi*}$,
we see that $u=0$ and $\omega=0$. Just using that $u = 0$ implies $\mathcal L_X^2\eta=\mathcal L_X\zeta=0$,
that is $\eta\in\Res^{2,2}_0$ as needed.
\end{proof}

\subsection{Relation to harmonic forms}
\label{s:hyp-harmonic}

In this section we show that pushforwards of
elements of $\Res^3_0=d\alpha\wedge (\mathcal C\oplus \mathcal C_\psi)$ to the
base $\Sigma$ are harmonic 1-forms. Recall that a $1$-form $h$ is called harmonic if $dh = 0$ and $d\star h = 0$, where $\star$ is the Hodge star on $(\Sigma,g)$. We will denote the set of such forms as $\mathcal{H}^1(\Sigma)$.
We start with the following identity:
\begin{lemm}
\label{l:hodgy-podgy}
Assume that $u\in\mathcal D'_{E_u^*}(M;\Omega^1_0)$ is unstable
in the sense of Definition~\ref{d:totally-stun}
and $\beta\in C^\infty(\Sigma;\Omega^1)$. Then
\begin{align}
\label{e:hodgy-podgy-1}
\psi\wedge u\wedge \pi_\Sigma^*(\star\beta)&=-\alpha\wedge d\alpha\wedge u\wedge\pi_\Sigma^*\beta,\\
\label{e:hodgy-podgy-2}
d\alpha\wedge u\wedge\pi_\Sigma^*(\star\beta)&=\alpha\wedge\psi\wedge u\wedge\pi_\Sigma^*\beta.
\end{align}
\end{lemm}
\begin{proof}
We first show~\eqref{e:hodgy-podgy-1}. Take arbitrary
$(x,v)\in M=S\Sigma$ and assume that $v,w_1,w_2$ is a positively oriented $g$-orthonormal basis
of~$T_x\Sigma$. It suffices to prove that
\begin{equation}
  \label{e:hodgy-podgy-int-1}
(\psi\wedge u\wedge \pi_\Sigma^*(\star\beta))(x,v)(X,\xi_1,\xi_2,\xi_3,\xi_4)=
-(\alpha\wedge d\alpha\wedge u\wedge\pi_\Sigma^*\beta)(x,v)(X,\xi_1,\xi_2,\xi_3,\xi_4)
\end{equation}
where we write in terms of the horizontal/vertical decomposition~\eqref{e:hv-map}
$$
X=(v,0),\quad
\xi_1=(w_1,0),\quad
\xi_2=(w_2,0),\quad
\xi_3=(0,w_1),\quad
\xi_4=(0,w_2).
$$
Using~\eqref{e:psi-hv}, \eqref{e:alpha-psi-zero}, the fact that $d\pi_\Sigma(x,v)(\xi_H,\xi_V)=\xi_H$, the condition $\iota_X u = 0$, and the identities
$$
(\star\beta)(x)(v,w_1)=\beta(x)(w_2),\quad
(\star\beta)(x)(v,w_2)=-\beta(x)(w_1)
$$
we see that the left-hand side of~\eqref{e:hodgy-podgy-int-1}
is equal to
$$
-u(x,v)(\xi_1)\beta(x)(w_1)-u(x,v)(\xi_2)\beta(x)(w_2).
$$
Using~\eqref{e:alpha-hv}, we next see that the right-hand side of~\eqref{e:hodgy-podgy-int-1} is equal to
$$
u(x,v)(\xi_3)\beta(x)(w_1)+u(x,v)(\xi_4)\beta(x)(w_2).
$$
It remains to note that by~\eqref{e:stun-hv} the vectors
$\xi_1+\xi_3$ and $\xi_2+\xi_4$ lie in $E_u(x,v)$ and thus
$u(x,v)(\xi_1+\xi_3)=u(x,v)(\xi_2+\xi_4)=0$ since $u$ is unstable.

The identity~\eqref{e:hodgy-podgy-2} is verified by a similar calculation,
or simply by applying~\eqref{e:hodgy-podgy-1} to $u\circ\mathcal I$
and using Lemma~\ref{l:dalpha-psi-1} and the fact that $u\circ\mathcal I\circ\mathcal I=-u$.
\end{proof}
We can now prove item 7 of Theorem \ref{t:hyperbolic}:
\begin{lemm}
\label{l:harmonic-push}
The map $\pi_{\Sigma*}^{}$ annihilates $d\alpha\wedge\mathcal C_{(*)}$ and it
is an isomorphism from $d\alpha\wedge \mathcal C_{\psi(*)}$
onto the space $\mathcal H^1(\Sigma)$. In particular,
by Lemma~\ref{l:c-psi-2} we have $\pi_{\Sigma*}^{}:d\alpha\wedge\Res^1_{0(*)}\to\mathcal H^1(\Sigma)$.
\end{lemm}
\begin{proof}
We consider the case of resonant 3-forms, with coresonant 3-forms handled similarly
(using a version of Lemma~\ref{l:hodgy-podgy} for stable 1-forms).
We first show that for any $u\in\mathcal C$,
the push-forwards to~$\Sigma$ of $d\alpha\wedge u$ and $\psi\wedge u$ are coclosed, that is
\begin{equation}
  \label{e:harmopu-1}
d\star \pi_{\Sigma*}^{}(d\alpha\wedge u)=0,\quad d\star\pi_{\Sigma*}^{}(\psi\wedge u)=0.
\end{equation}
To show the first equality in~\eqref{e:harmopu-1}, it suffices to prove that
$$
\int_\Sigma \pi_{\Sigma*}^{}(d\alpha\wedge u)\wedge \star df=0\quad\text{for all}\quad
f\in C^\infty(\Sigma;\mathbb C).
$$
Using~\eqref{e:pushforward-integral} and~\eqref{e:pfm-2}, we compute this integral
as
$$
-\int_M d\alpha\wedge u\wedge \pi_\Sigma^*(\star df)=-\int_M \alpha\wedge\psi\wedge u\wedge d(\pi_\Sigma^*f)=\int_M \pi_\Sigma^*f d\alpha\wedge\psi\wedge u=0
$$
Here in the first equality we used~\eqref{e:hodgy-podgy-2},
where $u$ is unstable by Lemma~\ref{l:c-psi-1.5}.
In the second equality we integrated by parts and used that
$d\psi=0$ and $du=0$. In the third equality we used that
$\iota_X$ of the 5-form under the integral is equal to~0.
The second equality in~\eqref{e:harmopu-1} is proved similarly,
using~\eqref{e:hodgy-podgy-1} instead of~\eqref{e:hodgy-podgy-2}.

Next, by~\eqref{e:pfm-1}, since all forms in $d\alpha\wedge\mathcal C$
are exact, their pushforwards to~$\Sigma$ are exact as well. Since these pushforwards
are also coclosed, we get $\pi_{\Sigma*}^{}(d\alpha\wedge\mathcal C)=0$.
Similarly, all forms in $d\alpha\wedge\mathcal C_\psi=\psi\wedge\mathcal C$
are closed, so their pushforwards are closed as well; since these pushforwards
are also coclosed, we get $\pi_{\Sigma*}^{}(d\alpha\wedge\mathcal C_\psi)\subset\mathcal H^1(\Sigma)$.

Finally, by~\eqref{e:pi-3-first} we see that $\pi_{\Sigma*}^{}$ is an isomorphism
from $d\alpha\wedge\mathcal C_\psi$ onto $\mathcal H^1(\Sigma)$.
\end{proof}
We finally remark that for any 1-form $u\in\mathcal D'(M;\Omega^1)$ we have
\begin{equation}
  \label{e:pushforward-nm}
\pi_{\Sigma*}^{}(\alpha\wedge u)=0.
\end{equation}
Indeed, by~\eqref{e:alpha-hv} we see that $\alpha$, and thus $\alpha\wedge u$,
vanish when restricted to the tangent spaces of the fibers $S_x\Sigma$.
From~\eqref{e:pushforward-nm} and~\eqref{e:pfm-1} we get
for any $u\in\mathcal D'(M;\Omega^1)$
\begin{equation}
  \label{e:pushforward-nm-2}
\pi_{\Sigma*}^{}(d\alpha\wedge u)=\pi_{\Sigma*}^{}(\alpha\wedge du).
\end{equation}

\section{Contact perturbations of geodesic flows on hyperbolic 3-manifolds}
\label{s:general-perturb}

Let $M=S\Sigma$ where $(\Sigma,g)$ is a hyperbolic 3-manifold
and $\alpha_0$ be the contact form on~$M$ corresponding to the geodesic
flow on $\Sigma$, see~\S\S\ref{s:geodesic-flows},\ref{s:hyp-3-basics}. In this section we study Pollicott--Ruelle resonances at~$\lambda=0$ for perturbations of $\alpha_0$. Ultimately, we will study perturbations of the metric, but via perturbations of the contact form. In particular, we give the proof of Theorem~\ref{thm:main} in~\S\ref{s:main-proof} below,
relying on Theorem~\ref{t:main-identity} (in~\S\ref{s:main-identity})
and Proposition~\ref{prop:harmoniconeforms} proved later.

Let
$$
\alpha_\tau\in C^\infty(M;T^*M),\quad
\tau\in(-\varepsilon,\varepsilon)
$$
be a family of 1-forms depending smoothly on $\tau$.
We may shrink $\varepsilon>0$ so that each $\alpha_\tau$ is a contact form on $M$
and the corresponding Reeb vector field
$$
X_\tau\in C^\infty(M;TM)
$$
is Anosov; the latter
follows from stability of the Anosov condition under perturbations
(see for instance~\cite[Corollary 5.1.12]{Fisher-Hasselblatt-19} or~\cite[Corollary~6.4.7]{Katok-Hasselblatt}
for the related case of Anosov diffeomorphisms).

We will use first variation methods, introducing the 1-form
$$
\beta:=\partial_\tau\alpha_\tau|_{\tau=0}\in C^\infty(M;\Omega^1).
$$

We use the subscript or superscript $(\tau)$ to refer to the objects
corresponding to the contact manifold $(M,\alpha_\tau)$ and the flow
$\varphi_t^{(\tau)}:=e^{tX_\tau}$. For example,
we use the operators (see~\S\ref{section:PRresonances})
$$
P_k^{(\tau)}=-i\mathcal L_{X_\tau},\quad P_{k,0}^{(\tau)},\quad
R_k^{(\tau)}(\lambda),\quad
\Pi_k^{(\tau)}:=\Pi_k^{(\tau)}(0),
$$
the spaces of (generalized) resonant states at~$\lambda=0$
$$
\Res^{k,\ell}_{(\tau)},\quad
\Res^{k,\ell}_{0(\tau)},\quad
\Res^{k}_{(\tau)},\quad
\Res^{k}_{0(\tau)},
$$
and the algebraic multiplicities of~0 as a resonance of the operators $P_k^{(\tau)}$,
$P_{k,0}^{(\tau)}$
$$
m_k^{(\tau)}(0),\quad
m_{k,0}^{(\tau)}(0).
$$
When we omit $\tau$ this means that we are considering the unperturbed hyperbolic
case $\tau=0$, that is
\begin{equation}
  \label{e:tau-0}
\alpha:=\alpha_0,\quad
P_k:=P_k^{(0)},\quad
R_k:=R_k^{(0)},\quad
\Res^{k,\ell}:=\Res^{k,\ell}_{(0)},\quad
\Pi_k:=\Pi_k^{(0)},\dots
\end{equation}
The first result of this section, proved in~\S\ref{s:general-perturber} below, is the following
theorem. (Here the maps $\pi_k^{(\tau)}:\Res^{k}_{0(\tau)}\cap\ker d\to H^k(M;\mathbb C)$
are defined in~\eqref{e:pi-k-def}.)
\begin{theo}
  \label{t:general-perturber}
Let the assumptions above in this section hold. Assume moreover
the following nondegeneracy condition:
\begin{equation}
\label{e:nondegeneracy-assumption}
\llangle\iota_X\beta\,\bullet,\bullet\rrangle\quad\text{defines a nondegenerate pairing on}\quad
d(\Res^1_0)\times d(\Res^1_{0*}).
\end{equation}
Then there exists $\varepsilon_0>0$ such that for all $\tau$ with $0<|\tau|<\varepsilon_0$ we have:

1. $d(\Res^{1}_{0(\tau)})=0$ and thus by Lemma~\ref{l:1-forms-ish} and~\eqref{e:bettiGysin} we have $\dim\Res^{1}_{0(\tau)}=b_1(\Sigma)$.

2. $d(\Res^{2}_{0(\tau)})=0$, $\dim\Res^{2}_{0(\tau)}=b_1(\Sigma)+2$, and
the map $\pi_2^{(\tau)}$ is onto and has kernel $\mathbb C d\alpha_\tau$.

3. $d(\Res^{3}_{0(\tau)})=0$ and the map $\pi_3^{(\tau)}$ is equal to~0.

4. The semisimplicity condition~\eqref{e:semi-simple} holds at $\lambda_0=0$ for the operators
$P_{k,0}^{(\tau)}$ for all $k=0,1,2,3,4$.
\end{theo}
Theorem~\ref{t:general-perturber} together with Lemma~\ref{l:0-forms}
and~\eqref{e:alternator} give the following
\begin{corr}
  \label{l:general-perturber-corr}
Under the assumptions of Theorem~\ref{t:general-perturber} we have for $0<|\tau|<\varepsilon_0$
$$
m_{0,0}^{(\tau)}(0)=m_{4,0}^{(\tau)}(0)=1,\quad
m_{1,0}^{(\tau)}(0)=m_{3,0}^{(\tau)}(0)=b_1(\Sigma),\quad
m_{2,0}^{(\tau)}(0)=b_1(\Sigma)+2
$$
and the order of vanishing of the Ruelle zeta function $\zeta_{\mathrm R}$ at~0 is
$$
m_{\mathrm R}(0)=2m_{0,0}^{(\tau)}(0)-2m_{1,0}^{(\tau)}(0)+m_{2,0}^{(\tau)}(0)=4-b_1(\Sigma).
$$
\end{corr}
Corollary~\ref{l:general-perturber-corr}
is in contrast with the hyperbolic case $\tau=0$, where
Corollary~\ref{l:hyperbolic-corr} gives the
order of vanishing~$4-2b_1(\Sigma)$.

To give an application of Theorem~\ref{t:general-perturber}
which is simpler to prove than Theorem~\ref{thm:main},
we show in~\S\S\ref{s:pert-supp}--\ref{s:pert-endgame} below
that the nondegeneracy condition~\eqref{e:nondegeneracy-assumption}
holds for a large set of conformal perturbations of the contact form $\alpha$:%
\footnote{By the Gray Stability Theorem (see \cite[Theorem 2.2.2]{Geiges-08}), any perturbation of a contact
form is a conformal perturbation up to pullback by a diffeomorphism.}
\begin{theo}
  \label{t:pertnondeg-contact}
Let $M=S\Sigma$ where $(\Sigma,g)$ is a hyperbolic
3-manifold. Fix a nonempty open set $\mathscr U\subset M$,
and denote by $\CIc(\mathscr U;\mathbb R)$ the
space of all smooth real-valued functions
on $M$ with support inside $\mathscr U$, with the topology inherited
from $C^\infty(M;\mathbb R)$.

Then there exists an open dense subset of $\CIc(\mathscr U;\mathbb R)$ such that for any $\mathbf a$ in this set, the 1-form $\beta:=\mathbf a\alpha$
satisfies the condition~\eqref{e:nondegeneracy-assumption}.
It follows that for $\tau\neq 0$ small enough depending on $\mathbf a$
the contact flow on $M$ corresponding to the contact form $\alpha_\tau:=e^{\tau \mathbf a}\alpha$
satisfies the conclusions of Theorem~\ref{t:general-perturber},
in particular the Ruelle zeta function has order of vanishing $4-b_1(\Sigma)$
at~0.
\end{theo}

\subsection{Proof of Theorem~\ref{t:general-perturber}}
\label{s:general-perturber}

We first prove an identity relating the action of the vector field
\begin{equation}
  \label{e:Y-def}
Y:=\partial_\tau X_\tau|_{\tau=0}\in C^\infty(M;TM)
\end{equation}
on resonant and coresonant 1-forms to the bilinear form featured in~\eqref{e:nondegeneracy-assumption}. It reformulates the pairing \eqref{e:nondegeneracy-assumption} and will subsequently (see Lemma \ref{l:perturb-1-forms}) be used to show that the non-closed $1$-forms may be perturbed away.
\begin{lemm}
  \label{l:derivative-computed}
For all $u\in\Res^1_0$ and $u_*\in\Res^1_{0*}$, we have
\begin{equation}
  \label{e:derivative-computed}
\llangle{\Pi_1 \mathcal{L}_Y \Pi_1 u, d\alpha \wedge u_*}\rrangle = \llangle{\mathcal{L}_Y u, d\alpha \wedge u_*}\rrangle = \llangle (\iota_X\beta)du,du_*\rrangle.
\end{equation}
\end{lemm}
\begin{proof}
1. To show the first equality in~\eqref{e:derivative-computed},
we note that by the decomposition~\eqref{e:Pi-k-split} and Lemma~\ref{l:0-forms}
we have for all $w\in\mathcal D'_{E_u^*}(M;\Omega^1)$
$$
\Pi_1w=\Pi_{1,0}(w-(\iota_X w)\alpha)+{1\over\vol_\alpha(M)}\bigg(\int_M \iota_X w\,d\vol_\alpha\bigg)\alpha.
$$
We now compute
$$
\begin{aligned}
\int_M \alpha\wedge d\alpha\wedge (\Pi_1\mathcal L_Y\Pi_1u)\wedge u_*
&=\llangle \Pi_{1,0}(\mathcal L_Y u-(\iota_X \mathcal L_Y u)\alpha),d\alpha\wedge u_*\rrangle
\\&=\llangle \mathcal L_Y u-(\iota_X\mathcal L_Yu)\alpha,d\alpha\wedge u_*\rrangle
\\&=\int_M\alpha\wedge d\alpha\wedge \mathcal L_Y u\wedge u_*.
\end{aligned}
$$
Here in the first equality we used that $u\in\Res^1_0$ and thus $\Pi_1 u=u$.
In the second equality we used that $d\alpha\wedge u_*\in\Res^3_{0*}$ and thus
$(\Pi_{1,0})^T (d\alpha\wedge u_*)=d\alpha\wedge u_*$ (see~\S\ref{s:transposes}). This proves
the first equality in~\eqref{e:derivative-computed}.

\noindent 2. We now show the second equality in~\eqref{e:derivative-computed}. Differentiating the relations $\iota_{X_\tau} \alpha_\tau = 1$ and $\iota_{X_\tau} d\alpha_{\tau} = 0$ (see~\eqref{e:reeb-field}) at $\tau=0$, we get
\begin{equation}\label{e:derivativeatzero}
	\iota_Y \alpha = - \iota_X \beta, \quad \iota_Y d\alpha = -\iota_X d\beta.
\end{equation}
Note also that
\begin{equation}
  \label{e:dcom-0}
\alpha\wedge d\alpha\wedge du=\alpha\wedge d\alpha\wedge du_*=0
\end{equation}
as follows from Lemma~\ref{l:0-forms} as the 5-forms above are in $\Res^0_0d\vol_\alpha$, respectively $\Res^0_{0*}d\vol_\alpha$, and integrate to~0 on~$M$ using integration by parts
(since the 5-forms $d\alpha\wedge d\alpha\wedge u$, $d\alpha\wedge d\alpha\wedge u_*$
lie in the kernel of $\iota_X$ and thus are equal to~0).

We have
\begin{equation}
  \label{e:dcom-1}
\int_M\alpha\wedge d\alpha\wedge \mathcal L_Y u\wedge u_*=
\int_M \alpha\wedge d\alpha\wedge \iota_Y du\wedge u_*+
\int_M \alpha\wedge d\alpha\wedge d\iota_Y u\wedge u_*.
\end{equation}
We first compute
\begin{equation}
  \label{e:dcom-2}
\begin{aligned}
\int_M \alpha\wedge d\alpha\wedge\iota_Y du\wedge u_*&=
-\int_M \alpha\wedge \iota_Y d\alpha\wedge du\wedge u_*-\int_M(\iota_Y u_*)\alpha\wedge d\alpha\wedge du\\
&=\int_M \alpha\wedge\iota_X d\beta\wedge du\wedge u_*
=\int_M d\beta\wedge du\wedge u_*\\
&=\int_M \beta\wedge du\wedge du_*=\int_M (\iota_X\beta)\alpha\wedge du\wedge du_*.
\end{aligned}
\end{equation}
Here in the first equality we used that the 5-form
$d\alpha\wedge du\wedge u_*$ lies in the kernel of $\iota_X$ and is thus equal to~0,
implying
$\iota_Y(d\alpha\wedge du\wedge u_*)=0$.
In the second equality we used the identities~\eqref{e:derivativeatzero} and~\eqref{e:dcom-0}.
In the third equality we used that $\alpha\wedge \iota_X d\beta\wedge du\wedge u_* = d\beta \wedge du \wedge u_*$ as the difference of the two forms belongs to $\ker \iota_X$, by $\iota_Xdu=0$ and $\iota_X u_*=0$. In the fourth equality we integrated by parts,
and in the fifth equality we used that $\iota_X$ of the integrated 5-forms are equal.

We next compute
\begin{equation}
  \label{e:dcom-3}
\int_M \alpha\wedge d\alpha\wedge d\iota_Y u\wedge u_*=\int_M \iota_Y u(d\alpha\wedge d\alpha\wedge u_*-\alpha\wedge d\alpha\wedge du_*)=0.
\end{equation}
Here in the first equality we integrated by parts and in the second one we used~\eqref{e:dcom-0} and the fact that $d\alpha\wedge d\alpha\wedge u_*=0$
(as $\iota_X$ of this 5-form is equal to~0).

Plugging~\eqref{e:dcom-2}--\eqref{e:dcom-3} into~\eqref{e:dcom-1}, we get the second equality in~\eqref{e:derivative-computed}.
\end{proof}
The pairing in~\eqref{e:derivative-computed} controls how the resonance
at~0 for the operator $P_{1,0}^{(\tau)}$ moves as we perturb $\tau$ from~0,
and the nondegeneracy condition~\eqref{e:nondegeneracy-assumption}
roughly speaking means that the multiplicity of $0$ as a resonance of $P_{1,0}^{(\tau)}$
drops by $\dim d(\Res^1_0)=b_1(\Sigma)$. This observation
is made precise in Lemma~\ref{l:perturb-1-forms} below, but first we need to review perturbation theory of Pollicott--Ruelle resonances. It will
be more convenient for us to use the operators~$P_k^{(\tau)}$ rather than~$P_{k,0}^{(\tau)}$ since the latter act
on the $\tau$-dependent space of $k$-forms annihilated by $\iota_{X_\tau}$.
In the rest of this section we assume that $\varepsilon_0>0$
is chosen small, with the precise value varying from line to line.

We will use the perturbation theory developed in~\cite{guedes-bonthonneau-20}.
For an alternative approach,
see~\cite[\S6]{dang-guillarmou-riviere-shen-20}. Since we are interested in the resonance
at~0, we may restrict ourselves to the strip $\{\Im\lambda >-1\}$. Following the notation of~\cite[\S6.1]{cekic-paternain-19}, we consider the $\tau$-independent anisotropic Sobolev spaces
\begin{equation}
  \label{e:anisotropic-sobolev}
\mathcal H_{rG,s}(M;\Omega^k):=e^{-r\Op(G)}H^s(M;\Omega^k),\quad
r\geq 0,\quad s\in\mathbb R.
\end{equation}
Here $\Op$ is a quantization procedure on~$M$, $G(\rho,\xi)=m(\rho,\xi)\log (1+|\xi|)$
is a logarithmically growing symbol on the cotangent bundle $T^*M$, $|\xi|$ denotes an appropriately chosen norm on the fibers of $T^*M$, and the function $m(\rho,\xi)$,
homogeneous of order~0 in~$\xi$, satisfies certain conditions~\cite[(4)]{guedes-bonthonneau-20} with respect to the
vector field $X_\tau$ for all~$\tau\in(-\varepsilon_0,\varepsilon_0)$. The space $H^s$
is the usual Sobolev space of order~$s$. Denote the domain of $P^{(\tau)}_k$
on $\mathcal H_{rG,s}$ by
$$
\mathcal D^{(\tau)}_{rG,s}(M;\Omega^k):=\{u\in \mathcal H_{rG,s}(M;\Omega^k)\mid P_k^{(\tau)}u\in \mathcal H_{rG,s}(M;\Omega^k)\}.
$$
The following lemma summarizes the perturbation theory used here.
For details see for example~\cite[Theorem~1 and Corollary~2]{guedes-bonthonneau-20}
or~\cite[Lemma~6.1 and~\S6.2]{cekic-paternain-19}.
\begin{lemm}
There exists a constant $C_0$ such that for $r>C_0+|s|$ and~$\tau\in(-\varepsilon_0,\varepsilon_0)$, the operator
\begin{equation}
  \label{e:Fredholm}
P_k^{(\tau)}-\lambda:\mathcal D^{(\tau)}_{rG,s}(M;\Omega^k)\to\mathcal H_{rG,s}(M;\Omega^k),\quad
\Im\lambda>-1
\end{equation}
is Fredholm and its inverse (assuming $\lambda$ is not
a resonance) is given by $R_k^{(\tau)}(\lambda)$. Moreover, the set of pairs $(\tau,\lambda)$ such that $\lambda$ is a resonance of $P_k^{(\tau)}$ is closed and the resolvent
$R_k^{(\tau)}(\lambda):\mathcal H_{rG,s}\to \mathcal H_{rG,s}$
is bounded locally uniformly in $\tau,\lambda$ outside of this set.
\end{lemm}

Since $R_k^{(\tau)}(\lambda)$ is the inverse of $P_k^{(\tau)}-\lambda$
on anisotropic Sobolev spaces, we have the resolvent identity for all~$\tau,\tau'\in(-\varepsilon_0,\varepsilon_0)$
\begin{equation}
  \label{e:res-tau-id}
R_k^{(\tau)}(\lambda)-R_k^{(\tau')}(\lambda)=R_k^{(\tau)}(\lambda)(P_k^{(\tau')}-P_k^{(\tau)})R_k^{(\tau')}(\lambda),\quad
\Im\lambda>-1.
\end{equation}
Here the right-hand side is well-defined since for $r>C_0+|s|+1$ the operator
$R_k^{(\tau')}(\lambda)$ maps $\mathcal H_{rG,s}$ to itself,
$P_k^{(\tau)}$ and $P_k^{(\tau')}$ map $\mathcal H_{rG,s}$ to $\mathcal H_{rG,s-1}$,
and $R_k^{(\tau)}(\lambda)$ maps $\mathcal H_{rG,s-1}$ to itself.
Using~\eqref{e:res-tau-id} we see that for $r>C_0+|s|+1$ the family
$R_k^{(\tau)}(\lambda):\mathcal H_{rG,s}\to\mathcal H_{rG,s-1}$ is locally Lipschitz continuous in~$\tau$. Next, recalling~\eqref{e:Y-def} and that
$P_k^{(\tau)}=-i\mathcal L_{X_\tau}$, we have by~\eqref{e:res-tau-id}
\begin{equation}
  \label{e:R-k-diff}
\partial_\tau R_k^{(\tau)}(\lambda)|_{\tau=0}=iR_k(\lambda)\mathcal L_Y R_k(\lambda)
\end{equation}
as operators $\mathcal H_{rG,s}\to\mathcal H_{rG,s-2}$ when $r>C_0+|s|+2$.

Fix a contour $\gamma$ in the complex plane which encloses~0
but no other resonances of the unperturbed operators $P_k=P_k^{(0)}$.
For $|\tau|<\varepsilon_0$, no resonances of $P_k^{(\tau)}$ lie on
the contour $\gamma$, so we may
define the operators
$$
\widetilde\Pi_k^{(\tau)}:=-{1\over 2\pi i}\oint_\gamma R_k^{(\tau)}(\lambda)\,d\lambda.
$$
Unlike the spectral projectors $\Pi_k^{(\tau)}$ corresponding to the resonance at~0, the operators $\widetilde\Pi_k^{(\tau)}$ depend continuously on~$\tau$, since $R_k^{(\tau)}(\lambda)$ is continuous in~$\tau$.  Moreover, the rank of $\widetilde\Pi_k^{(\tau)}$ is constant in $\tau\in(-\varepsilon_0,\varepsilon_0)$, see~\cite[Lemma~6.2]{cekic-paternain-19}. By~\eqref{e:merex} we have
$$
\widetilde\Pi_k^{(0)}=\Pi_k:=\Pi_k(0)
$$
so the rank of $\widetilde\Pi_k^{(\tau)}$ can be computed using the algebraic multiplicities
of~0 as a resonance in the unperturbed case $\tau=0$ (using~\eqref{e:res-0-split}):
\begin{equation}
  \label{e:perprez-1}
\rank\widetilde\Pi_{k}^{(\tau)}=m_k(0)=m_{k,0}(0)+m_{k-1,0}(0).
\end{equation}
By~\eqref{e:merex}, we also have
\begin{equation}
  \label{e:tilde-Pi-tau-sum}
\widetilde\Pi_k^{(\tau)}=\sum_{\lambda\in\Upsilon^k_\tau} \Pi_k^{(\tau)}(\lambda)
\end{equation}
where $\Upsilon^k_\tau$ is the set of resonances of the operator $P_k^{(\tau)}$ which are enclosed by the contour~$\gamma$. Note that by~\eqref{e:tilde-Pi-tau-sum} and \eqref{e:Pi-k-idem}
\begin{equation}
  \label{e:tilde-Pi-tau-idem}
\widetilde\Pi_k^{(\tau)}\Pi_k^{(\tau)}(\lambda)=\Pi_k^{(\tau)}(\lambda)\quad\text{for all}\quad \lambda\in\Upsilon^k_\tau
\end{equation}
and the range of $\widetilde\Pi_k^{(\tau)}$ is the direct sum
of the ranges $\Res^{k,\infty}_{(\tau)}(\lambda)$ of $\Pi_k^{(\tau)}(\lambda)$ over $\lambda\in\Upsilon^k_\tau$. In particular, using~\eqref{e:res-0-split} we get
\begin{equation}
  \label{e:tilde-Pi-tau-mult}
\rank\widetilde\Pi_k^{(\tau)}=\sum_{\lambda\in\Upsilon^k_\tau} \big(m_{k,0}^{(\tau)}(\lambda)+m_{k-1,0}^{(\tau)}(\lambda)\big).
\end{equation}
Together with~\eqref{e:perprez-1} and induction on~$k$ this implies for $|\tau|<\varepsilon_0$
\begin{equation}
  \label{e:perprez-2}
\sum_{\lambda\in\Upsilon^k_\tau} m_{k,0}^{(\tau)}(\lambda)=m_{k,0}(0).
\end{equation}
We are now ready to show that under the condition~\eqref{e:nondegeneracy-assumption}
the space $\Res^1_{0(\tau)}$ of resonant 1-forms at~0 for the perturbed
operator $P^{(\tau)}_{1,0}$, $\tau\neq 0$, consists of closed forms:
\begin{lemm}
  \label{l:perturb-1-forms}
Under the assumptions of Theorem~\ref{t:general-perturber}, there exists
$\varepsilon_0>0$ such that for $0<|\tau|<\varepsilon_0$ we have
$d(\Res^1_{0(\tau)})=0$.
\end{lemm}
\begin{proof}
1. Define the operator
$$
Z(\tau):=P_1^{(\tau)}\widetilde\Pi_1^{(\tau)}.
$$
Roughly speaking this operator contains information about the nonzero resonances
of $P_1^{(\tau)}$ enclosed by~$\gamma$; in particular,
each of the corresponding spaces of generalized resonant states is in the range of $Z(\tau)$
as can be seen from~\eqref{e:tilde-Pi-tau-idem}.

In the hyperbolic case $\tau=0$, the semisimplicity condition~\eqref{e:semi-simple}
holds for the operator $P_1$ at $\lambda=0$, as follows from Lemmas~\ref{l:0-forms}
and~\ref{l:c-psi-3} together with~\eqref{e:res-0-split}.
Therefore, the range of $\widetilde\Pi_1^{(0)}=\Pi_1$ is contained in $\Res^1$, implying that
\begin{equation}
  \label{e:Pi-1-initial}
Z(0)=0.
\end{equation}
By~\eqref{e:perprez-1},
the rank of $\widetilde\Pi_1^{(\tau)}$ can be computed using
the algebraic multiplicities of~0 as a resonance in the hyperbolic case $\tau=0$,
which are known by~\eqref{e:hyperbolic-alg}:
\begin{equation}
  \label{e:Pi-1-tilde-rank}
\rank \widetilde\Pi_1^{(\tau)}=2b_1(\Sigma)+1.
\end{equation}
The intersection of the range of $\widetilde\Pi_1^{(\tau)}$ with the kernel of $P_1^{(\tau)}$ is equal to~$\Res^1_{(\tau)}$. By~\eqref{e:res-0-split} and Lemma~\ref{l:0-forms} we have $\Res^1_{(\tau)}=\Res^1_{0(\tau)}\oplus \mathbb C\alpha_\tau$. Next, by Lemma~\ref{l:1-forms-ish} and~\eqref{e:bettiGysin} we have $\dim\Res^1_{0(\tau)}=b_1(\Sigma)+\dim d(\Res^1_{0(\tau)})$. Therefore
$$
\dim\Res^1_{(\tau)}=b_1(\Sigma)+1+\dim d(\Res^1_{0(\tau)}).
$$
By the Rank-Nullity Theorem and~\eqref{e:Pi-1-tilde-rank} we then have
\begin{equation}
  \label{e:Pi-1-tilde-rank-2}
\rank Z(\tau)= b_1(\Sigma)-\dim d(\Res^1_{0(\tau)}).
\end{equation}

\noindent 2.
Since $(P_1^{(\tau)}-\lambda) R_1^{(\tau)}(\lambda)$ is the identity operator, we
have for all~$\tau$
$$
Z(\tau)=-{1\over 2\pi i}\oint_\gamma \lambda R_1^{(\tau)}(\lambda)\,d\lambda.
$$
Using~\eqref{e:R-k-diff} we now compute the derivative
$$
\partial_\tau Z(0)
=-{1\over 2\pi}\oint_\gamma \lambda R_1(\lambda)\mathcal L_Y R_1(\lambda)\,d\lambda=-i\Pi_1\mathcal L_Y\Pi_1.
$$
Here in the second equality we used the Laurent expansion~\eqref{e:merex}
for $R_1(\lambda)$ at $\lambda_0=0$ (recalling that $J_1(0)=1$ by semisimplicity).

By Lemma~\ref{l:derivative-computed}, for any $u\in\Res^1_0$, $u_*\in\Res^1_{0*}$
we have
\begin{equation}
  \label{e:nondegeneracy-exhibit}
\int_M \alpha\wedge d\alpha\wedge \big(\partial_\tau Z(0)u\big)\wedge u_*=-i\llangle (\iota_X \beta)du,du_*\rrangle.
\end{equation}
By the nondegeneracy assumption~\eqref{e:nondegeneracy-assumption} the bilinear form~\eqref{e:nondegeneracy-exhibit}
is nondegenerate on $u\in\mathcal C_\psi$,
$u_*\in\mathcal C_{\psi*}$. This implies that
\begin{equation}
  \label{e:Pi-1-der-rank}
\rank\partial_{\tau}Z(0)\geq \dim\mathcal C_\psi=b_1(\Sigma).
\end{equation}
Together~\eqref{e:Pi-1-initial} and~\eqref{e:Pi-1-der-rank}
show that for $0<|\tau|<\varepsilon_0$
$$
\rank Z(\tau)\geq b_1(\Sigma).
$$
Then by~\eqref{e:Pi-1-tilde-rank-2} we have $\dim d(\Res^1_{0(\tau)})=0$ for $0 < |\tau| < \varepsilon_0$ which finishes the proof.
\end{proof}
\Remark Lemma \ref{l:perturb-1-forms} holds more generally whenever $P_{1, 0}$ is semisimple. If for all contact perturbations $(\alpha_\tau)_\tau$ we would have that \eqref{e:nondegeneracy-assumption} is trivial, this would imply that $du \wedge du_* = 0$ for all $u \in \Res_0^1$ and $u_* \in \Res_{0*}^1$. When $(\Sigma, g)$ is hyperbolic, we will show in \S \ref{s:pert-supp} that this is impossible, while for general $(\Sigma, g)$ proving such a statement seems out of reach for now.

Together with Lemma~\ref{l:0-forms}, Lemma~\ref{l:res-3-closed}, Lemma~\ref{l:1-means-2},
and~\eqref{e:bettiGysin} Lemma~\ref{l:perturb-1-forms}
gives all the conclusions of Theorem~\ref{t:general-perturber} except
semisimplicity on 2-forms.
In particular we have for $0<|\tau|<\varepsilon_0$
(using~\eqref{e:res-0-split})
\begin{align}
  \label{e:alth}
\dim\Res^2_{0(\tau)}&=b_1(\Sigma)+2,\\
  \label{e:alth2}
d(\Res^{1,\infty}_{(\tau)})&=\mathbb Cd\alpha_\tau.
\end{align}
To finish the proof of Theorem~\ref{t:general-perturber} it remains
to establish semisimplicity on 2-forms:
\begin{lemm}
  \label{l:perturb-2-forms}
Under the assumptions of Theorem~\ref{t:general-perturber}, there exists
$\varepsilon_0>0$ such that for $0<|\tau|<\varepsilon_0$ the
semisimplicity condition~\eqref{e:semi-simple} holds at $\lambda_0=0$
for the operator $P^{(\tau)}_{2,0}$.
\end{lemm}
\begin{proof}
We first claim that for $0<|\tau|<\varepsilon_0$
\begin{equation}
  \label{e:p2f0-1}
\rank\big(\alpha_\tau\wedge(\widetilde\Pi_2^{(\tau)}-\Pi_2^{(\tau)})\big)\geq 
\rank\big(\alpha_\tau\wedge d(\widetilde\Pi_1^{(\tau)}-\Pi_1^{(\tau)})\big)\geq
b_1(\Sigma).
\end{equation}
Indeed, by~\eqref{e:commuter-Pi} and~\eqref{e:tilde-Pi-tau-sum}
we have $d(\widetilde\Pi_1^{(\tau)}-\Pi_1^{(\tau)})=(\widetilde\Pi_2^{(\tau)}-\Pi_2^{(\tau)})d$ which implies the first inequality in~\eqref{e:p2f0-1}. Next, we have $\rank(\alpha\wedge d\widetilde\Pi_1^{(0)})=b_1(\Sigma)+1$ as
the range of $d\widetilde\Pi_1^{(0)}$
is equal to $d\Res^1=\mathbb Cd\alpha\oplus d\mathcal C_\psi$.
Since $\widetilde\Pi_1^{(\tau)}$ depends continuously
on $\tau$, we see that $\rank(\alpha_\tau\wedge d\widetilde\Pi_1^{(\tau)})\geq b_1(\Sigma)+1$
for all small enough $\tau$.
On the other hand, for $\tau$ small but nonzero we have
$\rank d\Pi_1^{(\tau)}=1$ by~\eqref{e:alth2}.
Together these imply the second inequality in~\eqref{e:p2f0-1}.

Now, by~\eqref{e:tilde-Pi-tau-sum} and~\eqref{e:res-0-split}
the range of $\alpha_\tau\wedge(\widetilde\Pi_2^{(\tau)}-\Pi_2^{(\tau)})$
is contained in the sum of the spaces
$\alpha_\tau\wedge\Res^{2,\infty}_{0(\tau)}(\lambda)$ over
$\lambda\in\Upsilon^2_\tau\setminus \{0\}$.
Therefore~\eqref{e:p2f0-1} implies that for $0<|\tau|<\varepsilon_0$
\begin{equation}
  \label{e:p2f0-2}
\sum_{\lambda\in\Upsilon^2_\tau\setminus \{0\}}
m_{2,0}^{(\tau)}(\lambda)\geq b_1(\Sigma).
\end{equation}
From~\eqref{e:perprez-2} and~\eqref{e:hyperbolic-alg} we see that
$$
\sum_{\lambda\in\Upsilon^2_\tau}m_{2,0}^{(\tau)}(\lambda)=m_{2,0}(0)=2b_1(\Sigma)+2
$$
therefore by~\eqref{e:p2f0-2} we have $m_{2,0}^{(\tau)}(0)\leq b_1(\Sigma)+2$.
Since $\dim\Res^2_{0(\tau)}=b_1(\Sigma)+2$ by~\eqref{e:alth}, we showed that the algebraic
and geometric multiplicities for~0 as a resonance
of $P^{(\tau)}_{2,0}$ coincide, finishing the proof.
\end{proof}

\subsection{The full support property}
\label{s:pert-supp}

In this section, we prove a full support statement
which will be used in the proof of Theorem~\ref{t:pertnondeg-contact}. In fact, we recall that we need to prove the nondegeneracy assumption~\eqref{e:nondegeneracy-assumption}, that is, that $\llangle{\iota_X \beta \bullet, \bullet}\rrangle$ is nondegenerate on $d\Res_0^1 \times d\Res_{0*}^1$, and the support properties of elements of $d\Res_{0(*)}^1$ will be useful. In~\S\S\ref{s:pert-supp}--\ref{s:main-proof}
we assume that $M=S\Sigma$ where $(\Sigma,g)$ is a hyperbolic 3-manifold
and the contact form $\alpha$ and the spaces of (co-)resonant states
at zero
$\Res^1_0$, $\Res^1_{0*}$ are defined using the geodesic flow on $(\Sigma,g)$.
\begin{prop}\label{prop:main3}
For all $u\in\mathrm{Res}_0^1$, $u_*\in \mathrm{Res}_{0\ast}^1$ with $du\neq0$, $du_*\neq 0$, the distributional $5$-form 
$\alpha\wedge du \wedge du_*$  fulfills 
$\mathrm{supp}(\alpha\wedge du \wedge du_*)= M$.
\end{prop}
To show Proposition~\ref{prop:main3}, we first study properties
of the 2-forms $du$ and $du_*$. Define the smooth 2-forms
$$
\omega_\pm\in C^\infty(M;\Omega^2_0)
$$
by requiring that $E_0\oplus E_u$ be in the kernel of
$\omega_-$, $E_0\oplus E_s$ be in the kernel of $\omega_+$,
and, using the horizontal/vertical decomposition~\eqref{e:hv-map}
\begin{equation}
  \label{e:omega-pm-1}
\omega_\pm(x,v)\big((w_1,\pm w_1),(w_2,\pm w_2)\big)=
\langle v\times w_1,w_2\rangle_g\quad\text{for all}\quad
w_1,w_2\in \{v\}^\perp\subset T_x\Sigma
\end{equation}
where `$\times$' denotes the cross product on $T_x\Sigma$
defined in~\S\ref{s:hyp-psi-1}. In terms
of the canonical 1-forms on the frame bundle $\mathcal F\Sigma$
defined in~\S\ref{s:canonical-forms} the lifts of $\omega_\pm$
to $\mathcal F\Sigma$ are given by
\begin{equation}
  \label{e:omega-pm-2}
\omega_\pm=U_1^{\pm*}\wedge U_2^{\pm*}.
\end{equation}
One can think of $\omega_\pm$ as canonical volume forms on the stable/unstable spaces.

By~\eqref{e:omega-pm-2} and~\eqref{e:canonical-diff} we compute
\begin{equation}
  \label{e:d-omega-pm}
d\omega_\pm=\pm 2\alpha\wedge\omega_\pm.
\end{equation}
\begin{lemm}
  \label{l:du-dv}
Assume that $u\in\Res^1_0$, $u_*\in\Res^1_{0*}$. Then
\begin{gather}
  \label{e:du-dv-1}
du=f_-\omega_-,\quad
du_*=f_+\omega_+;\\
  \label{e:du-dv-2}
\alpha\wedge du\wedge du_*=-\textstyle{1\over 8}f_-f_+ d\vol_\alpha
\end{gather}
where the distributions $f_-\in\mathcal D'_{E_u^*}(M;\mathbb C)$, $f_+\in\mathcal D'_{E_s^*}(M;\mathbb C)$ satisfy for any vector fields $U_-\in C^\infty(M;E_u)$,
$U_+\in C^\infty(M;E_s)$
\begin{equation}
  \label{e:f-pm-equiv}
(X\pm 2)f_\pm=0,\quad
U_\pm f_\pm=0.
\end{equation}
\end{lemm}
\begin{proof}
We consider the case of $du$, with $du_*$ studied similarly.
From Lemma~\ref{l:c-psi-1.5} we know that $u$ is a totally unstable 1-form,
which implies that $du$ is a section of $E_u^*\wedge E_u^*$. The latter is a one-dimensional vector
bundle over~$M$ and $\omega_-$ is a nonvanishing smooth section of it,
so $du=f_-\omega_-$ for some $f_-\in\mathcal D'_{E_u^*}(M;\mathbb C)$.
Using~\eqref{e:d-omega-pm} we compute
$$
0=d(f_-\omega_-)=(df_- - 2f_-\alpha)\wedge\omega_-.
$$
Taking $\iota_X$ and $\iota_{U_-}$ of this identity and using that
$\iota_X\omega_-=\iota_{U_-}\omega_-=\iota_{U_-}\alpha=0$ (recalling the definitions of $U_1^{\pm *}, U_2^{\pm *}$ in \eqref{e:forms-dual} and below), we get~\eqref{e:f-pm-equiv}.

Finally, \eqref{e:du-dv-2} follows from~\eqref{e:du-dv-1} and the following identity
which can be verified using either~\eqref{e:omega-pm-1} and~\eqref{e:alpha-hv}
or~\eqref{e:omega-pm-2} and~\eqref{e:canonical-diff}:
$$
\alpha\wedge\omega_-\wedge\omega_+=-\textstyle{1\over 8}d\vol_\alpha.
$$
\end{proof}
We can now finish the proof of Proposition~\ref{prop:main3}. Given~\eqref{e:du-dv-2}
it suffices to prove that, assuming that $f_-\neq 0$ and $f_+\neq 0$,
\begin{equation}
  \label{e:f-pm-prod-supp}
\supp(f_-f_+)=M.
\end{equation}
Let $\pi_\Gamma:S\mathbb H^3\to S\Sigma=M$ be the covering map corresponding to~\eqref{e:S-Sigma-Gamma} and $\Phi_\pm,B_\pm$ be defined in~\eqref{e:B-pm}. Then by~\eqref{e:X-Phi-pm} and~\eqref{e:f-pm-equiv} we have for any $U_-\in C^\infty(S\mathbb H^3;E_u)$, $U_+\in C^\infty(S\mathbb H^3;E_s)$
$$
X(\Phi_\pm^{2}(f_\pm\circ\pi_\Gamma))=U_\pm(\Phi_\pm^{2}(f_\pm\circ\pi_\Gamma))=0,
$$
that is $\Phi_+^2(f_+\circ\pi_\Gamma)$ is totally stable and $\Phi_-^2(f_-\circ\pi_\Gamma)$ is totally unstable in the sense of Definition~\ref{d:totally-stun}. Similarly to Lemma~\ref{l:lifted-forms} we can then describe
the lifts of $f_\pm$ to $S\mathbb H^3$ in terms of some distributions~$g_\pm$
on the conformal infinity~$\mathbb S^2$:
\begin{equation}
  \label{e:f-pm-shape}
f_\pm\circ\pi_\Gamma=\Phi_\pm^{-2}(g_\pm\circ B_\pm)\quad\text{for some}\quad
g_\pm\in\mathcal D'(\mathbb S^2;\mathbb C).
\end{equation}
Since $f_\pm$ are resonant states of $X$, a result of Weich~\cite[Theorem~1]{weich-17} shows that $\supp f_+=\supp f_-= M$, which from \eqref{e:f-pm-shape} and the facts that $\Phi_\pm > 0$, and that $B_\pm$ are submersions which map $S\mathbb H^3$ onto~$\mathbb S^2$, implies that
\begin{equation}
  \label{e:Weich-useful}
\supp g_+=\supp g_-=\mathbb S^2.
\end{equation}
We will now use the coordinates $(\nu_-, \nu_+, t) \in (\mathbb{S}^2 \times \mathbb{S}^2)_- \times \mathbb{R}$ on $S\mathbb{H}^3$ introduced in \eqref{e:Xi-def}. Then by~\eqref{e:f-pm-shape} and~\eqref{e:Phi-pm-B-pm-id} we can write
in these coordinates
$$
(f_-f_+)\circ\pi_\Gamma=\textstyle{1\over 16}|\nu_--\nu_+|^4 g_-(\nu_-)g_+(\nu_+).
$$
By~\eqref{e:Weich-useful}, we see that the support of the tensor
product $g_-\otimes g_+(\nu_-,\nu_+)=g_-(\nu_-)g_+(\nu_+)$ is equal to the entire $\mathbb S^2\times\mathbb S^2$, which implies that
$\supp(f_-f_+)\circ\pi_\Gamma=S\mathbb H^3$ and thus
$\supp(f_-f_+)=M$. This shows~\eqref{e:f-pm-prod-supp} and finishes the proof.

\subsection{Proof of Theorem~\ref{t:pertnondeg-contact}}
\label{s:pert-endgame}
	
We first remark that in the special case $\dim d(\Res^1_0)=b_1(\Sigma)=1$,
it is straightforward to see that Proposition~\ref{prop:main3} implies
the following simplified version of Theorem~\ref{t:pertnondeg-contact}:
for each nonempty open set $\mathscr U\subset M$ there exists
$\mathbf a\in C^\infty(M;\mathbb R)$ with $\supp\mathbf a\subset \mathscr U$
and such that $\beta:=\mathbf a\alpha$ satisfies~\eqref{e:nondegeneracy-assumption}.
Indeed, it suffices to fix any nonzero $du\in d(\Res^1_0)$, $du_*\in d(\Res^1_{0*})$,
and choose $\mathbf a$ such that $\int_M \mathbf a \alpha\wedge du\wedge du_*\neq 0$. We note that there are examples of hyperbolic $3$-manifolds with $b_1(\Sigma) = 1$, see for instance \cite[Theorem 13.4]{Farb-Margalit-12}.

For the general case, we will use the following basic fact from linear algebra:
\begin{lemm}\label{lem:matrixsubspace}
Denote by $\otimes^2 \mathbb C^n$ the space of complex $n\times n$ matrices.
Assume that $V\subset \otimes^2\mathbb C^n$ is a subspace such that
for each $v_1,v_2\in\mathbb C^n\setminus\{0\}$ there exists $B\in V$ such that $\langle Bv_1,v_2\rangle\neq 0$. (Here $\langle\bullet,\bullet\rangle$ denotes the canonical bilinear
inner product on $\mathbb C^n$.)
Then the set of invertible matrices in~$V$ is dense.
\end{lemm}
\begin{proof}
Let $\mathscr O$ be a nonempty open subset of $V$. We need to show that $\mathscr O$ contains an invertible matrix. Assume that there are no invertible matrices in $\mathscr O$. Let $A$ be a matrix of maximal rank in $\mathscr O$, then $k:=\rank A<n$ since $A$ cannot be invertible.
There exist bases $e_1,\dots,e_n$ and $e_1^*,\dots,e_n^*$ of $\mathbb C^n$ such that
$$
\langle Ae_j,e^*_{\ell}\rangle=\begin{cases}
1&\text{if }j=\ell\leq k;\\
0&\text{otherwise.}
\end{cases}
$$
By the assumption of the lemma,
there exists $B\in V$ such that $\langle Be_{k+1},e^*_{k+1}\rangle\neq 0$.
Consider the matrix $A_t=A+tB$ which lies in $\mathscr O$
for sufficiently small~$t$, and let $b(t)$ be the determinant of the matrix
$(\langle A_t e_j,e^*_\ell\rangle)_{j,\ell=1}^{k+1}$. Then $b(0)=0$
and $b'(0)=\langle Be_{k+1},e^*_{k+1}\rangle\neq 0$.
Therefore, for small enough $t\neq 0$ we have $b(t)\neq 0$,
which means that $\rank A_t\geq k+1$. This contradicts
the fact that $k$ was the maximal rank of any matrix in $\mathscr O$.
\end{proof}
We are now ready to give the proof of Theorem~\ref{t:pertnondeg-contact}.
For $\mathbf a\in C^\infty(M;\mathbb R)$, define the bilinear form
$$
S_{\mathbf a}:d(\Res^1_0)\times d(\Res^1_{0*})\to\mathbb C,\quad
S_{\mathbf a}(du,du_*)=\int_M\mathbf a
\alpha\wedge du\wedge du_*.
$$
To prove Theorem~\ref{t:pertnondeg-contact}, it then suffices to show
that the set of $\mathbf a\in \CIc(\mathscr U;\mathbb R)$
such that $S_{\mathbf a}$ is nondegenerate is open and dense. Since
nondegeneracy is an open condition, this set is automatically open.
To show that it is dense, consider the finite dimensional vector space
$$
V:=\{S_{\mathbf a}\mid\mathbf a\in \CIc(\mathscr U;\mathbb R)\}.
$$
Choosing bases of the $b_1(\Sigma)$-dimensional spaces $d(\Res^1_0)$
and $d(\Res^1_{0*})$, we can identify $V$ with a subspace
of $\otimes^2\mathbb C^{b_1(\Sigma)}$.
Let $du\in d(\Res^1_0)$, $du_*\in d(\Res^1_{0*})$ be nonzero,
then by Proposition~\ref{prop:main3} we have
$\supp(\alpha\wedge du\wedge du_*)=M$, so there exists
$\mathbf a\in \CIc(\mathscr U;\mathbb R)$ such that
$S_{\mathbf a}(du,du_*)\neq 0$.
Then by Lemma~\ref{lem:matrixsubspace} the set of nondegenerate bilinear forms in~$V$ is dense.

Let $\mathbf U$ be a nonempty open subset of~$\CIc(\mathscr U;\mathbb R)$. Then $\{S_{\mathbf a}\mid \mathbf a\in \mathbf U\}$
is a nonempty open subset of~$V$. Thus there exists
$\mathbf a\in\mathbf U$ such that $S_{\mathbf a}$ is nondegenerate, which 
finishes the proof.

\subsection{Proof of Theorem~\ref{thm:main}}
\label{s:main-proof}
	
We now give the proof of part~2 of Theorem~\ref{thm:main},
relying on Theorem~\ref{t:main-identity} (in~\S\ref{s:main-identity}) and Proposition~\ref{prop:harmoniconeforms} below, combined together in Corollary~\ref{c:key-corr}.
(Part~1 of Theorem~\ref{thm:main} was proved in Corollary~\ref{l:hyperbolic-corr} above.)

We start by computing how a general metric perturbation
affects the contact form for the geodesic flow.
Let $(\Sigma,g)$ be any compact 3-dimensional Riemannian manifold
and the contact form $\alpha$ and the generator $X$ of the geodesic flow
on $S\Sigma$ be defined as in~\S\ref{s:geodesic-flows}.
Let
$$
g_\tau,\quad
\tau\in (-\varepsilon,\varepsilon)
$$
be a family of Riemannian metrics on $\Sigma$ depending smoothly on~$\tau$, such that $g_0=g$. The associated geodesic flows act on the $\tau$-dependent sphere bundles
$$
S^{(\tau)}\Sigma=\{(x,v)\in T\Sigma\colon |v|_{g_\tau}=1\}.
$$
To bring these geodesic flows to $S\Sigma$, we use the diffeomorphisms
$$
\Phi_\tau:S\Sigma\to S^{(\tau)}\Sigma,\quad
\Phi_\tau(x,v)=\bigg(x,{v\over |v|_{g_\tau}}\bigg).
$$
Denote by $\alpha_\tau$ the contact form on $S^{(\tau)}\Sigma$
corresponding to $g_\tau$. Then
$$
\widetilde\alpha_\tau:=\Phi_\tau^*\alpha_\tau
$$
is a contact 1-form on $S\Sigma$ and the corresponding contact
flow is the geodesic flow of $(\Sigma,g_\tau)$ pulled back by $\Phi_\tau$.

Let $\pi^{(\tau)}_\Sigma:S^{(\tau)}\Sigma\to \Sigma$ be the projection map. Using~\eqref{e:geodesic-alpha} and the fact that $\pi^{(\tau)}_\Sigma\circ\Phi_\tau$
is equal to $\pi_\Sigma:=\pi^{(0)}_\Sigma$, we compute for all $(x,v)\in S\Sigma$
and $\xi\in T_{(x,v)}(S\Sigma)$
$$
\langle \widetilde\alpha_\tau(x,v),\xi\rangle={\langle v,d\pi_\Sigma(x,v)\xi\rangle_{g_\tau}\over|v|_{g_\tau}}.
$$
Recalling $d\pi_{\Sigma}(x, v)X(x, v) = v$ (see \eqref{e:X-hv}) and using $g_0(v, v) = 1$, it follows that
\begin{equation}
\iota_X\partial_\tau\widetilde\alpha_\tau|_{\tau=0}(x,v) = \partial_\tau g_\tau(v, v)|_{\tau = 0} - \frac{1}{2} g_0(v, v) \cdot \partial_\tau g_\tau(v, v)|_{\tau = 0} = \partial_\tau|v|_{g_\tau}|_{\tau=0}.
\end{equation}
In particular, if the metric $g_\tau$ is given by a conformal perturbation
$g_\tau=e^{-2\tau \mathbf b}g$, where $\mathbf b\in C^\infty(\Sigma;\mathbb R)$, then
\begin{equation}
  \label{e:X-metric-pert}
\iota_X\partial_\tau\widetilde\alpha_\tau|_{\tau=0}(x,v)=-\mathbf b\circ\pi_\Sigma.
\end{equation}

We are now ready to prove Theorem~\ref{thm:main}. Assume that
$(\Sigma,g)$ is a hyperbolic 3-manifold as defined in~\S\ref{s:hyp-3-basics}
and put $g_\tau:=e^{-2\tau \mathbf b}g$.
By Theorem~\ref{t:general-perturber} applied to the family of contact forms $\widetilde\alpha_\tau$, with $\beta=\partial_\tau\widetilde\alpha_\tau|_{\tau=0}$ satisfying~\eqref{e:X-metric-pert}, it suffices to show that for $\mathbf b$ in an open and dense subset of~$C^\infty(\Sigma;\mathbb R)$ the bilinear form
$$
(du,du_*)\mapsto \int_M (\mathbf b\circ\pi_\Sigma)\alpha\wedge du\wedge du_*
$$
is nondegenerate on $d(\Res^1_0)\times d(\Res^1_{0*})$.

The space $\Res^1_0$ is preserved by complex conjugation
as follows from its definition~\eqref{e:res-k-0};
here we use that for any $u$ we have $\WF(\bar u)=\{(\rho,-\xi)\mid (\rho,\xi)\in\WF(u)\}$.
Denote by $\Res^1_{0\mathbb R}$ the space of real-valued 1-forms in $\Res^1_0$
and let $\mathcal J(x,v)=(x,-v)$ be the map defined in~\eqref{e:J-def}.
By~\eqref{e:J-star-def}, the pullback $\mathcal J^*$ is an isomorphism
from $\Res^1_0$ onto $\Res^1_{0*}$. Thus it suffices to show that
for $\mathbf b$ in an open and dense subset of~$C^\infty(\Sigma;\mathbb R)$ the real bilinear form
$$
\widetilde S_{\mathbf b}(du,du'):=\int_M (\mathbf b\circ\pi_\Sigma) \alpha\wedge du\wedge \mathcal J^*(du')
$$
is nondegenerate on $d(\Res^1_{0\mathbb R})\times d(\Res^1_{0\mathbb R})$.

Since $\mathbf b\circ\pi_\Sigma$ is $\mathcal J$-invariant, $\mathcal J^*\alpha=-\alpha$,
and $\mathcal J$ is an orientation reversing diffeomorphism on~$M$,
we see that $\widetilde S_{\mathbf b}$ is  a symmetric bilinear form. Unlike in the contact perturbation case in \S \ref{s:pert-endgame}, we will not be able to produce for every pair $(du, du') \in d(\Res^1_{0\mathbb R})\times d(\Res^1_{0\mathbb R})$ an element $\mathbf{b} \in C^\infty(\Sigma; \mathbb{R})$ such that $\widetilde S_{\mathbf b}(du,du') \neq 0$. Instead, we will only produce $\mathbf{b}$ such that $\widetilde S_{\mathbf b}(du,du) \neq 0$. Hence, we will need the following variant of Lemma~\ref{lem:matrixsubspace} for
symmetric matrices:
\begin{lemm}\label{lem:matrixsubspace-2}
Denote by $\otimes^2_S \mathbb R^n$ the space of real symmetric $n\times n$ matrices.
Assume that $V\subset \otimes^2_S\mathbb R^n$ is a subspace such that
for each $w\in\mathbb R^n\setminus\{0\}$ there exists $B\in V$ such that $\langle Bw,w\rangle\neq 0$.
Then the set of invertible matrices in~$V$ is dense.
\end{lemm}
\begin{proof}
Similarly to the proof of Lemma~\ref{lem:matrixsubspace},
assume that $\mathscr O$ is a nonempty open subset of~$V$ which does not
contain any invertible matrices and $A$ is a matrix in $\mathscr O$
of maximal rank $k<n$.
Since $A$ is symmetric, it can be diagonalized, i.e. there exists an orthonormal basis $e_1,\dots,e_n$ of $\mathbb R^n$ such that $Ae_j=\lambda_j e_j$ where $\lambda_j$ are real and, since $\rank A=k$, we may assume that $\lambda_1,\dots,\lambda_k\neq 0$ and~$\lambda_{k+1}=\dots=\lambda_n=0$.

By the assumption of the lemma,
there exists $B\in V$ such that $\langle Be_{k+1},e_{k+1}\rangle\neq 0$.
Consider the matrix $A_t=A+tB$ which lies in $\mathscr O$
for sufficiently small~$t$, and let $b(t)$ be the determinant of the matrix
$(\langle A_t e_i,e_j\rangle)_{i,j=1}^{k+1}$. Then $b(0)=0$
and $b'(0)=\lambda_1\cdots\lambda_k\langle Be_{k+1},e_{k+1}\rangle\neq 0$.
Therefore, for small enough $t\neq 0$ we have $b(t)\neq 0$,
which means that $\rank A_t\geq k+1$. This contradicts
the fact that $k$ was the maximal rank of any matrix in $\mathscr O$.
\end{proof}

Now to show Theorem~\ref{thm:main} it remains to follow the argument at the end of~\S\ref{s:pert-endgame}, with Lemma~\ref{lem:matrixsubspace} replaced by Lemma~\ref{lem:matrixsubspace-2}
and using the following
\begin{prop}
  \label{l:nondeg-metric}
Assume that $u\in\Res^1_{0\mathbb R}$ and $du\neq 0$.
Then there exists $\mathbf b\in C^\infty(\Sigma;\mathbb R)$
such that $\widetilde S_{\mathbf b}(du,du)\neq 0$.
\end{prop}
\begin{proof}
Using the pushforward map $\pi_{\Sigma*}^{}$ defined in~\eqref{e:forms-pf} we compute by~\eqref{e:pushforward-integral} and~\eqref{e:pfm-2}
\begin{equation}
  \label{e:nondeg-metric-comp}
\widetilde S_{\mathbf b}(du,du)=-\int_\Sigma \mathbf b\pi_{\Sigma*}^{}(\alpha\wedge du\wedge \mathcal J^*(du)).
\end{equation}
By Corollary~\ref{c:key-corr} below we have 
$\pi_{\Sigma*}^{}(\alpha\wedge du\wedge\mathcal J^*(du))\neq 0$
which finishes the proof.
\end{proof}

\section{The pushforward identity}
\label{s:main-identity}

In this section we prove an identity, Theorem~\ref{t:main-identity},
used in Proposition~\ref{l:nondeg-metric} above
which is a key component in the proof of our main Theorem~\ref{thm:main}.

We assume throughout this section that $(\Sigma,g)$ is a compact hyperbolic
3-manifold as defined in~\S\ref{s:hyp-3-basics} and write
$\Sigma=\Gamma\backslash\mathbb H^3$ where $\Gamma\subset \SO_+(1,3)$.
For $s>2$, define the operator
\begin{equation}
  \label{e:Q-s-def}
Q_s:\CIc(\mathbb{H}^3)\to C^\infty(\mathbb{H}^3),\quad
Q_s f(x):=\int_{\mathbb H^3} \big(\cosh d_{\mathbb H^3}(x,y)\big)^{-s}\,f(y)\,d\vol_g(y).
\end{equation}
As shown in~\S\ref{s:Q-s-prop} below, the operator $Q_s$
can be extended to $\Gamma$-invariant distributions on~$\mathbb H^3$
and it is smoothing, so it descends to an operator
\begin{equation}
  \label{e:Q-s-intro}
Q_s:\mathcal D'(\Sigma;\mathbb C)\to C^\infty(\Sigma;\mathbb C).
\end{equation}
Let $\Delta_g$ be the (nonpositive) Laplace--Beltrami operator
on $(\Sigma,g)$. Recall the pushforward map on forms $\pi_{\Sigma*}^{}$ defined in~\eqref{e:forms-pf}
and the spaces of (co-)resonant $k$-forms $\Res^k_0,\Res^k_{0*}$ on $M=S\Sigma$ associated to
the geodesic flow on $(\Sigma,g)$, see~\S\S\ref{s:geodesic-flows}--\ref{section:PRresonances}.

The main result of this section is the following
\begin{theo}
  \label{t:main-identity}
Assume that $u\in\Res^1_0$, $u_*\in\Res^1_{0*}$. Define the pushforwards
\begin{equation}
  \label{e:omega-u-def}
\sigma_-:=\pi_{\Sigma*}^{}(d\alpha\wedge u),\quad
\sigma_+:=\pi_{\Sigma*}^{}(d\alpha\wedge u_*)
\end{equation}
which are harmonic 1-forms on $\Sigma$
by Lemma~\ref{l:harmonic-push}. Define $F\in \mathcal D'(\Sigma;\mathbb C)$ by
\begin{equation}
  \label{e:F-def}
\pi_{\Sigma*}^{}(\alpha\wedge du\wedge du_*)=F\,d\vol_g.
\end{equation}
Then we have
\begin{equation}
  \label{e:main-identity}
Q_4F=-\textstyle{1\over 6}\Delta_g(\sigma_-\cdot\sigma_+)
\end{equation}
where the inner product $\sigma_-\cdot\sigma_+$ is the function on~$\Sigma$ defined by
$\sigma_-\cdot\sigma_+(x)=\langle\sigma_-(x),\sigma_+(x)\rangle_g$.
\end{theo}
\Remark By~\eqref{e:nondeg-metric-comp} and since $Q_4$ is self-adjoint we can rewrite~\eqref{e:main-identity}
as follows: for each~$\mathbf b\in \mathcal D'(\Sigma)$,
\begin{equation}
  \label{e:PaSa}
{1\over 6}\int_\Sigma \mathbf b\,\Delta_g(\sigma_-\cdot\sigma_+)\,d\vol_g
=\int_{S\Sigma}(\pi_\Sigma^*Q_4\mathbf b)\alpha\wedge du\wedge du_*.
\end{equation}
One can think of the right-hand side of~\eqref{e:PaSa}
as the integral of $\pi_\Sigma^*Q_4\mathbf b$ against
a \emph{Patterson--Sullivan distribution} $\alpha\wedge du\wedge du_*$ (note that
this distribution is invariant under the geodesic flow) and
the left-hand side of~\eqref{e:PaSa} as a \emph{topological} quantity because
it features harmonic 1-forms. Then~\eqref{e:PaSa} bears some similarity
to the result of Anantharaman--Zelditch~\cite[Theorem~1.1]{Anantharaman-Zelditch-07}
for the symbol $a:=\pi_\Sigma^*\mathbf b$;
the latter is in the setting when $\Sigma$ is a surface and the left-hand side
there has a \emph{spectral} interpretation because it features an eigenfunction of the Laplacian.
However, the operator $L_r$ used in~\cite{Anantharaman-Zelditch-07} is different
in nature from the operator $Q_4$ featured in~\eqref{e:PaSa}:
for our application is crucial that the right-hand side of~\eqref{e:PaSa}
depends only on the pushforward of $\alpha\wedge du\wedge du_*$ to~$\Sigma$ and
that does not seem to typically be the case for the right-hand side of~\cite[Theorem~1.1]{Anantharaman-Zelditch-07}.
See also the work of Hansen--Hilgert--Schr\"oder~\cite{Hansen-Hilgert-Schroder} giving an asymptotic statement
for higher dimensional situations.

The formula~\eqref{e:PaSa} in the special case $\mathbf b\equiv 1$
(which is trivial in our situation because both sides are equal to~0) also has some similarity to the pairing formulas of Dyatlov--Faure--Guillarmou~\cite[Lemma~5.10]{dyatlov-faure-guillarmou-15} and Guillarmou--Hilgert--Weich~\cite[Theorem~5]{guillarmou-hilgert-weich-18}.
In this vague analogy between Theorem~\ref{t:main-identity}
and the results of~\cite{Anantharaman-Zelditch-07,dyatlov-faure-guillarmou-15,guillarmou-hilgert-weich-18} our setting would correspond to an exceptional value of the spectral parameter:
comparing~\eqref{e:mi-lhs-2} with~\cite[(1.3)]{Anantharaman-Zelditch-07} gives the value $\mathbf s=-2$ (in the notation
of~\cite{Anantharaman-Zelditch-07}).

\medskip

Together with Proposition~\ref{prop:harmoniconeforms}, Theorem~\ref{t:main-identity} gives
the following statement which is used in the proof of Proposition~\ref{l:nondeg-metric}.
Recall the map $\mathcal J(x,v)=(x,-v)$ defined in~\eqref{e:J-def}.
\begin{corr}
  \label{c:key-corr}
Assume that $u\in\Res^1_0$ is real-valued and $du\neq 0$. Then
$\pi_{\Sigma*}^{}(\alpha\wedge du\wedge \mathcal J^*(du))\neq 0$.
\end{corr}
\begin{proof}
Put $u_*=\mathcal J^*u\in\Res^1_{0*}$. 
By~\eqref{eq:J-prop} and~\eqref{e:J-push-forward} we have $\sigma_+=\sigma_-$
where the 1-forms $\sigma_\pm$ are defined in~\eqref{e:omega-u-def}.
By Lemma~\ref{l:harmonic-push},
$\sigma=\sigma_+=\sigma_-$
is a real-valued harmonic 1-form on~$\Sigma$,
and $du\neq 0$ implies that $\sigma\neq 0$.

Let $F$ be defined in~\eqref{e:F-def}, then
by Theorem~\ref{t:main-identity} we have
\begin{equation}
  \label{e:endie-2}
Q_4F=-\textstyle{1\over 6}\Delta_g|\sigma|_g^2.
\end{equation}
Now, by Proposition~\ref{prop:harmoniconeforms} we see that $|\sigma|_g^2$ is not constant, that is $\Delta_g|\sigma|_g^2\neq 0$. Therefore, $Q_4F\neq 0$ which implies that $F\neq 0$.
\end{proof}

\subsection{Preliminary steps}

We first prove several preliminary statements. We will
use the hyperboloid model of~\S\ref{s:hyp-3-basics}.

\subsubsection{Hyperbolic Laplacian}

We first write the Laplacian $\Delta_g$ of the hyperbolic metric on~$\mathbb H^3$
using the hyperboloid model. Consider the open cone
$$
\mathcal C_+:=\{(\tilde x_0,\tilde x')\in \mathbb R^{1,3}\colon \tilde x_0>|\tilde x'|\}.
$$
Each point $\tilde x\in \mathcal C_+$ can be written in polar coordinates as
$$
\tilde x=rx,\quad
r>0,\quad
x\in\mathbb H^3.
$$
Define the d'Alembert operator on $\mathcal C_+$ as $\Box=\partial_{\tilde x_0}^2-\partial_{\tilde x_1}^2-
\partial_{\tilde x_2}^2-\partial_{\tilde x_3}^2$. In polar coordinates it can be written as
\begin{equation}
  \label{e:wave-radial}
\Box=r^{-2}\big((r\partial_r)^2+2r\partial_r-\Delta_g\big)
\end{equation}
where the hyperbolic Laplacian $\Delta_g$ acts in the~$x$ variable.

Using~\eqref{e:wave-radial}, we derive the following
useful identity: for any $\psi\in C^\infty((0,\infty))$ and $y\in\mathbb H^3$
\begin{equation}
  \label{e:laplace-conv}
-\Delta_g\psi(\langle x,y\rangle_{1,3})=\widetilde\psi(\langle x,y\rangle_{1,3})\quad\text{where}\quad
\widetilde\psi(\rho):=(1-\rho^2)\psi''(\rho)-3\rho\psi'(\rho)
\end{equation}
and the operator $\Delta_g$ acts in the $x$ variable (note that $\widetilde{\psi}(\rho)$ is given by the radial part of $-\Delta_g$ applied to $\psi(\rho)$ by \eqref{e:dist-h3}).
Indeed, it suffices to apply~\eqref{e:wave-radial} to the
function $f(\tilde x):=\psi(\langle \tilde x,y\rangle_{1,3})$,
$\tilde x\in\mathcal C_+$, and use that
$\Box f(\tilde x)=\psi''(\langle\tilde x,y\rangle_{1,3})$.
Taking in particular $\psi(\rho)=\rho^{-s}$ where $s\in\mathbb C$, we get
\begin{equation}
  \label{e:laplace-s}
\big(-\Delta_g-s(2-s)\big)\langle x,y\rangle_{1,3}^{-s}=s(s+1)\langle x,y\rangle_{1,3}^{-s-2}.
\end{equation}
Similarly, if $\nu_-,\nu_+\in\mathbb S^2\subset\mathbb R^3$, then by applying~\eqref{e:wave-radial} to the function
$$
f_{\nu_-,\nu_+}(\tilde x)=\big(\langle \tilde x,(1,\nu_-)\rangle_{1,3}\,\langle \tilde x,(1,\nu_+)\rangle_{1,3}\big)^{-1},\quad\tilde x\in \mathcal C_+
$$
and using that $\Box f_{\nu_-,\nu_+}=2(1-\nu_-\cdot\nu_+)f_{\nu_-,\nu_+}^2$, where we recall `$\cdot$' denotes the Euclidean inner product, we get
\begin{equation}
  \label{e:laplace-poisson}
-\Delta_g\big( P(x,\nu_-)P(x,\nu_+)\big)=2(1-\nu_-\cdot\nu_+)\big(P(x,\nu_-)P(x,\nu_+)\big)^2
\end{equation}
where the Poisson kernel $P(x,\nu)$ is defined in~\eqref{e:Poisson-def}
and the Laplacian $\Delta_g$ acts in the $x$ variable.

\subsubsection{Properties of the operators $Q_s$}
  \label{s:Q-s-prop}

Let $Q_s:\CIc(\mathbb H^3)\to C^\infty(\mathbb H^3)$ be the operator
defined in~\eqref{e:Q-s-def}. Using~\eqref{e:dist-h3} we can rewrite it as
\begin{equation}
  \label{e:Q-s-other-form}
Q_s f(x)=\int_{\mathbb H^3} \langle x,y\rangle_{1,3}^{-s}f(y)\,d\vol_g(y).
\end{equation}
Note that the operator $Q_s$ is equivariant under the action of the group
$\SO_+(1,3)$:
\begin{equation}
  \label{e:Q-s-equiv}
Q_s (\gamma^*f)=\gamma^* (Q_sf)\quad\text{for all}\quad
\gamma\in \SO_+(1,3).
\end{equation}
For $s>2$, the function $y\mapsto \langle x,y\rangle_{1,3}^{-s}$
lies in $L^1(\mathbb H^3;d\vol_g)$ and its $L^1$ norm is independent
of~$x$; indeed, using the $\SO_+(1,3)$-invariance we may reduce
to the case $x=(1,0,0,0)$, which can be handled by an explicit computation. Therefore,
$Q_s:L^\infty(\mathbb H^3)\to L^\infty(\mathbb H^3)$.

The space $L^\infty(\Sigma)$ is isomorphic to the space of $\Gamma$-invariant functions
in~$L^\infty(\mathbb H^3)$. Using~\eqref{e:Q-s-equiv}, we see that
$Q_s$ descends to the quotient $\Sigma=\Gamma\backslash\mathbb H^3$ as an operator
\begin{equation}
  \label{e:Q-s-Sigma-1}
Q_s:L^\infty(\Sigma)\to L^\infty(\Sigma),\quad
s>2.
\end{equation}
Next, using~\eqref{e:laplace-s}, we get the following identity relating the operators $Q_s$ with the hyperbolic Laplacian $\Delta_g$ on~$\Sigma$:
\begin{equation}
  \label{e:Q-s-Sigma-2}
(-\Delta_g-s(2-s))Q_s=Q_s(-\Delta_g-s(2-s))=s(s+1)Q_{s+2}.
\end{equation}
Putting together~\eqref{e:Q-s-Sigma-1} and~\eqref{e:Q-s-Sigma-2}
and using elliptic regularity, we see that for any $s>2$, $Q_s$ in fact extends to a smoothing operator $\mathcal D'(\Sigma)\to C^\infty(\Sigma)$, proving~\eqref{e:Q-s-intro}.

We now show that for $f\in\mathcal D'(\Sigma)$ one can obtain $Q_sf$ as
a limit of cutoff integrals:
\begin{lemm}
  \label{l:Q-s-regularization}
Fix a cutoff function $\chi(\rho)\in\CIc(\mathbb R)$ such that $\chi=1$ near~$0$.
For $\varepsilon>0$ and $s>2$, define the operator
$$
Q_{s,\chi,\varepsilon}:\mathcal D'(\mathbb H^3)\to C^\infty(\mathbb H^3),\quad
Q_{s,\chi,\varepsilon}f(x)=\int_{\mathbb H^3}\chi(\varepsilon\langle x,y\rangle_{1,3})
\langle x,y\rangle_{1,3}^{-s}f(y)\,d\vol_g(y).
$$
Note that $Q_{s,\chi,\varepsilon}$ satisfies the equivariance relation~\eqref{e:Q-s-equiv}
and thus descends to an operator $\mathcal D'(\Sigma)\to C^\infty(\Sigma)$. Then we have
for all $f\in\mathcal D'(\Sigma)$
\begin{equation}
  \label{e:Q-s-regularization}
Q_{s,\chi,\varepsilon}f\to Q_s f\quad\text{in}\quad C^\infty(\Sigma)\quad\text{as}\quad
\varepsilon\to +0.
\end{equation}
\end{lemm}
\begin{proof}
It suffices to show that for all $n\geq 0$,
$$
\|\Delta_g^n (Q_s-Q_{s,\chi,\varepsilon})\Delta_g^n\|_{L^\infty(\Sigma)\to L^\infty(\Sigma)}
\to 0\quad\text{as}\quad \varepsilon\to +0.
$$
By~\eqref{e:laplace-conv} with $\psi(\rho):=\rho^{-s}(1-\chi(\varepsilon\rho))$ we have
(with each instance of $\Delta_g$ in $\Delta_g^{2n}$ below acting in either $x$ or $y$)
$$
\Delta_g^{2n}\big(\langle x,y\rangle_{1,3}^{-s}(1-\chi(\varepsilon\langle x,y\rangle_{1,3}))\big)
=\langle x,y\rangle_{1,3}^{-s}\psi_{s,\chi,\varepsilon}^{(n)}(\langle x,y\rangle_{1,3})
$$
where, putting $T_s:=\rho^s\big((1-\rho^2)\partial_\rho^2-3\rho\partial_\rho)\rho^{-s}$,
\begin{equation}
  \label{e:QSR-1}
\psi_{s,\chi,\varepsilon}^{(n)}(\rho):=T_s^{2n}(1-\chi(\varepsilon\bullet))(\rho).
\end{equation}
For any $f\in L^\infty(\mathbb H^3)$ we have (integrating by parts in~$y$ and using the fact that $\Delta_g$ is formally self-adjoint)
$$
\Delta_g^n (Q_s-Q_{s,\chi,\varepsilon})\Delta_g^n f(x)
=\int_{\mathbb H^3} \langle x,y\rangle_{1,3}^{-s}\psi_{s,\chi,\varepsilon}^{(n)}(\langle x,y\rangle_{1,3})f(y)\,d\vol_g(y).
$$
Estimating the $L^\infty_xL^1_y$ norm of the integral kernel of the latter operator we get
for any $\delta\in (0,s-2)$ (we will use that $\delta > 0$ at the end of the proof) and
for some $C_{s,\delta}>0$ depending only on~$s,\delta$
\begin{equation}
  \label{e:QSR-2}
\|\Delta_g^n (Q_s-Q_{s,\chi,\varepsilon})\Delta_g^n\|_{L^\infty(\Sigma)\to L^\infty(\Sigma)}
\leq C_{s,\delta}\sup_{\rho\geq 1}|\rho^{-\delta}\psi_{s,\chi,\varepsilon}^{(n)}(\rho)|.
\end{equation}
For $k\in\mathbb N_0$ and $\psi\in C^\infty((0,\infty))$, define the seminorm
$$
\|\psi\|_{\delta,k}:=\max_{0\leq j\leq k}\sup_{\rho\geq 1}|\rho^{-\delta}(\rho\partial_\rho)^j \psi(\rho)|.
$$
We have $\|T_s\psi\|_{\delta,k}\leq C_{s,\delta,k}\|\psi\|_{\delta,k+2}$. Therefore
\begin{equation}
\sup_{\rho\geq 1}|\rho^{-\delta}\psi_{s,\chi,\varepsilon}^{(n)}(\rho)|
\leq C_{s,\delta,n}\|1-\chi(\varepsilon \rho)\|_{\delta,4n}
=\mathcal O(\varepsilon^\delta),
\end{equation}
which finishes the proof.
\end{proof}

\subsubsection{Spherical convolution operators}

Let $\kappa\in C^\infty([0,4])$. Define the smoothing operator
\begin{equation}
  \label{e:A-psi-def}
A_\kappa:\mathcal D'(\mathbb S^2)\to C^\infty(\mathbb S^2),\quad
A_\kappa f(\nu)=\int_{\mathbb S^2}\kappa(|\nu-\nu'|^2)f(\nu')\,dS(\nu').
\end{equation}
Here $|\nu-\nu'|$ denotes the Euclidean distance between the points $\nu,\nu'\in\mathbb S^2\subset\mathbb R^3$.

In this section we prove an estimate on the norm of $A_\kappa$ between
Sobolev spaces, Lemma~\ref{l:A-psi-bound}, which is used in the regularization argument in~\S\ref{s:regularization} below.
Before we state this estimate, we establish a few basic properties of $A_\kappa$:
\begin{lemm}
  \label{l:A-psi-L2}
We have
$$
\|A_\kappa\|_{L^2(\mathbb S^2)\to L^2(\mathbb S^2)}
\leq \pi \|\kappa\|_{L^1([0,4])}.
$$  
\end{lemm}
\begin{proof}
By Schur's lemma we have
$$
\|A_\kappa\|_{L^2(\mathbb S^2)\to L^2(\mathbb S^2)}
\leq\sup_{\nu'\in\mathbb S^2}\int_{\mathbb S^2}\big|\kappa(|\nu-\nu'|^2)\big|\,dS(\nu).
$$
By $\SO(3)$-invariance we see that the integral above is independent of~$\nu'$.
Choose $\nu'=(0,0,-1)$ and use spherical coordinates
$\nu=(\sin\theta\cos\varphi,\sin\theta\sin\varphi,\cos\theta)$ to compute
$$
\int_{\mathbb S^2}\big|\kappa(|\nu-\nu'|^2)\big|\,dS(\nu)
=2\pi\int_0^{\pi}\big|\kappa(2+2\cos\theta)\big|\sin\theta\,d\theta
=\pi\int_0^4 |\kappa(r)|\,dr
$$
which finishes the proof.
\end{proof}
%
\begin{lemm}
  \label{l:A-psi-laplace}
Denote by $\Delta_{\mathbb S^2}$ the (nonpositive) Laplace--Beltrami operator
on $\mathbb S^2$. Then
\begin{equation}
\label{eq:commuteA_F}
A_\kappa\Delta_{\mathbb S^2}=\Delta_{\mathbb S^2}A_\kappa=A_{\tilde\kappa},\quad
\tilde\kappa(r):=(4-r)r\kappa''(r)+(4-2r)\kappa'(r).
\end{equation}
\end{lemm}
\begin{proof}
It is enough to show that, with $\Delta_{\mathbb S^2}$ acting in the $\nu$ variable,
$$
\Delta_{\mathbb S^2}(\kappa(|\nu-\nu'|^2))=\tilde\kappa(|\nu-\nu'|^2).
$$
Similarly to the proof of Lemma~\ref{l:A-psi-L2}, by $\SO(3)$-invariance we may reduce to the case $\nu'=(0,0,-1)$ and take spherical coordinates $(\theta,\varphi)$ for $\nu$, in which the Laplace operator is $\Delta_{\mathbb S^2}=(\sin\theta)^{-1}\partial_\theta\sin\theta\partial_\theta+(\sin\theta)^{-2}\partial_\varphi^2$ and $|\nu-\nu'|^2=2+2\cos\theta$. Then we compute
$$
\begin{aligned}
\Delta_{\mathbb S^2}(\kappa(|\nu-\nu'|^2))&={1\over\sin\theta}\partial_\theta\sin\theta\partial_\theta \kappa(2+2\cos\theta)\\
&=4\sin^2\theta\kappa''(2+2\cos\theta)-4\cos\theta\kappa'(2+2\cos\theta)\\
&=\tilde\kappa(2+2\cos\theta)
\end{aligned}
$$
which finishes the proof.
\end{proof}
We can now give
\begin{lemm}
  \label{l:A-psi-bound}
Assume that $s_1,s_2\in\mathbb R$ and $s_2-s_1=2\ell$ for some $\ell\in\mathbb N_0$.
Then there exists a constant $C$ depending only on $s_1,s_2$ such that
for all $\kappa\in C^\infty([0,4])$
\begin{equation}
  \label{e:A-psi-bound}
\|A_\kappa\|_{H^{s_1}(\mathbb S^2)\to H^{s_2}(\mathbb S^2)}\leq C\sum_{j=0}^{2\ell}\|r^{\max(j-\ell,0)}\partial_r^j\kappa(r)\|_{L^1([0,4])}.
\end{equation}
\end{lemm}
\begin{proof}
Define the differential operator arising from \eqref{eq:commuteA_F} (corresponding to $1 - \Delta_{\mathbb{S}^2}$)
$$
W:=(r-4)r\partial_r^2+(2r-4)\partial_r+1.
$$
Denote by $C$ a constant depending only on $s_1,s_2$, whose precise value may change from line to line. We have
$$
\begin{aligned}
\|A_\kappa\|_{H^{s_1}(\mathbb S^2)\to H^{s_2}(\mathbb S^2)}
&\leq C\|(1-\Delta_{\mathbb S^2})^{s_2/2}A_\kappa (1-\Delta_{\mathbb S^2})^{-s_1/2}\|_{L^2(\mathbb S^2)\to L^2(\mathbb S^2)}\\
&= C\|(1-\Delta_{\mathbb S^2})^\ell A_\kappa\|_{L^2(\mathbb S^2)\to L^2(\mathbb S^2)}\\
&= C\|A_{W^\ell\kappa}\|_{L^2(\mathbb S^2)\to L^2(\mathbb S^2)}\\
&\leq C\|W^\ell\kappa\|_{L^1([0,4])}.
\end{aligned}
$$
Here in the second equality we used that $A_\kappa$ commutes with $\Delta_{\mathbb S^2}$ by Lemma~\ref{l:A-psi-laplace}. In the third inequality we used Lemma~\ref{l:A-psi-laplace}
again. In the last inequality we used Lemma~\ref{l:A-psi-L2}.

By induction in $\ell$ we see that $W^\ell$ is a linear combination
with constant coefficients of the operators $r^k\partial_r^j$ where
$0\leq j\leq 2\ell$ and $k\geq \max(j-\ell,0)$. Therefore,
$\|W^\ell\kappa\|_{L^1([0,4])}$ is bounded by the right-hand side of~\eqref{e:A-psi-bound},
which finishes the proof.
\end{proof}

\subsection{Proof of Theorem~\ref{t:main-identity}}

Here we give the proof of Theorem~\ref{t:main-identity}, proceeding in several steps.
In~\S\ref{s:reduction-to-boundary} we write both sides of~\eqref{e:main-identity}
as integrals featuring some distributions $g_\pm$ on $\mathbb S^2$.
In~\S\ref{s:change-of-variables} we introduce a change of variables
which shows that the two integrals are formally equal. In~\S\ref{s:regularization}
we prove that regularized versions of the two integrals are equal and
show convergence of the regularization to finish the proof.

Denote by $\pi_\Gamma$ the covering maps $\mathbb H^3\to \Sigma$ and $S\mathbb H^3\to M=S\Sigma$ (which one is meant will be clear from the context).
Since we can choose the representation of $\Sigma$ as the quotient $\Gamma\backslash\mathbb H^3$ arbitrarily, for any given $x\in\Sigma$ we may arrange that $\pi_\Gamma(e_0)=x$ where
\begin{equation}
  \label{e:e-0}
e_0:=(1,0,0,0)\in\mathbb H^3.
\end{equation}
Therefore, in order to prove Theorem~\ref{t:main-identity} it suffices to consider the case $x=\pi_\Gamma(e_0)$, i.e. to show that
\begin{equation}
  \label{e:main-identity-2}
\pi_\Gamma^* Q_4 F(e_0)=-\textstyle{1\over 6}\pi_\Gamma^*\Delta_g(\sigma_-\cdot\sigma_+)(e_0).
\end{equation}

\subsubsection{Reduction to the conformal boundary}
\label{s:reduction-to-boundary}

We first express both sides of~\eqref{e:main-identity-2} in terms of some distributions $g_\pm$ on the conformal boundary $\mathbb S^2$. 

Let $u\in\Res^1_0$, $u_*\in\Res^1_{0*}$. By Lemma~\ref{l:du-dv} we have
$$
du=f_-\omega_-,\quad
du_*=f_+\omega_+,\quad
\alpha\wedge du\wedge du_*=-\textstyle{1\over 8}f_-f_+ d\vol_\alpha,
$$
where by~\eqref{e:f-pm-shape}, the lifts of $f_-\in\mathcal D'_{E_u^*}(M;\mathbb C)$,
$f_+\in\mathcal D'_{E_s^*}(M;\mathbb C)$ to the covering
space $S\mathbb H^3$ have the form
(recalling the definitions~\eqref{e:B-pm} of $\Phi_\pm$, $B_\pm$) 
\begin{equation}
  \label{e:f-pm-again}
\pi_\Gamma^*f_\pm=\Phi_\pm^{-2} (g_\pm\circ B_\pm)\quad\text{for some}\quad
g_\pm\in\mathcal D'(\mathbb S^2;\mathbb C).
\end{equation}
Arguing similarly to~\eqref{e:pushforward-volume}, we see that
the distribution $F\in\mathcal D'(\Sigma;\mathbb C)$ defined in~\eqref{e:F-def}
can be written as the pushforward
$$
F(x)={1\over 4}\int_{S_x\Sigma} f_-(x,v)f_+(x,v)\,dS(v),\quad
x\in\Sigma
$$
where $dS$ is the canonical volume form on the spherical fiber $S_x\Sigma$.
Therefore, the lift of~$F$ to~$\mathbb H^3$ has the form
\begin{equation}
  \label{e:F-lift}
\pi_\Gamma^*F(x)={1\over 4}\int_{S_x\mathbb H^3}
\big(\Phi_-(x,v)\Phi_+(x,v)\big)^{-2}g_-(B_-(x,v))g_+(B_+(x,v))\,dS(v).
\end{equation}
We next express the harmonic 1-forms $\sigma_\pm$ defined in~\eqref{e:omega-u-def}
in terms of the distributions~$g_\pm$:	
\begin{lemm}
  \label{l:omega-u-form}
Using the hyperbolic metric, identify the pullbacks $\pi_\Gamma^*\sigma_\pm$
with vector fields on~$\mathbb H^3$. Then for any $x\in\mathbb H^3$
$$
\pi_\Gamma^*\sigma_\pm(x)={1\over 4}\int_{\mathbb S^2} g_\pm(\nu)v_\pm(x,\nu)\,dS(\nu)
$$
where $v_\pm(x,\nu)\in S_x\mathbb H^3\subset T_x\mathbb H^3$ is defined in~\eqref{eq:xi_pmdef}.
\end{lemm}
\begin{proof}
By~\eqref{e:pushforward-nm-2} and since $du=f_-\omega_-$, $du_*=f_+\omega_+$ we have
$$
\sigma_\pm=\pi_{\Sigma*}^{}(f_\pm\alpha\wedge\omega_\pm).
$$
Recall the horizontal/vertical decomposition~\eqref{e:hv-map}.
For any $(x,v)\in M=S\Sigma$, $\xi = (\xi_H, \xi_V) \in T_{(x,v)}M$, and a positively oriented $g$-orthonormal basis $v,v_1,v_2\in T_x\Sigma$ we compute by~\eqref{e:alpha-hv} and~\eqref{e:omega-pm-1}
$$
(\alpha\wedge\omega_\pm)(x,v)(\xi,(0,v_1),(0,v_2))=\textstyle{1\over 4}\langle \xi_H,v\rangle_g.
$$
Using the metric~$g$, we identify $\sigma_\pm$ with a vector field on~$\Sigma$.
Then 
$$
\sigma_\pm(x)={1\over 4}\int_{S_x\Sigma}f_\pm(x,v)v\,dS(v),\quad
x\in\Sigma.
$$
It follows that for each $x\in\mathbb H^3$
$$
\begin{aligned}
\pi_\Gamma^*\sigma_\pm(x)&={1\over 4}\int_{S_x\mathbb H^3}
\Phi_\pm(x,v)^{-2}g_\pm(B_\pm(x,v)) v\,dS(v)\\
&={1\over 4}\int_{\mathbb S^2} g_\pm(\nu)v_\pm(x,\nu)\,dS(\nu).
\end{aligned}
$$
Here in the first equality we used~\eqref{e:f-pm-again}.
In the second equality we made the change of variables
$\nu=B_\pm(x,v)$ and used~\eqref{eq:jacobianxi_pm}.
\end{proof}
We note that by the preceding lemma $v_\pm(x, \nu)$ define vector-valued Poisson kernels in the sense of \cite{olbrichdiss,kuster-weich2}. From Lemma~\ref{l:omega-u-form} we get the following
formula for the right-hand side of~\eqref{e:main-identity-2}
in terms of the distributions~$g_\pm$:
\begin{lemm}
  \label{l:omega-u-form-2}
We have (here $e_0$ is defined in~\eqref{e:e-0})
\begin{equation}
  \label{e:omega-u-form-2}
-\pi_\Gamma^*\Delta_g(\sigma_-\cdot\sigma_+)(e_0)
={1\over 8}\int_{\mathbb S^2\times\mathbb S^2}
(1-\nu_-\cdot\nu_+)^2g_-(\nu_-)g_+(\nu_+)\,dS(\nu_-)dS(\nu_+).
\end{equation}
\end{lemm}
\begin{proof}
By~\eqref{eq:xi_pmdef} we have for each $\nu_-,\nu_+\in\mathbb S^2$
and $x\in\mathbb H^3$
$$
\langle v_-(x,\nu_-),v_+(x,\nu_+)\rangle_g=-\langle v_-(x,\nu_-),v_+(x,\nu_+)\rangle_{1,3}
=P(x,\nu_-)P(x,\nu_+)(1-\nu_-\cdot\nu_+)-1.
$$
With the hyperbolic Laplacian $\Delta_g$ acting in the~$x$ variable, we then
compute by~\eqref{e:laplace-poisson}
$$
-\Delta_g\langle v_-(x,\nu_-),v_+(x,\nu_+)\rangle_g=2(1-\nu_-\cdot\nu_+)^2\big(P(x,\nu_-)P(x,\nu_+)\big)^2.
$$
Now~\eqref{e:omega-u-form-2} follows from Lemma~\ref{l:omega-u-form} by integration
and using that $P(e_0,\nu_\pm)=1$ by~\eqref{e:Poisson-def}.
\end{proof}

\subsubsection{Change of variables}
\label{s:change-of-variables}

By~\eqref{e:F-lift} and~\eqref{e:Q-s-other-form} we can formally write the left-hand side of~\eqref{e:main-identity-2} as follows:
\begin{equation}
  \label{e:mi-lhs-1}
\begin{gathered}
\pi_\Gamma^*Q_4 F(e_0)\\={1\over 4}\int_{S\mathbb H^3}y_0^{-4}
\big(\Phi_-(y,v)\Phi_+(y,v)\big)^{-2}g_-(B_-(y,v))g_+(B_+(y,v))
dS(v)d\vol_g(y),
\end{gathered}
\end{equation}
where we recall $y = (y_0, y_1, y_2, y_3) \in \mathbb{H}^3$. Note that one has to take care when defining the integral above, as $g_\pm$ are distributions
and $S\mathbb H^3$ is noncompact, see~\S\ref{s:regularization} below.

On the other hand, the right-hand side of~\eqref{e:main-identity-2}
can be expressed using~\eqref{e:omega-u-form-2} as an
integral over $(\nu_-,\nu_+)\in\mathbb S^2\times \mathbb S^2$.
To prove~\eqref{e:main-identity-2} and relate the two integrals we will use
the change of variables $\Xi:(y,v)\mapsto (\nu_-,\nu_+,t)$,
where $t\in\mathbb R$, introduced in \eqref{e:Xi-def}. The basic properties of~$\Xi$ are collected below in
\begin{lemm}
  \label{l:Xi-prop}
1. Let $(\nu_-,\nu_+,t)=\Xi(y,v)$. Then
\begin{align}
\label{e:Xi-1}
\Phi_-(y,v)\Phi_+(y,v)&={4\over |\nu_--\nu_+|^2}={2\over 1-\nu_-\cdot\nu_+},
\\
y_0&={2\cosh t\over |\nu_--\nu_+|}.
\label{e:Xi-2}
\end{align}
(As before, we write elements of $\mathbb H^3$ as $y=(y_0,y_1,y_2,y_3)\in\mathbb R^{1,3}$.)

2. The Jacobian of $\Xi$ at $(y,v)$ with respect to the densities
$d\vol_g(y)dS(v)$ and $dS(\nu_-)dS(\nu_+)dt$ is equal to~$4\big(\Phi_-(y,v)\Phi_+(y,v)\big)^{-2}$.
\end{lemm}
\Remark The identity in part~2 of the above is well-known, see~\cite[Theorem~8.1.1 on p.~131]{Nicholls}.
\begin{proof}
1. The identity~\eqref{e:Xi-1} follows immediately from~\eqref{e:Phi-pm-B-pm-id},
noting that $|\nu_--\nu_+|^2=2(1-\nu_-\cdot\nu_+)$.
To see~\eqref{e:Xi-2}, we compute by~\eqref{e:Xi-1} and~\eqref{e:Xi-def}
$$
\Phi_\pm(y,v)={2e^{\pm t}\over |\nu_--\nu_+|}
$$
which by~\eqref{e:B-pm-def} gives
$$
y_0={\Phi_-(y,v)+\Phi_+(y,v)\over 2}
={2\cosh t\over |\nu_--\nu_+|}.
$$

\noindent 2. Take $(y,v)\in S\mathbb H^3$. Let $w\in T_y\mathbb H^3$
satisfy~$\langle v,w\rangle_{1,3}=0$. Then
\begin{equation}\label{e:B-pm-length}
|dB_\pm(y,v)(w,\pm w)|_{\mathbb S^2}=2|dB_\pm(y,v)(0,w)|_{\mathbb S^2}={2|w|_g\over\Phi_\pm(y,v)}.
\end{equation}
Here in the first equality we write $(w,\pm w)=(w,\mp w)\pm 2(0,w)$ and use
that by~\eqref{e:B-pm-fibers}, $dB_\pm(y,v)(w,\mp w)=0$.
In the second equality we use~\eqref{eq:jacobianxi_pm}. Denoting by $X$
the generator of the geodesic flow and defining~$t$ by~\eqref{e:Xi-def}, we also have by~\eqref{e:X-Phi-pm} and~\eqref{e:B-pm-fibers}
$$
dB_\pm(y,v)(X(y,v))=0,\quad
dt(X(y,v))=1.
$$
Fix a $g$-orthonormal basis $v,v_1,v_2$ of $T_y\mathbb H^3$ and consider
the following basis of $T_{(y,v)}S\mathbb H^3$:
$$
\xi_0=X(y,v),\quad
\xi_1^\pm=(v_1,\pm v_1),\quad
\xi_2^\pm=(v_2,\pm v_2).
$$
Since $\xi_j^-\wedge \xi_j^+=2 (v_j,0)\wedge (0,v_j)$, the value of the density $d\vol_g(y)dS(v)$ on $\xi_0,\xi_1^-,\xi_2^-,\xi_1^+,\xi_2^+$
is equal to~4. On the other hand, writing
$(\eta_-(\xi),\eta_+(\xi),\tau(\xi))=d\Xi(y,v)(\xi)$, we have
$$
\eta_\pm(\xi_j^\mp)=\eta_\pm(\xi_0)=0,\quad
\tau(\xi_0)=1
$$
and the vectors $\eta_\pm(\xi_1^\pm),\eta_\pm(\xi_2^\pm)$
are orthogonal to each other and have length $2\Phi_\pm(y,v)^{-1}$ each by \eqref{e:B-pm-length}.
It follows that the value of the density $dS(\nu_-)dS(\nu_+)dt$
on the images of $\xi_0,\xi_1^-,\xi_2^-,\xi_1^+,\xi_2^+$ under
$d\Xi(y,v)$ is equal to~$16\big(\Phi_-(y,v)\Phi_+(y,v)\big)^{-2}$.
Thus the Jacobian of $\Xi$ at $(y,v)$ is equal to~$4\big(\Phi_-(y,v)\Phi_+(y,v)\big)^{-2}$.
\end{proof}
Using Lemma~\ref{l:Xi-prop} and~\eqref{e:mi-lhs-1}, we can formally write
the left-hand side of~\eqref{e:main-identity-2} as
\begin{equation}
  \label{e:mi-lhs-2}
\pi_\Gamma^*Q_4F(e_0)
={1\over 64}\int\limits_{(\mathbb S^2\times\mathbb S^2)_-\times \mathbb R}
{(1-\nu_-\cdot\nu_+)^2\over\cosh^4t}g_-(\nu_-)g_+(\nu_+)
\,dS(\nu_-)dS(\nu_+)dt.
\end{equation}
Using the change of variables $s=\tanh t$, we compute
\begin{equation}
  \label{e:cosh-integral}
\int_{\mathbb R}{dt\over\cosh^4t}=\int_{-1}^1 (1-s^2)\,ds={4\over 3}.
\end{equation}
Comparing~\eqref{e:mi-lhs-2} with~\eqref{e:omega-u-form-2}, we formally obtain the identity~\eqref{e:main-identity-2}. However, our argument is incomplete since the integrals in~\eqref{e:mi-lhs-1} and~\eqref{e:mi-lhs-2} are over the noncompact manifolds $S\mathbb H^3$, $(\mathbb S^2\times\mathbb S^2)_-\times\mathbb R$ and $g_\pm$ are distributions.
Thus one cannot immediately apply the change of variables formula to get~\eqref{e:mi-lhs-2}
from~\eqref{e:mi-lhs-1}, or Fubini's Theorem to get~\eqref{e:main-identity-2} from~\eqref{e:mi-lhs-2}.
To deal with these issues, we will employ a regularization procedure.

\subsubsection{Regularization and end of the proof}
\label{s:regularization}

Fix a cutoff function
$$
\chi\in \CIc(\mathbb R;[0,1]),\quad
\supp\chi\subset [-2,2],\quad
\chi|_{[-1,1]}=1.
$$
For $\varepsilon>0$, define the integral
$$
I_\varepsilon:=\int_{\mathbb H^3} \chi(\varepsilon y_0) y_0^{-4} \pi_\Gamma^*F(y)\,d\vol_g(y).
$$
(As before, we embed $\mathbb H^3$ into $\mathbb R^{1,3}$ and
we have $y_0=\langle e_0,y\rangle_{1,3}$ where $e_0=(1,0,0,0)$.)
By Lemma~\ref{l:Q-s-regularization} with $x=e_0$, $I_\varepsilon$ converges to
the left-hand side of~\eqref{e:main-identity-2}:
\begin{equation}
  \label{e:approxer-1}
I_\varepsilon\to \pi_\Gamma^*Q_4F(e_0)\quad\text{as}\quad \varepsilon\to +0.
\end{equation}
By~\eqref{e:approxer-1} and~\eqref{e:omega-u-form-2}, the proof
of~\eqref{e:main-identity-2} (and thus of Theorem~\ref{t:main-identity})
is finished once we show that
\begin{equation}
  \label{e:approxer-2}
I_\varepsilon\to {1\over 48}\int\limits_{\mathbb S^2\times\mathbb S^2}
(1-\nu_-\cdot\nu_+)^2g_-(\nu_-)g_+(\nu_+)\,
dS(\nu_-)dS(\nu_+)\quad\text{as}\quad \varepsilon\to +0.
\end{equation}
By~\eqref{e:F-lift} we have the following regularized version of~\eqref{e:mi-lhs-1}:
$$
I_\varepsilon={1\over 4}\int_{S\mathbb H^3}\chi(\varepsilon y_0)y_0^{-4}
\big(\Phi_-(y,v)\Phi_+(y,v)\big)^{-2}
g_-(B_-(y,v))g_+(B_+(y,v))
\,dS(v)d\vol_g(y).
$$
Making the change of variables $(\nu_-,\nu_+,t)=\Xi(y,v)$ and using Lemma~\ref{l:Xi-prop}, we then get the following regularized version of~\eqref{e:mi-lhs-2} (we keep in mind that $g_\pm$ are merely distributions so that all of the integrals around these lines are understood in the distributional sense):
$$
I_\varepsilon={1\over 64}\int\limits_{\mathbb S^2\times\mathbb S^2\times \mathbb R}
\chi\Big({2\varepsilon\cosh t\over |\nu_--\nu_+|}\Big)
{(1-\nu_-\cdot\nu_+)^2\over\cosh^4t}g_-(\nu_-)g_+(\nu_+)
\,dS(\nu_-)dS(\nu_+)dt.
$$
For $r\geq 0$, define the function
\begin{equation}
  \label{e:psi-epsilon}
\psi_\varepsilon(r):={3\over 4}\int_{\mathbb R}\chi\Big({2\varepsilon\cosh t\over \sqrt r}\Big)
\cosh^{-4}t\,dt.
\end{equation}
Note that $\psi_\varepsilon\in C^\infty([0,\infty))$ and $\psi_\varepsilon(r)=0$ for $r\ll\varepsilon^2$.
We now have 
\begin{equation}
  \label{e:I-eps-formula}
I_\varepsilon={1\over 48}\int\limits_{\mathbb S^2\times\mathbb S^2}
\psi_\varepsilon(|\nu_--\nu_+|^2)(1-\nu_-\cdot\nu_+)^2g_-(\nu_-)g_+(\nu_+)\,
dS(\nu_-)dS(\nu_+).
\end{equation}
Recalling that $|\nu_--\nu_+|^2=2(1-\nu_-\cdot\nu_+)$, we see from~\eqref{e:I-eps-formula} that it suffices to prove the following version of~\eqref{e:approxer-2}:
\begin{equation}
  \label{e:approxer-3}
\int\limits_{\mathbb S^2\times\mathbb S^2}\big(1-\psi_\varepsilon(|\nu_--\nu_+|^2)\big)
|\nu_--\nu_+|^4g_-(\nu_-)g_+(\nu_+)\,dS(\nu_-)dS(\nu_+)\to 0\quad\text{as}\quad
\varepsilon\to +0.
\end{equation}
If $g_\pm$ were smooth functions on $\mathbb S^2$, then~\eqref{e:approxer-3}
would follow from the Dominated Convergence Theorem since by~\eqref{e:cosh-integral}
we have $\psi_\varepsilon(r)\to 1$ as $\varepsilon\to +0$ for all $r>0$.
However, $g_\pm$ are merely distributions, so one has to be more careful.
We start by establishing the Sobolev regularity of $g_\pm$
by following the standard proof of the Fredholm property in anisotropic
Sobolev spaces.
(We use the proof in~\cite{dyatlov-zworski-16}; one could alternatively
carefully examine the proof in~\cite{faure-sjostrand-11}.)
See the papers of Adam--Baladi~\cite[\S3.3]{Adam-Baladi},
Guillarmou--Poyferr\'e--Bonthonneau~\cite[Appendix~A]{Guillarmou-Poyferre},
and Dyatlov~\cite{Dyatlov-Nir} for a general discussion of Sobolev regularity thresholds
for the Pollicott--Ruelle resolvent.
\begin{lemm}
  \label{l:g-pm-sobolev}
We have $g_\pm \in H^{-2-\delta}(\mathbb S^2)$ for all $\delta>0$.
\end{lemm}
\begin{proof}
We show the regularity of $g_-$, with $g_+$ handled similarly.
Recall that $g_-$ is related to the distribution $f_-\in\mathcal D'_{E_u^*}(M;\mathbb C)$
by~\eqref{e:f-pm-again}. Since $\Phi_-$ is smooth and $B_-$ is a submersion,
it suffices to show that $f_-\in H^{-2-\delta}(M)$.

By Lemma~\ref{l:du-dv}, we have $(X-2)f_-=0$, that is $f_-$ is a Pollicott--Ruelle resonant state for the operator $P=-iX$ corresponding to the resonance $\lambda_0=-2i$, see~\S\ref{s:resonant-states}. Given that Pollicott--Ruelle resonant states are eigenfunctions of $P$ on anisotropic Sobolev spaces (see~\eqref{e:anisotropic-sobolev}), it suffices to show that one can choose the order function $m$ in the definition of the weight~$G(\rho,\xi)=m(\rho,\xi)\log(1+|\xi|)$ such that the Fredholm property~\eqref{e:Fredholm} holds on the anisotropic Sobolev space~$\mathcal H_{G,0}$ for $\Im\lambda\geq -2$ and $\mathcal H_{G,0}\subset H^{-2-\delta}$; the latter is equivalent to requiring that $m\geq -2-\delta$ everywhere.

In~\cite[\S\S3.3--3.4]{dyatlov-zworski-16} the Fredholm property~\eqref{e:Fredholm} is shown
using propagation of singularities and microlocal radial estimates.
Following the proof of~\cite[Proposition~3.4]{dyatlov-zworski-16}, we see that
one only needs to check that the low regularity radial estimate~\cite[Proposition~2.7]{dyatlov-zworski-16} applies to the operator $P-\lambda$ (where $\Im\lambda\geq -2$)
at the radial sink $E_u^*$ (see~\eqref{e:anosov-dual}) in the space $H^{-2-\delta}$.
(The high regularity radial estimate~\cite[Proposition~2.6]{dyatlov-zworski-16}
would apply once $m$ is sufficiently large on $E_s^*$,
which can be arranged.)
The threshold regularity for this estimate is computed in~\cite[Theorem~E.54]{dyatlov-zworski-book}.
In our setting, since the operator $P$ is symmetric on $L^2(M;d\vol_\alpha)$
and it has order~$k=1$, it is enough that
$$
2+(-2-\delta){H_p|\xi|\over |\xi|}<0\quad\text{on}\quad E_u^*
$$
where $p(\rho,\xi)=\langle X(\rho),\xi\rangle$ is the principal symbol of $P$
and its Hamiltonian flow is given by $e^{tH_p}(\rho,\xi)=(\varphi_t(\rho),d\varphi_t^{-T}(\rho)\xi)$, see~\cite[\S3.1]{dyatlov-zworski-16}. Choosing the norm $|\xi|$ induced by the Sasaki metric and using~\eqref{e:sasaki-conformal}, we see that
$$
{H_p|\xi|\over |\xi|}=1\quad\text{on}\quad E_u^*
$$
which means that the threshold regularity condition for the radial estimate is satisfied and the proof is finished.
\end{proof}
Coming back to the proof of~\eqref{e:approxer-3}, we rewrite it as
\begin{equation}
  \label{e:approxer-4}
\langle A_{\kappa_\varepsilon} g_-,\overline{g_+}\rangle_{L^2(\mathbb S^2)}\to 0\quad\text{as}\quad
\varepsilon\to +0
\end{equation}
where the operator $A_{\kappa_\varepsilon}$ is given by~\eqref{e:A-psi-def}:
$$
A_{\kappa_\varepsilon}f(\nu_+)=\int_{\mathbb S^2}\kappa_\varepsilon(|\nu_--\nu_+|^2)f(\nu_-)\,dS(\nu_-)
$$
and the function $\kappa_\varepsilon\in C([0,4])$ is given by
(using~\eqref{e:cosh-integral} and~\eqref{e:psi-epsilon} in the second equality below)
$$
\kappa_\varepsilon(r):={4\over 3}r^2(1-\psi_\varepsilon(r))=r^2\int_{\mathbb R}\bigg(1-\chi\Big({2\varepsilon\cosh t\over\sqrt r}\Big)\bigg)\cosh^{-4} t\,dt.
$$
Using Lemma~\ref{l:g-pm-sobolev}, we have in particular $g_\pm\in H^{-5/2}(\mathbb S^2)$.
Thus to finish the proof of~\eqref{e:approxer-4}, and thus of Theorem~\ref{t:main-identity}, it remains to prove the norm bound
\begin{equation}
  \label{e:approxer-5}
\|A_{\kappa_\varepsilon}\|_{H^{-5/2}(\mathbb S^2)\to H^{5/2}(\mathbb S^2)}\to 0\quad\text{as}\quad\varepsilon\to +0.
\end{equation}
To show~\eqref{e:approxer-5}, we will bound the norms of $A_{\kappa_\varepsilon}$
between Sobolev spaces using Lemma~\ref{l:A-psi-bound}. To do this
we estimate the derivatives of $\kappa_\varepsilon$:
\begin{lemm}
  \label{l:kappa-eps-estimate}
Let $j,k\in\mathbb N_0$. Then there exists $C$ depending only on $j,k$
such that for all $\varepsilon\in (0,1]$
\begin{equation}
  \label{e:L1symbolestimate}
\|r^k\partial^j_r \kappa_\varepsilon(r)\|_{L^1([0,4])}\leq
\begin{cases}
C\varepsilon^4,& k\geq j;\\
C\varepsilon^4\log(1/\varepsilon),& k=j-1;\\
C\varepsilon^{2(3+k-j)},& k\leq j-2.
\end{cases}
\end{equation}
\end{lemm}
\begin{proof}
Throughout the proof we denote by $C$ a constant depending only on $j,k$ whose
precise value might change from line to line.

\noindent 1. 
For any $G(s)\in C^\infty([0,\infty))$ which is constant near $s=\infty$ define
$$
\Phi_G(\tau):=\int_{\mathbb R}G\Big({2\cosh t\over \sqrt \tau}\Big)\cosh^{-4}t\,dt,\quad
\tau> 0.
$$
We have the identity
\begin{equation}
  \label{e:kapest-1}
\tau\partial_\tau\Phi_G=-\textstyle{1\over 2}\Phi_{s\partial_s G}.
\end{equation}
Moreover, we have the estimate
\begin{equation}
  \label{e:kapest-2}
G|_{[-1,1]}=0\quad\Longrightarrow\quad
|\Phi_G(\tau)|\leq {C\|G\|_{L^\infty}\over 1+\tau^2}
\end{equation}
which can be proved by bounding $|\Phi_G(\tau)|$ by $\|G\|_{L^\infty}$
times the integral of $\cosh^{-4}t\,dt$ over the set of $t$ such that
$\cosh t\geq\sqrt{\tau}/2$ and using that
$\int \cosh^{-4}t\,dt=\tanh t-{1\over 3}\tanh^3 t+C$
and $\sqrt{1-\lambda}-{1\over 3}(1-\lambda)^{3/2}={2\over 3}+\mathcal O(\lambda^2)$
as $\lambda={4\over \tau}\to 0$.

\noindent 2. We have
$$
\kappa_\varepsilon(r)=r^2\Phi_{1-\chi}(\varepsilon^{-2}r).
$$
By~\eqref{e:kapest-1} for each $j\geq 0$
$$
\begin{gathered}
(r\partial_r)^j\kappa_\varepsilon(r)
=r^2(r\partial_r+2)^j\big(\Phi_{1-\chi}(\varepsilon^{-2}r)\big)=
r^2\Phi_{G_j}(\varepsilon^{-2}r)\\
\text{where}\quad
G_j(s):=(2-\textstyle{1\over 2}s\partial_s)^j (1-\chi)(s).
\end{gathered}
$$
Since $\chi|_{[-1,1]}=1$, we have $G_j|_{[-1,1]}=0$. Thus by~\eqref{e:kapest-2}
$$
|(r\partial_r)^j\kappa_\varepsilon(r)|\leq {Cr^2\over 1+\varepsilon^{-4}r^2}.
$$
Writing $r^j\partial_r^j$ as a linear combination of $(r\partial_r)^q$
with $0\leq q\leq j$, we get
$$
|\partial_r^j\kappa_\varepsilon(r)|\leq {Cr^{2-j}\over 1+\varepsilon^{-4}r^2}\leq C\varepsilon^4r^{-j}.
$$
Since $\supp\chi\subset [-2,2]$, we have by~\eqref{e:cosh-integral}
$$
\kappa_\varepsilon(r)=\textstyle{4\over 3}r^2\quad\text{for}\quad 0\leq r\leq \varepsilon^2.
$$
Therefore
$$
\|r^k\partial_r^j\kappa_\varepsilon(r)\|_{L^1([0,4])}\leq
C\int_0^{\varepsilon^2}r^k\partial_r^j(r^2)\,dr
+C\varepsilon^4\int_{\varepsilon^2}^4 r^{k-j}\,dr
$$
which gives~\eqref{e:L1symbolestimate}.
\end{proof}
Combining Lemma~\ref{l:A-psi-bound} and Lemma~\ref{l:kappa-eps-estimate}, we get
$$
\|A_{\kappa_\varepsilon}\|_{H^{-5/2}\to H^{3/2}}\leq C\varepsilon^2,\quad
\|A_{\kappa_\varepsilon}\|_{H^{-5/2}\to H^{7/2}}\leq C.
$$
By interpolation in Sobolev spaces (taking $f\in H^{-5/2}(\mathbb S^2)$
and using that $\|v\|_{H^1(\mathbb S^2)}^2$ is bounded by $\langle (1-\Delta_{\mathbb S^2})v,v\rangle_{L^2(\mathbb S^2)}\leq C\|v\|_{L^2(\mathbb S^2)}\|v\|_{H^2(\mathbb S^2)}$ for $v:=(1-\Delta_{\mathbb S^2})^{3/4}A_{\kappa_\varepsilon}f$) we then have
$$
\|A_{\kappa_\varepsilon}\|_{H^{-5/2}\to H^{5/2}}\leq C\varepsilon.
$$
This gives~\eqref{e:approxer-5} and finishes the proof of Theorem~\ref{t:main-identity}.

\appendix

\section{Harmonic 1-forms of constant length}\label{sec:appC}

The purpose of this appendix is to give an elementary proof of the fact that there are no harmonic 1-forms of constant non-zero length on closed hyperbolic 3-manifolds:
\begin{prop}
\label{prop:harmoniconeforms}
Let $(\Sigma,g)$ be a compact hyperbolic 3-manifold (see~\S\ref{s:hyp-3-basics}).
Assume that $\omega\in C^\infty(\Sigma;T^*\Sigma)$ is a harmonic 1-form
such that its length $|\omega|_g$ is constant. Then $\omega=0$.
\end{prop}
\Remark Proposition~\ref{prop:harmoniconeforms} follows directly from the more general work of~\cite{zeghib-93}. The presentation in the appendix borrows from ideas in~\cite{harris-paternain-16}.

To prove Proposition~\ref{prop:harmoniconeforms} we argue by contradiction.
Assume that $\omega\neq 0$; dividing $\omega$ by its length we arrange that, where $\delta = -\star d \star$ is the formal adjoint of $d$ (here $\star$ is the Hodge star)
$$
d\omega=0,\quad
\delta\omega=0,\quad
|\omega|_g=1.
$$
Using the metric~$g$, define the dual vector field to~$\omega$,
$$
W\in C^\infty(\Sigma;T\Sigma),\quad
|W|_g=\omega(W)=1.
$$
\begin{lemm}
There exist one-dimensional smooth subbundles $E_\pm\subset T\Sigma$
such that $T\Sigma=\mathbb RW\oplus E_+\oplus E_-$.
\end{lemm}
\begin{proof}
1. The Levi-Civita covariant derivative $\nabla W$ is an endomorphism on the fibers of~$T\Sigma$.
This endomorphism is symmetric with respect to the metric~$g$; indeed we compute
for any two vector fields $Y,Z\in C^\infty(\Sigma;T\Sigma)$
\begin{equation}
  \label{e:h1-0}
\begin{aligned}
0=d\omega(Y,Z)&=Yg(W,Z)-Zg(W,Y)-g(W,[Y,Z])
\\&=g(\nabla_YW,Z)-g(\nabla_ZW,Y).
\end{aligned}
\end{equation}
Taking $Z:=W$ and using that $g(\nabla_YW,W)={1\over 2}Yg(W,W)=0$
we see that the vector field~$W$ is geodesible, that is
\begin{equation}
  \label{e:h1-1}
\nabla_WW=0.
\end{equation}
Since $\delta\omega=0$, the vector field~$W$ is also divergence free; that is,
\begin{equation}
  \label{e:h1-0-trace}
\tr(\nabla W)=0.
\end{equation}

\noindent 2. We next claim that
\begin{equation}
  \label{e:h1-0-more-trace}
\tr((\nabla W)^2)=2.
\end{equation}
To see this, take locally defined vector fields $Y_1,Y_2$
such that $W,Y_1,Y_2$ is a $g$-orthonormal frame
and $\nabla_W Y_j=0$. These can be obtained using parallel transport
along the flow lines of~$W$ (which are geodesics since $\nabla_WW=0$).
We compute
$$
\begin{aligned}
1&=g(\nabla_{W}\nabla_{Y_j}W-\nabla_{Y_j}\nabla_WW+\nabla_{\nabla_{Y_j}W}W-\nabla_{\nabla_WY_j}W,Y_j)
\\&=
Wg(\nabla_{Y_j}W,Y_j)-g(\nabla_{Y_j}W,\nabla_WY_j)+g(\nabla_{\nabla_{Y_j}W}W,Y_j)
-g(\nabla_{\nabla_WY_j}W,Y_j)
\\&=Wg(\nabla_{Y_j}W,Y_j)+g((\nabla W)^2Y_j,Y_j).
\end{aligned}
$$
Here in the first line we used that $\Sigma$ has sectional curvature~$-1$,
in the second line we used~\eqref{e:h1-1}, and in the last line we used
that $\nabla_W Y_j=0$. Summing over $j=1,2$ and using again~\eqref{e:h1-1} we get
$$
2=W\tr(\nabla W)+\tr((\nabla W)^2)
$$
and~\eqref{e:h1-0-more-trace} now follows from~\eqref{e:h1-0-trace}.

3. From~\eqref{e:h1-1}, \eqref{e:h1-0-trace}, and~\eqref{e:h1-0-more-trace} we see
that $\nabla W$ has eigenvalues $0,1,-1$. It remains
to let $E_\pm$ be the eigenspaces of $\nabla W$ with eigenvalues $\pm 1$.
\end{proof}
We are now ready to finish the proof of Proposition~\ref{prop:harmoniconeforms}.
We can approximate the 1-form~$\omega$ by a closed 1-form with rational periods
(integrals over closed curves on $\Sigma$);
indeed, for an appropriate choice of linear isomorphism $H^1(\Sigma;\mathbb C)\simeq \mathbb C^{b_1(\Sigma)}$
the forms with rational periods correspond to points in $\mathbb Q^{b_1(\Sigma)}$. 
In particular, we can find a number $q\in\mathbb N$ and
a closed 1-form
$\widetilde\omega$ with integer periods such that
\begin{equation}
  \label{e:approx-form}
\sup_\Sigma|\omega-q^{-1}\widetilde\omega|_g\leq\textstyle{1\over 2}.
\end{equation}
Since $\widetilde\omega$ has integer periods,
we can write $\widetilde\omega=df$ for some smooth map $f$ from $\Sigma$
to the circle $\mathbb S^1=\mathbb R/\mathbb Z$. Since $\omega(W)=1$,
\eqref{e:approx-form} implies that $Wf=\widetilde\omega(W)>0$ which in turn
gives $df\neq 0$ everywhere, that is $f$ is a fibration.
Next, for each $x\in\Sigma$ define the one-dimensional spaces
$$
\widetilde E_\pm(x):=(\mathbb RW(x)\oplus E_\pm(x))\cap \ker df(x),
$$
then the tangent bundle of each fiber $f^{-1}(c)$ decomposes
into a direct sum $\widetilde E_+\oplus\widetilde E_-$.
Since $\Sigma$ is orientable, so is $f^{-1}(c)$, which implies
that $f^{-1}(c)$ is topologically a torus. Then $\Sigma$ is a torus bundle
over a circle, which gives a contradiction because such bundles do not
admit hyperbolic metrics: by the homotopy long exact sequence of a fibration
the fundamental group of $\Sigma$ contains a subgroup
isomorphic to $\mathbb Z\oplus\mathbb Z$, which is impossible for compact negatively
curved manifolds by Preissman's Theorem~\cite[Theorem~12.19]{Lee-Book}.

\medskip\noindent\textbf{Acknowledgements.}
We would like to thank Bernd Sturmfels and Nathan Ilten-Gee for useful discussions related to Lemma \ref{lem:matrixsubspace},
and Tomasz Mrowka for discussions related to Proposition~\ref{prop:harmoniconeforms}.
We are also grateful to the anonymous referees for many suggestions on improving the article. 
MC and BD have received funding from the European Research Council (ERC) under the European Union's Horizon 2020 research and innovation programme (grant agreement No. 725967). MC is further supported by an Ambizione grant (project number 201806) from the Swiss National Science foundation. BD has received further funding from the Deutsche Forschungsgemeinschaft (German Research Foundation, DFG) through the Priority Programme (SPP) 2026 ``Geometry at Infinity''. SD was supported by the National Science Foundation (NSF) CAREER grant DMS-1749858 and a Sloan Research Fellowship. GPP was supported by the Leverhulme trust and EPSRC grant EP/R001898/1. Part of this project was carried out while the four authors were participating in the Mathematical Sciences Research Institute Program in Microlocal
Analysis in Fall 2019, supported by the NSF grant DMS-1440140.


\bibliographystyle{alpha}
\bibliography{Biblio}

\begin{thebibliography}{{Med}21}

\bibitem[AB18]{Adam-Baladi}
Alexander Adam and Viviane Baladi.
\newblock Horocycle averages on closed manifolds and transfer operators, 2018.
\newblock \arXiv{1809.04062}.

\bibitem[AZ07]{Anantharaman-Zelditch-07}
Nalini Anantharaman and Steve Zelditch.
\newblock Patterson-{S}ullivan distributions and quantum ergodicity.
\newblock {\em Ann. Henri Poincar\'{e}}, 8(2):361--426, 2007.

\bibitem[Bal05]{baladi-05}
Viviane Baladi.
\newblock Anisotropic {S}obolev spaces and dynamical transfer operators:
  {$C^\infty$} foliations.
\newblock In {\em Algebraic and topological dynamics}, volume 385 of {\em
  Contemp. Math.}, pages 123--135. Amer. Math. Soc., Providence, RI, 2005.

\bibitem[Bis11]{Bismut}
Jean-Michel Bismut.
\newblock {\em Hypoelliptic {L}aplacian and orbital integrals}, volume 177 of
  {\em Annals of Mathematics Studies}.
\newblock Princeton University Press, Princeton, NJ, 2011.

\bibitem[BKL02]{BKL_02}
Michael Blank, Gerhard Keller, and Carlangelo Liverani.
\newblock Ruelle--{P}erron--{F}robenius spectrum for {A}nosov maps.
\newblock {\em Nonlinearity}, 15(6):1905--1973, 2002.

\bibitem[BL07]{butterley-car-07}
Oliver Butterley and Carlangelo Liverani.
\newblock Smooth {A}nosov flows: correlation spectra and stability.
\newblock {\em J. Mod. Dyn.}, 1(2):301--322, 2007.

\bibitem[Bon20]{guedes-bonthonneau-20}
Yannick~Guedes Bonthonneau.
\newblock Perturbation of {R}uelle resonances and {F}aure--{S}j\"{o}strand
  anisotropic space.
\newblock {\em Rev. Un. Mat. Argentina}, 61(1):63--72, 2020.

\bibitem[BT82]{bott-tu-82}
Raoul Bott and Loring~W. Tu.
\newblock {\em Differential forms in algebraic topology}, volume~82 of {\em
  Graduate Texts in Mathematics}.
\newblock Springer-Verlag, New York-Berlin, 1982.

\bibitem[BT07]{baladi-tsuji-07}
Viviane Baladi and Masato Tsujii.
\newblock Anisotropic {H}\"{o}lder and {S}obolev spaces for hyperbolic
  diffeomorphisms.
\newblock {\em Ann. Inst. Fourier (Grenoble)}, 57(1):127--154, 2007.

\bibitem[BWS21]{Borns-Weil-Shen-21}
Yonah Borns-Weil and Shu Shen.
\newblock Dynamical zeta functions in the nonorientable case.
\newblock {\em Nonlinearity}, 34(10):7322--7334, 2021.

\bibitem[CD20]{chaubet-dang-19}
Yann Chaubet and Nguyen~Viet Dang.
\newblock Dynamical torsion for contact {A}nosov flows, 2020.
\newblock Preprint, \arXiv{1911.09931}.

\bibitem[CP20]{cekic-paternain-19}
Mihajlo Ceki\'{c} and Gabriel~P. Paternain.
\newblock Resonant spaces for volume-preserving {A}nosov flows.
\newblock {\em Pure Appl. Anal.}, 2(4):795--840, 2020.

\bibitem[dC92]{do-Carmo-92}
Manfredo Perdig\~{a}o do~Carmo.
\newblock {\em Riemannian geometry}.
\newblock Mathematics: Theory \& Applications. Birkh\"{a}user Boston, Inc.,
  Boston, MA, 1992.
\newblock Translated from the second Portuguese edition by Francis Flaherty.

\bibitem[DFG15]{dyatlov-faure-guillarmou-15}
Semyon Dyatlov, Fr\'{e}d\'{e}ric Faure, and Colin Guillarmou.
\newblock Power spectrum of the geodesic flow on hyperbolic manifolds.
\newblock {\em Anal. PDE}, 8(4):923--1000, 2015.

\bibitem[DG16]{dyatlov-guillarmou}
Semyon Dyatlov and Colin Guillarmou.
\newblock Pollicott--{R}uelle resonances for open systems.
\newblock {\em Ann. Henri Poincar\'{e}}, 17(11):3089--3146, 2016.

\bibitem[DGRS20]{dang-guillarmou-riviere-shen-20}
Nguyen~Viet Dang, Colin Guillarmou, Gabriel Rivi\`ere, and Shu Shen.
\newblock The {F}ried conjecture in small dimensions.
\newblock {\em Invent. Math.}, 220(2):525--579, 2020.

\bibitem[DR19]{Dang-Riviere-19}
Nguyen~Viet Dang and Gabriel Rivi\`ere.
\newblock Spectral analysis of {M}orse-{S}male gradient flows.
\newblock {\em Ann. Sci. \'{E}c. Norm. Sup\'{e}r. (4)}, 52(6):1403--1458, 2019.

\bibitem[DR20]{Dang-Riviere-Topology}
Nguyen~Viet Dang and Gabriel Rivi\`ere.
\newblock Topology of {P}ollicott--{R}uelle resonant states.
\newblock {\em Ann. Sc. Norm. Super. Pisa Cl. Sci. (5)}, 21:827--871, 2020.

\bibitem[Dya21]{Dyatlov-Nir}
Semyon Dyatlov.
\newblock Pollicott--{R}uelle resolvent and {S}obolev regularity, 2021.
\newblock To appear in Pure Appl. Funct. Anal.; \arXiv{2108.06611}.

\bibitem[DZ16]{dyatlov-zworski-16}
Semyon Dyatlov and Maciej Zworski.
\newblock Dynamical zeta functions for {A}nosov flows via microlocal analysis.
\newblock {\em Ann. Sci. \'{E}c. Norm. Sup\'{e}r. (4)}, 49(3):543--577, 2016.

\bibitem[DZ17]{dyatlov-zworski-17}
Semyon Dyatlov and Maciej Zworski.
\newblock Ruelle zeta function at zero for surfaces.
\newblock {\em Invent. Math.}, 210(1):211--229, 2017.

\bibitem[DZ19]{dyatlov-zworski-book}
Semyon Dyatlov and Maciej Zworski.
\newblock {\em Mathematical theory of scattering resonances}, volume 200 of
  {\em Graduate Studies in Mathematics}.
\newblock American Mathematical Society, Providence, RI, 2019.

\bibitem[FH19]{Fisher-Hasselblatt-19}
Todd {Fisher} and Boris {Hasselblatt}.
\newblock {\em {Hyperbolic flows}}.
\newblock Berlin: European Mathematical Society (EMS), 2019.

\bibitem[FM12]{Farb-Margalit-12}
Benson Farb and Dan Margalit.
\newblock {\em A primer on mapping class groups}, volume~49 of {\em Princeton
  Mathematical Series}.
\newblock Princeton University Press, Princeton, NJ, 2012.

\bibitem[Fri86a]{Fried86}
David Fried.
\newblock Analytic torsion and closed geodesics on hyperbolic manifolds.
\newblock {\em Invent. Math.}, 84(3):523--540, 1986.

\bibitem[Fri86b]{Fried0}
David Fried.
\newblock Fuchsian groups and {R}eidemeister torsion.
\newblock In {\em The {S}elberg trace formula and related topics ({B}runswick,
  {M}aine, 1984)}, volume~53 of {\em Contemp. Math.}, pages 141--163. Amer.
  Math. Soc., Providence, RI, 1986.

\bibitem[Fri87]{Fried87}
David Fried.
\newblock Lefschetz formulas for flows.
\newblock In {\em The {L}efschetz centennial conference, {P}art {III} ({M}exico
  {C}ity, 1984)}, volume~58 of {\em Contemp. Math.}, pages 19--69. Amer. Math.
  Soc., Providence, RI, 1987.

\bibitem[FRS08]{faure-roy-sjostrand}
Fr\'{e}d\'{e}ric Faure, Nicolas Roy, and Johannes Sj\"{o}strand.
\newblock Semi-classical approach for {A}nosov diffeomorphisms and {R}uelle
  resonances.
\newblock {\em Open Math. J.}, 1:35--81, 2008.

\bibitem[FS11]{faure-sjostrand-11}
Fr\'{e}d\'{e}ric Faure and Johannes Sj\"{o}strand.
\newblock Upper bound on the density of {R}uelle resonances for {A}nosov flows.
\newblock {\em Comm. Math. Phys.}, 308(2):325--364, 2011.

\bibitem[GdP21]{Guillarmou-Poyferre}
Colin Guillarmou and Thibault de~Poyferr{\'{e}}.
\newblock A paradifferential approach for hyperbolic dynamical systems and
  applications, 2021.
\newblock with an appendix by Yannick Guedes Bonthonneau; \arXiv{2103.15397}.

\bibitem[Gei08]{Geiges-08}
Hansj\"{o}rg Geiges.
\newblock {\em An introduction to contact topology}, volume 109 of {\em
  Cambridge Studies in Advanced Mathematics}.
\newblock Cambridge University Press, Cambridge, 2008.

\bibitem[GHW21]{guillarmou-hilgert-weich-18}
Colin Guillarmou, Joachim Hilgert, and Tobias Weich.
\newblock High frequency limits for invariant {R}uelle densities.
\newblock {\em Ann. H. Lebesgue}, 4:81--119, 2021.

\bibitem[GL06]{gou-car-06}
S\'{e}bastien Gou\"{e}zel and Carlangelo Liverani.
\newblock Banach spaces adapted to {A}nosov systems.
\newblock {\em Ergodic Theory Dynam. Systems}, 26(1):189--217, 2006.

\bibitem[GLP13]{giuletti-liverani-pollicott-13}
Paolo Giulietti, Carlangelo Liverani, and Mark Pollicott.
\newblock Anosov flows and dynamical zeta functions.
\newblock {\em Ann. of Math. (2)}, 178(2):687--773, 2013.

\bibitem[Had18]{Hadfield}
Charles Hadfield.
\newblock Zeta function at zero for surfaces with boundary, 2018.
\newblock To appear in J. Eur. Math. Soc.; \arXiv{1803.10982}.

\bibitem[Ham95]{hamenstaedt-95}
Ursula Hamenst\"{a}dt.
\newblock Invariant two-forms for geodesic flows.
\newblock {\em Math. Ann.}, 301(4):677--698, 1995.

\bibitem[HHS12]{Hansen-Hilgert-Schroder}
S\"{o}nke Hansen, Joachim Hilgert, and Michael Schr\"{o}der.
\newblock Patterson-{S}ullivan distributions in higher rank.
\newblock {\em Math. Z.}, 272(1-2):607--643, 2012.

\bibitem[H{\"o}r03]{hoermander-03}
Lars H{\"o}rmander.
\newblock {\em The analysis of linear partial differential operators. {I}
  Distribution theory and Fourier analysis}.
\newblock Classics in Mathematics. Springer-Verlag, Berlin, 2003.

\bibitem[HP16]{harris-paternain-16}
Adam Harris and Gabriel~P. Paternain.
\newblock Conformal great circle flows on the 3-sphere.
\newblock {\em Proc. Amer. Math. Soc.}, 144(4):1725--1734, 2016.

\bibitem[Kan93]{kanai-93}
Masahiko Kanai.
\newblock Differential-geometric studies on dynamics of geodesic and frame
  flows.
\newblock {\em Japan. J. Math. (N.S.)}, 19(1):1--30, 1993.

\bibitem[KH95]{Katok-Hasselblatt}
Anatole Katok and Boris Hasselblatt.
\newblock {\em Introduction to the modern theory of dynamical systems},
  volume~54 of {\em Encyclopedia of Mathematics and its Applications}.
\newblock Cambridge University Press, Cambridge, 1995.
\newblock With a supplementary chapter by Katok and Leonardo Mendoza.

\bibitem[Kli95]{klingenberg-95}
Wilhelm P.~A. Klingenberg.
\newblock {\em Riemannian geometry}, volume~1 of {\em De Gruyter Studies in
  Mathematics}.
\newblock Walter de Gruyter \& Co., Berlin, second edition, 1995.

\bibitem[KW19]{kuster-weich2}
Benjamin K{\"u}ster and Tobias Weich.
\newblock {Quantum-Classical Correspondence on Associated Vector Bundles Over
  Locally Symmetric Spaces}.
\newblock {\em International Mathematics Research Notices}, 04 2019.
\newblock rnz068.

\bibitem[KW20]{kuster-weich}
Benjamin K{\"u}ster and Tobias Weich.
\newblock {Pollicott-Ruelle Resonant States and Betti Numbers}.
\newblock {\em Communications in Mathematical Physics}, 378(2):917--941, 2020.

\bibitem[Lee18]{Lee-Book}
John~M. Lee.
\newblock {\em Introduction to {R}iemannian manifolds}, volume 176 of {\em
  Graduate Texts in Mathematics}.
\newblock Springer, Cham, 2018.

\bibitem[Liv04]{Liverani1}
Carlangelo Liverani.
\newblock On contact {A}nosov flows.
\newblock {\em Ann. of Math. (2)}, 159(3):1275--1312, 2004.

\bibitem[Liv05]{Liverani2}
Carlangelo Liverani.
\newblock Fredholm determinants, {A}nosov maps and {R}uelle resonances.
\newblock {\em Discrete Contin. Dyn. Syst.}, 13(5):1203--1215, 2005.

\bibitem[{Med}21]{Meddane-21}
Antoine {Meddane}.
\newblock {A Morse complex for Axiom A flows}, July 2021.
\newblock preprint; \arXiv{2107.08875}.

\bibitem[MS91]{Moscovici-Stanton}
Henri Moscovici and Robert~J. Stanton.
\newblock {$R$}-torsion and zeta functions for locally symmetric manifolds.
\newblock {\em Invent. Math.}, 105(1):185--216, 1991.

\bibitem[Nic89]{Nicholls}
Peter~J. Nicholls.
\newblock {\em The ergodic theory of discrete groups}, volume 143 of {\em
  London Mathematical Society Lecture Note Series}.
\newblock Cambridge University Press, Cambridge, 1989.

\bibitem[Olb95]{olbrichdiss}
M.~Olbrich.
\newblock {Die Poisson-Transformation f\"ur homogene Vektorb\"undel}.
\newblock Dissertation, Humboldt-Universit\"at zu Berlin, 1995.

\bibitem[Pat99]{paternain-99}
Gabriel~P. Paternain.
\newblock {\em Geodesic flows}, volume 180 of {\em Progress in Mathematics}.
\newblock Birkh\"{a}user Boston, Inc., Boston, MA, 1999.

\bibitem[She18]{Shen-Fried}
Shu Shen.
\newblock Analytic torsion, dynamical zeta functions, and the {F}ried
  conjecture.
\newblock {\em Anal. PDE}, 11(1):1--74, 2018.

\bibitem[Wei17]{weich-17}
Tobias Weich.
\newblock On the support of {P}ollicott-{R}uelle resonant states for {A}nosov
  flows.
\newblock {\em Ann. Henri Poincar\'{e}}, 18(1):37--52, 2017.

\bibitem[Zeg93]{zeghib-93}
Abdelghani Zeghib.
\newblock Sur les feuilletages g\'{e}od\'{e}siques continus des
  vari\'{e}t\'{e}s hyperboliques.
\newblock {\em Invent. Math.}, 114(1):193--206, 1993.

\end{thebibliography}

\end{document}